\documentclass[11pt]{article}
\usepackage{amssymb,amsmath,amsfonts,amsthm,mathtools,bm,enumerate,color}
\usepackage{mathrsfs}
\usepackage[toc,page,title,titletoc,header]{appendix}
\usepackage{graphicx}

\usepackage[colorlinks,citecolor=blue,urlcolor=blue]{hyperref}
\usepackage{paralist}

\usepackage{comment} 

\usepackage{tikz}
\usetikzlibrary{arrows,positioning,shapes.geometric}

\usepackage{algorithm,algorithmic}

\graphicspath{ {./figures/} }

\usepackage[shortlabels]{enumitem}

\usepackage{cancel}

\usepackage{indentfirst}
\usepackage{multicol}
\usepackage{booktabs}
\usepackage{url}
\usepackage[outdir=./]{epstopdf} 

\usepackage{hyperref}
\hypersetup{
    colorlinks=true, 
    linktoc=all,     
    linkcolor=blue,  
}
\numberwithin{equation}{section}

\newtheorem{Theorem}{Theorem}[section]
\newtheorem{Lemma}[Theorem]{Lemma}
\newtheorem{Proposition}[Theorem]{Proposition}
\newtheorem{Assumption}{H.\!\!}

\theoremstyle{definition}
\newtheorem{Definition}{Definition}[section]

\newtheorem{Example}{Example}[section]

\theoremstyle{remark}
\newtheorem{Remark}{Remark}[section]

 \def\p{\partial} \def\nb{\nonumber}
\def\to{\rightarrow}
 \def\ol{\overline}    
\def\Om{\Omega}  \def\om{\omega} 
 
\newcommand{\q}{\quad}

\def\l{\label}  \def\f{\frac}  \def\fa{\forall}
\def\b{\beta}  \def\a{\alpha} 
 
\def\eps{\varepsilon}

 \def\t{\times}  
\def\ms{\medskip}

\def \la{\langle} \def\ra{\rangle}

\def\cB{\mathcal{B}}

\def\cF{\mathcal{F}}

\def\cH{\mathcal{H}}

\def\cS{\mathcal{S}}

\def\cV{\mathcal{V}}

\def\cY{\mathcal{Y}}

\def\d{{\mathrm{d}}}

\def\bA{{\textbf{A}}}

\def\sD{{\mathbb{D}}}
\def\sE{{\mathbb{E}}}
\def\sF{{\mathbb{F}}}

\def\sL{{\mathbb{L}}}
\def\sN{{\mathbb{N}}}
\def\sP{\mathbb{P}}

\def\sR{{\mathbb R}}

\newcommand{\tr}{\textnormal{tr}}
\newcommand{\prox}{\textnormal{prox}}

\DeclareMathOperator*{\argmax}{arg\,max}
\DeclareMathOperator*{\argmin}{arg\,min}

\newcommand{\lc}
{\mathrel{\raise2pt\hbox{${\mathop<\limits_{\raise1pt\hbox
{\mbox{$\sim$}}}}$}}}

\newcommand{\gc}
{\mathrel{\raise2pt\hbox{${\mathop>\limits_{\raise1pt\hbox{\mbox{$\sim$}}}}$}}}

\newcommand{\ec}
{\mathrel{\raise2pt\hbox{${\mathop=\limits_{\raise1pt\hbox{\mbox{$\sim$}}}}$}}}

\def\bb{\begin{equation}} \def\ee{\end{equation}}
\def\bbn{\begin{equation*}} \def\een{\end{equation*}}

\def\beqn{\begin{eqnarray}}  \def\eqn{\end{eqnarray}}

\def\beqnx{\begin{eqnarray*}} \def\eqnx{\end{eqnarray*}}

\def\bn{\begin{enumerate}} \def\en{\end{enumerate}}

\def\bd{\begin{description}} \def\ed{\end{description}}



\begin{document}

\title{Linear convergence of 
a policy gradient method for {
some}
finite horizon continuous time 
control problems
}

\author{
Christoph Reisinger\thanks{
Mathematical Institute, University of Oxford, Oxford OX2 6GG, UK
 ({\tt christoph.reisinger@maths.ox.ac.uk, 
wolfgang.stockinger@maths.ox.ac.uk})}
\and
Wolfgang Stockinger\footnotemark[1]
\and
Yufei Zhang
\thanks{Department of Statistics, London School of Economics and Political Science, Houghton Street, London, WC2A 2AE, UK
({\tt y.zhang389@lse.ac.uk})}
}
\date{}

\maketitle

\noindent\textbf{Abstract.} 
Despite its popularity in the reinforcement learning community, a provably convergent policy gradient method for 
continuous space-time 
control problems
with nonlinear state dynamics 
has been
elusive.
This paper 
 proposes proximal gradient algorithms for feedback controls of finite-time horizon stochastic control problems. The state dynamics are  
 nonlinear  diffusions with
 {
 control-affine} drift,
and the  cost functions are nonconvex in the state and nonsmooth in the control. 
{
The system noise can degenerate, which allows for   deterministic control problems as special cases.} 
We prove under suitable conditions that the algorithm converges linearly to a stationary point of the control problem, and is stable with respect to policy updates by approximate gradient steps. The  convergence result justifies the recent reinforcement learning heuristics that adding entropy regularization or 
{
a fictitious discount factor} to  the optimization objective accelerates the convergence of policy gradient methods.
 The proof  exploits careful regularity estimates of backward stochastic differential equations.


\ms


\medskip


\section{Introduction}\l{sec:intro}

Stochastic control problems 
seek optimal strategies to control continuous time stochastic  systems and 
are ubiquitous in  modern science, engineering
and economics
\cite{kushner2001numerical, pham2009continuous}.
In most applications,
the agent aims to construct a  feedback control mapping states of the system to  optimal actions.
A feedback control 
  has the advantage that it allows for implementing an optimal control
  in real time
 through evaluating the   feedback map  at  observed system states.
 %
An effective approach to 
generate (nearly) optimal 
feedback controls for high-dimensional  control problems
is via
gradient-based 
algorithms 
(see e.g., \cite{munos2006policy,
han2016deep,
sutton2018reinforcement,
jia2021policy}).
These algorithms, often referred to as \textit{policy gradient methods} (PGMs) in 
the reinforcement learning community,
approximate a  policy (i.e., a  feedback control)   in  a parametric  form,
and  update the policy parameterization iteratively based on gradients of the control objective.

Despite the notable success of PGMs, a mathematical theory that guarantees the convergence of these algorithms
for general 
(continuous time) stochastic control problems has been
elusive. 
{
It is known that 
 the objective of a control problem  is typically nonconvex with respect to \emph{feedback controls},
even if 
all cost functions are convex in   
state and control variables;
see \cite[Proposition 2.4]{giegrich2022convergence}
for a concrete example with  deterministic linear state dynamics
and strongly convex quadratic   costs. 
}%
This lack of convexity creates an essential challenge
in analysing the convergence behavior of PGMs.
Most  existing theoretical results of PGMs,
 especially those establishing (optimal) linear convergence,
focus on discrete  time  problems 
and restrict  policies 
within specific parametric families.
This includes 
 Markov decision problems (MDPs) with
softmax parameterized policies \cite{mei2020global}
or 
overparametrized one-hidden-layer neural-network policies \cite{wang2019neural,gu2021mean, kerimkulov2022convergence},
and discrete  time linear-quadratic (LQ) control problems with linear parameterized policies \cite{fazel2018global, hambly2021policy}.
The analysis therein exploits heavily   the specific structure of the considered
(discrete  time) control problems and policy parameterization,
and hence is difficult to extend to general continuous  time control problems or  general policy parameterizations. 
This leads to the following natural question:

\ms
\emph{Can one design  provably convergent gradient-based algorithms
for feedback controls of  continuous  time 
{
 nonlinear control problems},
without requiring specific policy parameterization?
} 
\ms



Analyzing PGMs in the continuous space-time  setting 
avoids 
discretization
artifacts and 
yields algorithms whose convergence behavior is robust with respect to time and space mesh sizes  \cite{tallec2019making}. Similarly, analyzing gradient-descent algorithms without specific policy parametrization 
avoids searching for controls in a suboptimal class.
This approach also
highlights the essential structures of the control problem that affect the algorithmic performance, which  subsequently provides a basis for  developing improved algorithms with more effective policy parameterizations (see Remark
\ref{rml:convergence_condition}).

 {
This work 
takes an initial step  towards answering the  above challenging   question,
and designs a convergent PGM  
for certain     control problems
with uncontrolled diffusion coefficients {
and affine control of the drift}.} 
Let 
 $T\in (0,\infty)$ be a given terminal time,
$(\Om,\cF,\sP)$ be a complete  probability space
on which a  $d$-dimensional Brownian motion $W=(W_t)_{t\in [0,T]}$
is defined,
 $\sF $ be the natural filtration of $W$
 augmented with an independent $\sigma$-algebra $\cF_0$,
and 
 $\mathcal{H}^2(\sR^k)$ be the set of $\sR^k$-valued 
 square integrable 
$\sF$-progressively measurable processes
$\a=(\a_t)_{ t\in [0,T]}$.
For any  initial 
state $\xi_0\in L^2(\cF_0;\sR^n)$
and any  $\a \in \mathcal{H}^2(\sR^k)$,
  consider 
the following controlled dynamics:
\bb\label{eq:control_fwd}
\d X_t=
b_t(X_t,\a_t)\, \d t
+\sigma_t(X_t)\, \d W_t,
\;  t \in [0,T],
\quad 
X_0=\xi_0,
\ee
where $b:[0,T] \times \sR^n \t \sR^k \to \sR^n$ and $\sigma:[0,T] \times \sR^n  \to \sR^{n \t d}$ are  differentiable functions
such that 
 \eqref{eq:control_fwd} admits a unique 
strong solution $X^{\xi_0,\a}$.
The agent's objective   is to minimize the following cost functional
\bb\label{eq:control_value}
J(\a;\xi_0)=\sE\bigg[
\int_0^T {
e^{-\rho t}} \left( f_t(X^{\xi_0,\a}_t,\a_t) + \ell(\a_t) \right) \, \d t+  {
e^{-\rho T}}  g(X^{\xi_0,\a}_T)
\bigg]
\ee 
over all admissible 
controls
$\a\in \cH^2(\sR^k)$,
where {
$\rho \geq 0$ is a given discount factor}\footnotemark,
\footnotetext{We specify the explicit dependence 
on $\rho$, as it impacts the convergence rate of PGMs.}
 $f:[0,T]\t \sR^n\t \sR^k\to \sR$ and $g:\sR^n\to \sR$ are differentiable  functions,
 and 
 $\ell:\sR^k \to \sR\cup\{\infty\}$ is a
 (possibly nondifferentiable) convex function.

{
 The precise conditions
on the  coefficients in 
\eqref{eq:control_fwd}-\eqref{eq:control_value}
will be given 
in Section \ref{sec:main1}.
 In particular, 
 we require the drift coefficient to be    affine in the control,
but allow both drift and diffusion coefficients to be nonlinear in the state.
The diffusion coefficient
can    degenerate, and hence 
  \eqref{eq:control_fwd} 
 includes   as a special case 
the  deterministic 
\emph{control-affine} system in nonlinear control theory \cite{isidori1985nonlinear}.  
We allow the cost functions $f$ and $g$ to be   nonconvex in the state,
but require  the running cost $f+\ell$ to be
   strongly convex in the   control.
The function 
 $\ell$ can be  discontinuous 
and can take the value infinity,
which are important characteristics
of control problems with 
control constraints
and  entropy regularizations;
see Examples 
\ref{example:constraints}, 
\ref{example:sparse} and 
\ref{example:entropy} for details.
Note that    these structural conditions in general do not imply convexity of  the control objective  $J$
in either the open-loop or feedback controls.

%

}

\paragraph{Proximal PGMs for  the control problem 
\eqref{eq:control_fwd}-\eqref{eq:control_value}.}

By interpreting 
\eqref{eq:control_fwd}-\eqref{eq:control_value}
as an  optimization problem over  $\cH^2(\sR^k)$, one can  
design a 
 gradient-descent  algorithm for \emph{open-loop} controls of the  problem.
Let 
$H^{\textrm{re}}:[0,T] \t \sR^n \t \sR^k \t \sR^n \to \sR$ be 
defined by 
\begin{align}
\label{eq:Hamiltonian_re}
H_t^{\textrm{re}}(x,a,y)\coloneqq \left \langle b_t(x,a),y \right \rangle +  f_t(x,a) - {
\rho \left \langle x,y \right \rangle }, 
\quad \forall (t,x,a,y)\in [0,T]\t \sR^n\t \sR^k\t \sR^n,
\end{align}
and let 
 $H:[0,T] \t \sR^n \t \sR^k \t \sR^n \t \sR^{n \t d} \to \sR$ 
 be the  Hamiltonian
 defined by 
\begin{align}
\label{eq:Hamiltonian}
H_t(x,a,y,z)\coloneqq H_t^{\textrm{re}}(x,a,y) + \left \langle \sigma_t(x), z \right \rangle,
\quad \forall (t,x,a,y,z)\in [0,T]\t \sR^n\t \sR^k\t \sR^n\t \sR^{n\t d}.
\end{align} 
{
Note that $H^{\textrm{re}}$ and $H$ only involve 
the differentiable component of the running cost, 
while the nonsmooth component $\ell$
  will be handled separately  by 
the proximal map  defined in \eqref{eq:proximal}.} 
Then for an 
initial guess $\alpha^0 \in \mathcal{H}^2(\mathbb{R}^k)$ and a stepsize $\tau >0$,   consider the  sequence $(\alpha^m)_{m \in \mathbb{N}}\subset \cH^2(\sR^k)$ such that for all  $m\in \sN_0\coloneqq \sN\cup\{0\}$,
\begin{equation}
\label{eq:open_loop_pgm}
    \alpha^{m+1}_t  = \prox_{\tau\ell} \big( \alpha^m_t - \tau {
    \partial_{a}} H_t^{\textrm{re}}(X^{\xi_0,\alpha^m}_t,\alpha^m_t,Y^{\xi_0,\alpha^m}_t)), \quad 
    \textnormal{for 
    $\mathrm{d}t \otimes \mathrm{d} \mathbb{P}$ a.e.,
    }
\end{equation}
where 
$(X^{\xi_0,\alpha^m},Y^{\xi_0,\alpha^m},Z^{\xi_0,\alpha^m})$ are adapted processes satisfying  the following forward-backward stochastic differential equation (FBSDE): for all $t \in [0,T]$,
\begin{alignat}{2}
\d X^{\xi_0,\alpha^m}_t
&= b_t(X^{\xi_0,\alpha^m}_t,\alpha^m_t)\, \d t +\sigma_t(X^{\xi_0,\alpha^m}_t)\, \d W_t, 
&&\q
 X_0^{\xi_0,\alpha^m} = \xi_0,
 \label{sde_open}
\\
\mathrm{d}Y^{\xi_0,\alpha^m}_t
&= - \partial_x H_t(X^{\xi_0,\alpha^m}_t,\alpha^m_t,Y^{\xi_0,\alpha^m}_t,Z^{\xi_0,\alpha^m}_t) \, \mathrm{d}t + Z^{\xi_0,\alpha^m}_t \, \mathrm{d}W_t,
&&
\q Y^{\xi_0,\alpha^m}_T = \partial_x g(X^{\xi_0,\alpha^m}_T),
\label{bsde_open}
\end{alignat}
and
$\prox_{\tau\ell}:\sR^k\to \sR^k$ is the  proximal map of $\tau\ell$ defined by
  \bb\label{eq:proximal}
  \textrm{prox}_{\tau\ell}(a)=\arg\min_{ {
  p} \in \sR^k}\left(\frac{1}{2}|{
  p}-a|^2+\tau\ell({
  p})\right),
  \quad \forall a\in \sR^k.
  \ee
  {
Note that
  \eqref{bsde_open} involves  undiscounted costs and
  the term $\rho Y^{\xi_0,\alpha^m}_t$,
  and arises from stochastic maximum principle for the discounted problem (see \cite{maslowski2014sufficient}).
  }
  
The iteration \eqref{eq:open_loop_pgm}
is  a proximal gradient method 
for \eqref{eq:control_value}.
The term
${
 \partial_{a}} H_t^{\textrm{re}}(X^{\xi_0,\alpha^m}_t,\alpha^m_t,Y^{\xi_0,\alpha^m}_t)$
is related to 
{
(up to an exponential time scaling)} the  Fr\'{e}chet derivative of the differentiable component of  $J(\cdot;\xi_0)$ at the  iterate $\a^m$,
while the function $\prox_{\tau \ell}$
can be identified as the proximal map of the nonsmooth component of  $J(\cdot;\xi_0)$
(see the proof of Theorem \ref{thm:convergence_stationary}).
We refer the reader to 
\cite{reisinger2021fast} for a detailed derivation of the algorithm and
to 
\cite{kerimkulov2021modified,vsivska2020gradient} for similar gradient-based algorithms without the nonsmooth term $\ell$.

The main drawback of the proximal gradient  algorithm  \eqref{eq:open_loop_pgm} 
(as well as the algorithms in 
\cite{kerimkulov2021modified,vsivska2020gradient})
is that it  iterates over open-loop controls. 
As
for each $m\in \sN$,
the iterate $\a^m\in \cH^2(\sR^k)$ is 
 a stochastic process  depending on the initial information and  
 the driving Brownian noise terms from previous iterates,
 the iteration \eqref{eq:open_loop_pgm} 
   is difficult to implement
 in practice. 
 In the sequel, we overcome the  shortcoming 
 of \eqref{eq:open_loop_pgm} 
 and introduce an analogue
 proximal gradient method
 for feedback controls of \eqref{eq:control_fwd}-\eqref{eq:control_value},
 which is referred to as 
   the \emph{proximal policy gradient method} (PPGM). 
 
 To this end, 
 we consider a class  $\cV_\bA$
 of Lipschitz continuous policies $\phi:[0,T]\t \sR^n\to \sR^k$, whose precise definition is given in  Definition \ref{def:fb}.
 For a given initial guess
  $\phi^0\in \cV_\bA$
  and a stepsize  $\tau>0$, 
the PPGM generates  the sequence  $(\phi^m)_{m\in \sN}\subset\cV_\bA$
  such that for all $m\in \sN_0$, 
\begin{align}
\label{eq:phi_update}
 \phi_t^{m+1}(x)&= \prox_{\tau\ell} \big(\phi_t^m(x)-\tau {
  \partial_{a}} H_t^{\textrm{re}}(x,\phi_t^m(x),Y^{t,x,\phi^m}_t) \big),
 \q 
 \fa (t,x)\in [0,T]\t \sR^n,
\end{align}
where
$\prox_{\tau \ell}$ is defined  in \eqref{eq:proximal},
and for each 
 $\phi\in \cV_\bA$ and 
 $(t,x)\in [0,T]\t \sR^n$,
$(X^{t,x,\phi},Y^{t,x,\phi},Z^{t,x,\phi})$
are adapted processes satisfying  the  FBSDE:
for all $s\in [t,T]$,
\begin{alignat}{2}
\d X^{t,x,\phi}_s
&= b_s(X^{t,x,\phi}_s,\phi_s(X_s^{t,x,\phi}))\, \d s +\sigma_s(X_s^{t,x,\phi})\, \d W_s, 
&&\q
 X_t^{t,x,\phi} = x,
 \label{sde_feedback}
\\
\mathrm{d}Y^{t,x,\phi}_s
&= - \partial_x H_s(X^{t,x,\phi}_s, \phi_s(X^{t,x,\phi}_s),Y^{t,x,\phi}_s,Z^{t,x,\phi}_s) \, \mathrm{d}s + Z^{t,x,\phi}_s \, \mathrm{d}W_s,
&&
\q Y^{t,x,\phi}_T = \partial_x g(X^{t,x,\phi}_T).
\label{bsde_feedback}
\end{alignat}

The  iteration \eqref{eq:phi_update} 
is motivated by the observation that 
if $\a^{\phi^m}_t=\phi^m_t(X^{\xi_0,\phi^m}_t)$ with $X^{\xi_0,\phi^m}$ being the state process controlled by the policy $\phi^m$,
then 
for $\mathrm{d}t \otimes \mathrm{d} \mathbb{P}$ a.e., 
$Y^{\xi_0,\a^m}_t=Y^{t,x,\phi^m}_t|_{x=X^{\xi_0,\phi^m}_t}$ and 
{
$$
 \partial_{a} H_t^{\textrm{re}}(x,\phi_t^{m}(x),Y_t^{t,x,\phi^{m}}) \big\vert_{x = X_t^{\xi_0,\phi^{m}}}
=e^{\rho t} \left(\nabla_\alpha J_{\rm diff}(\alpha; \xi_0)\big|_{\alpha=\alpha^{\phi^{m}}}\right)_t,
$$ 
where $J_{\rm diff}(\cdot; \xi_0)$ is the differentiable component of $J(\cdot; \xi_0)$ in \eqref{eq:control_value}:
for all $\a\in \cH^2(\sR^k)$,
\begin{align*}
\begin{split}
J_{\rm diff}(\alpha; \xi_0) \coloneqq 
\mathbb{E} \bigg[ \int_0^T e^{-\rho t}f_s(X^{\xi_0,\a}_s,\a_s)  \, \d s+  {
e^{-\rho T}}  g(X^{\xi_0,\a}_T) \bigg],
\end{split}
\end{align*}
and 
$\nabla_\alpha J_{\rm diff}(\alpha; \xi_0)\in \cH^2(\sR^k)$  is the     Fr\'{e}chet derivative of 
$J_{\rm diff}$ at $\alpha$.
}%
In other words, at the $m$-th iteration, \eqref{eq:phi_update} evaluates the functional derivative of $J_{\rm diff}(\cdot;\xi_0)$ at the  open-loop control
$\a^{\phi^m}$
 induced by the current policy $\phi^m$,
and obtains the update direction  based on a Markovian representation of the gradient. 
{
This choice of gradient directions is crucial for the well-posedness and convergence of the policy iterates 
$(\phi^{m})_{m\in \sN_0}$ in \eqref{eq:phi_update};
see the end of  Section \ref{sec:main2} for a   detailed comparison
between the  proposed gradient and the vanilla 
gradient direction of $J_{\rm diff}$ over feedback controls.
}

The PPGM 
\eqref{eq:phi_update}
improves the efficiency of 
the policy iteration (see
\cite{kerimkulov2020exponential,
ito2021neural})
and the Method of Successive Approximation
(see \cite{li2018maximum,kerimkulov2021modified}) by
avoiding 
a pointwise minimization of  the Hamiltonian  over the action space,
which may be  expensive, especially in a high-dimensional setting.
It has been successfully applied to  high-dimensional 
control problems
 in 
\cite{reisinger2021fast} 
by solving
the  linear BSDE 
\eqref{bsde_feedback} numerically;
see e.g.,   \cite{hure2020deep} and \cite{reisinger2021fast}  and references therein 
for various numerical schemes.

\paragraph{Our contributions.}

This paper  identifies  conditions under which  the PPGM \eqref{eq:phi_update} converges linearly to a stationary point of \eqref{eq:control_fwd}-\eqref{eq:control_value}. These conditions allow for nonlinear state dynamics with   degenerate noise,  unbounded action space
and 
    unbounded
    cost functions that are 
     {
     nonconvex in   state and 
     involve a nonsmooth regulariser in control}. 
     To the best  of  our knowledge, this is the  first work 
     which proposes a linearly convergent PGM  for a
     continuous time finite horizon control problem.
    The convergence result theoretically underpins experimental observations where  recent reinforcement learning heuristics, including entropy regularization or fictitious discount factor, 
    accelerate the convergence of PGMs.

We further prove  that  the PPGM \eqref{eq:phi_update} remains linearly convergent even if
the FBSDEs are solved only approximately and 
the policies are updated based on these approximate gradients. 
This stability result
 allows for computationally efficient algorithms as it shows that
 it is sufficient to solve the linear BSDEs  with low accuracy  at 
the initial iterations, while
an accurate BSDE solver is only required for the last few iterations; a similar strategy has been used  to design approximate policy iteration algorithms in  \cite{ito2021neural}.

\paragraph{Our approach and related works.}

There are various reasons for the relatively slow theoretical progress in PGMs for continuous time stochastic control problems. 
 Due to the nonconvexity of most objective functions of control problems  with respect to  the policies,
 establishing linear convergence of PGMs  can be linked to analyzing
 nonasymptotic 
performance of 
gradient search for nonconvex objectives,
which has always been one of the formidable challenges in  optimization theory.
Allowing   nonparametric policies in the algorithm
  further compounds the complexity,  
  as  the  analysis
  has to be carried out
  in a suitable function space, instead of in a  finite-dimensional parameter space.
  
Due to these technical difficulties, most existing works on 
linear performance guarantees of PGMs  concentrate on discrete time control problems with specific policy parametrization.
The arguments therein  often require specific problem structure,
in order to 
derive a suitable 
Polyak-{\L}ojasiewicz inequality
(also known as the gradient dominance property) for the loss landscape.
 For instance, in the tabular MDP  setting, the policies must be uniformly lower bounded away from zero over the entire state space \cite{mei2020global},
 while in the LQ setting,  eigenvalues of  state covariance matrices must be lower bounded away from zero over the entire time horizon \cite{fazel2018global,hambly2021policy}. 
 Consequently, these analyses 
  are difficult to extend to general control problems (such as  those with deterministic initial condition and degenerate noise)
 or to more sophisticated policy parameterizations (such as deep neural networks).

 Here, we 
 introduce a new analytical technique to analyse  the PPGM \eqref{eq:phi_update},
 without  relying on the Polyak-{\L}ojasiewicz condition
 or convexity.
 By carrying out a precise regularity estimate of  associated FBSDEs, 
  we establish uniform Lipschitz continuity 
  and uniform linear growth 
  of the iterates  $(\phi^m)_{m\in \sN}$.
  These estimates further 
  allow us to prove that 
  $(\phi^m)_{m\in \sN}$
  forms a contraction in a weighted sup-norm,
  whose limit can be identified as a stationary point of \eqref{eq:control_value}.
To the best of our knowledge,  this is
the first time BSDEs have been used to study convergence of PGMs.

\paragraph{Notation.}
For each 
  Euclidean space $(E,|\cdot|)$,
we introduce the following spaces: 
 \begin{itemize}[leftmargin=*,noitemsep,topsep=0pt]
\item
 $\cS^p(t,T;E)$, for $t \in [0,T]$ and $p\ge 2$,
  is the space of 
 $E$-valued
 $\sF$-progressively  measurable 
processes
$Y:  [t,T]\t \Om\to E$ 
satisfying $\|Y\|_{\cS^p}=\sE[\sup_{s\in [t,T]}|Y_s|^p]^{1/p}<\infty$;
\footnotemark
\footnotetext{With a slight abuse of notation, we denote by $\sup$ the essential supremum of a real-valued (Borel) measurable function.}
\item
 $\cH^p(t,T;E)$, for $t \in [0,T]$
 and $p\ge 2$, is the space of 
   $E$-valued $\sF$-progressively measurable
 processes 
$Z:  [t,T]\t \Om \to E$ 
 satisfying $\|Z\|_{\cH^p}=\sE[(\int_t^T|Z_s|^2\,\d s)^{p/2}]^{1/p}<\infty$.
\end{itemize}
For notational simplicity, 
we denote
$\cS^p(E)=\cS^p(0,T;E)$ and  $\cH^p(E)=\cH^p(0,T;E)$.

\section{{  
Main results}}\label{sec:main}


This section    summarises  the model assumptions
and presents the  main results on the 
 linear convergence of the PPGM \eqref{eq:phi_update}.

\subsection{Standing assumptions}\label{sec:main1}

The following assumptions  on the  coefficients of \eqref{eq:control_fwd}-\eqref{eq:control_value} are imposed 
throughout the paper.

\begin{Assumption}\label{assum:pgm}
Let $T>0$, $\xi_0\in L^2(\cF_0;\sR^n)$,
$\ell:\sR^k\to \sR\cup\{\infty\}$,  
$f:[0,T]\t \sR^n \t \sR^k \to\sR$,
 $g:\sR^n \to\sR$,
$b:[0,T]\t \sR^n \t \sR^k \to \sR^n$,
and $\sigma:[0,T]\t \sR^n \to\sR^{n\t d}$
be measurable functions  such that:
\begin{enumerate}[(1)]
\item \l{item:l}
$\ell$ is  lower semicontinuous   and its
effective domain 
$\bA\coloneqq \{z\in \sR^k
\!\mid\! \ell(z)<\infty\}$
is nonempty;
\footnotemark
\footnotetext{We say a function $f:X\to \sR\cup\{\infty\}$ is proper if it has a nonempty effective domain 
$\operatorname{dom}f\coloneqq \{x\in X\mid f(x)<\infty\}$.}

\item \l{item:f}
for all 
$t\in [0,T]$,
$\sR^n\t \sR^k\ni (x,a)\mapsto f_t(x,a)\in \sR$
is  continuously differentiable, $|f_t(0,0)|<\infty$, and 
 there exist constants
$C_{fx}, C_{fa}, L_{fx},L_{fa}\ge 0$ such that 
for all 
$t\in [0,T]$, $(x,a),(x',a')\in 
\sR^{n} \t \bA$,
\begin{align}
&
|\p_x f_t(x,a)|\le C_{fx},
\q 
|\p_x f_t(x,a)-\p_x f_t(x',a')|
\le
L_{fx}(|x-x'|+|a-a'|),
\l{eq:f_x_lipschitz}
\\
&
|\p_a f_t(0,0)|\le C_{fa},
\q 
|\p_a f_t(x,a)-\p_a f_t(x',a')|
\le
L_{fa}(|x-x'|+|a-a'|);
\l{eq:f_a_lipschitz}
\end{align}

\item \l{item:strong_convex}
 there exist constants
 $\mu,\nu\ge 0$ such that 
 $\mu+\nu>0$ and 
for all 
$(t,x)\in [0,T] \t \sR^n$, $a,a'\in 
 \bA
$
and $\eta\in [0,1]$,
\begin{align}
\eta f_t(x,a)+(1-\eta ) f_t(x,a')
&\ge 
f_t(x, \eta a+(1-\eta) a')
+ \eta(1-\eta)\tfrac{\mu}{2}|a-a'|^2,
\l{eq:strong_convex_f}
\\
\eta \ell(a)+(1-\eta ) \ell(a')
&\ge 
\ell( \eta a+(1-\eta) a')
+ \eta(1-\eta)\tfrac{\nu}{2}|a-a'|^2;
 \l{eq:strong_convex_l}
\end{align}
\item \l{item:g}
$g$ is  differentiable
and there exist constants $C_g, L_g\ge 0$
such that for all $x,x' \in \sR^n$,
\bb\label{eq:g_bound}
|\p_x g(x)|\le C_g, \q 
|\p_x g(x)-\p_x g(x')|\le L_g|x-x'|;
\ee

\item \l{item:drift}
there exist 
$\hat{b} :[0,T] \t \sR^n\to \sR^n$, $\bar{b} :[0,T] \t \sR^n\to \sR^{n\t k}$ such that
\bb\label{eq:affine}
b_t(x,a) = \hat{b}_t(x) + \bar{b}_t(x)a, 
\q \forall (t,x,a)\in [0,T]\t\sR^n\t \sR^k,
\ee
with
$\sR^n\ni x\mapsto ( \hat{b}_t(x),\bar{b}_t(x))\in \sR^n \t \sR^{n\t k}$ differentiable for all $t\in [0,T]$,
and 
 there exist constants $C_{\hat{b}}, C_{\bar{b}},L_{\hat{b}},L_{\bar{b}}\ge 0$ and
$\kappa_{\hat{b}}\in \sR$
such that
for all $t\in [0,T]$, $(x,a), (x',a') \in \sR^n \t \bA$,
\begin{align}
&
|\hat{b}_t(0)|
+|\p_x\hat{b}_t(0)|
\le C_{\hat{b}},
\q
|\bar{b}_t(x)|\le C_{\bar{b}},
\label{eq:b_bound}
\\
&
  \langle x-x', \hat{b}_t(x) - \hat{b}_t(x')  \rangle \leq \kappa_{\hat{b}} |x-x'|^2, 
 \q
|\p_x\hat{b}_t(x)-\p_x\hat{b}_t(x')|
\le L_{\hat{b}}|x-x'|,
\label{eq:hat_b_lipschitz}
\\
&
|\bar{b}_t(x)-\bar{b}_t(x')|
+
|\bar{b}_t(x)a-\bar{b}_t(x')a'|
+|\p_x\bar{b}_t(x)a-\p_x\bar{b}_t(x')a'|
\le L_{\bar{b}}(|x-x'|+|a-a'|);
\label{eq:bar_b_lipschitz}
\end{align}  
\item \l{item:diff}
there exist constants $C_{{\sigma}}, L_{{\sigma}}\ge 0$
such that
for all $t\in [0,T]$, $x,x'\in  \sR^n$,
\begin{equation}
\label{eq:sigma}
|\sigma_t(x)|\le C_\sigma,
\q 
|\sigma_t(x)-\sigma_t(x')|
+|\p_x\sigma_t(x)-\p_x\sigma_t(x')|
\le L_\sigma|x-x'|.
\end{equation}

\end{enumerate}
\end{Assumption}

\begin{Remark}\label{rmk:assum}

The action set $\bA$ may be unbounded,
and hence
 \eqref{eq:bar_b_lipschitz} cannot be further simplified. 
 If one assumes further that $\bA$ is   bounded, 
 then, 
  by the boundedness of $\bar{b}$ in \eqref{eq:b_bound}, 
  \eqref{eq:bar_b_lipschitz}
is equivalent to the Lipschitz continuity of $\bar{b}$ and $\p_x\bar{b}$. Alternatively, if  $\bA$  is unbounded, 
then 
 \eqref{eq:bar_b_lipschitz} is equivalent to the condition that  $\bar{b}$ is independent in $x$.

To consider nonlinear state-dependent drift and diffusion coefficients,  we impose
in
 \eqref{eq:f_x_lipschitz} and \eqref{eq:g_bound}
 the boundedness conditions  on the 
spatial partial derivatives of 
cost functions.
 Observe from 
 \eqref{eq:Hamiltonian} 
 and \eqref{eq:affine}
 that 
 $\p_x H$ 
 (resp.~$\p_a H$)
 involve the term $(\p_x \hat{b}_t(x)+\p_x \bar{b}_t(x)a)^\top y+\p_x \sigma(x)^\top z$
(resp.~$\bar{b}_t(x)^\top y$),
 whose modulus of continuity in  $x$
 depends on the magnitude of $y$ and $z$.
 By exploiting the boundedness of $\p_x f$ and $\p_x g$, we establish an a-priori bound of the adjoint processes, and subsequently 
prove the iterative scheme \eqref{eq:phi_update}
generates Lipschitz continuous policies 
 $(\phi^m)_{m\in \sN_0}$ (see Proposition \ref{prop:Y_stability}).
 If the drift and diffusion coefficients
 are affine in $x$,
 then   \eqref{eq:f_x_lipschitz} and \eqref{eq:g_bound}
 can be relaxed to 
 quadratically growing  functions,
 which include  as special cases the linear-convex control problems  
studied in 
 \cite{guo2021reinforcement,szpruch2021exploration}.

 For clarity, 
 \eqref{eq:strong_convex_f} and  \eqref{eq:strong_convex_l} assume  convexity of   $a\mapsto f_t(x,a)$ and $a\mapsto \ell(a)$
 and  strong convexity of  
 $a\mapsto f_t(x,a)+\ell (a)$.
 This allows for characterizing  the  rate of convergence 
 of $(\phi^m)_{m\in\sN_0}$ in terms of $\mu$ and $\nu$. 
 Similar analysis can be performed if
  \eqref{eq:strong_convex_f} is relaxed into the following semi-convexity condition, i.e., there exists $\mu\in [-L_{fa},L_{fa}]$ such that for all 
$(t,x)\in [0,T] \t \sR^n$, $a,a'\in 
 \bA
$
and $\eta\in [0,1]$,
 $$
 \eta f_t(x,a)+(1-\eta ) f_t(x,a')
\ge 
f_t(x, \eta a+(1-\eta) a')
+ \eta(1-\eta)\tfrac{\mu}{2}|a-a'|^2,
$$
and $\ell$ is $\nu$-strongly convex with a sufficiently large $\nu$
(cf.~Condition \ref{item:large_mu_nu} below).
Such an assumption 
allows $f$ to be concave in $a$ and 
can be satisfied 
if the objective function 
\eqref{eq:control_value} 
involves entropy regularization
(see Example
\ref{example:entropy}).

\end{Remark}

Here we   present several important   nonsmooth costs used    in engineering and machine learning.

\begin{Example}[Control constraint]
\label{example:constraints}
Let  
 $\bA\subset \sR^k$ be a nonempty closed convex set
 and  
$\ell:\sR^k\to [0,\infty]$ be the indicator of $\bA$ satisfying 
$\ell(a)=0$ for $a\in \bA$ and 
$\ell(a)=\infty$ for $a\in \sR^k\setminus\bA$.
Then  \eqref{eq:strong_convex_l} holds 
with $\nu=0$, and for all $\tau>0$, 
$\prox_{\tau \ell}$ is the orthogonal projection on $\bA$. In this case, \eqref{eq:phi_update} extends the projected PGM in \cite{hambly2021policy} to general stochastic control problems. 
\end{Example}

\begin{Example}[Sparse control]
\label{example:sparse}
Let  $(\gamma_i)_{i=1}^k\subset [0,\infty)$
and $\ell:\sR^k\to [0,\infty)$ be such that    
$\ell(a)=\sum_{i=1}^k \gamma_i |a_i|$,
for $a=(a_i)_{i=1}^k \in \sR^k$.
Then \eqref{eq:strong_convex_l} holds 
with $\nu=0$, and  for all $\tau>0$, 
$\prox_{\tau \ell}(a)=(\max\{|a_i|-\tau \gamma_i,0\}\operatorname{sgn}(a_i))_{i=1}^k$
for each $a=(a_i)_{i=1}^k \in \sR^k$. 
In this case, \eqref{eq:phi_update} can be viewed as an infinite-dimensional extension of 
the iterative shrinkage-thresholding algorithm (see \cite{beck2009fast, reisinger2021fast}).
\end{Example}

\begin{Example}[$\mathfrak{f}$-divergence regularized control] 
\label{example:entropy}

Let  $\Delta_k\coloneqq\{
a\in  [0,1]^k\mid 
\sum_{i=1}^k a_i=1\}$,
  $\mathfrak{u}=(\mathfrak{u}_i)_{i=1}^k\in \Delta_k\cap (0,1)^k$, 
 and 
 $\ell:\sR^k\to \sR\cup \{\infty\}$ be the  $\mathfrak{f}$-divergence defined by
 $$
 \ell(a)\coloneqq 
\sum_{i=1}^k \mathfrak{u}_i \mathfrak{f}\Big(\frac{a_i}
{\mathfrak{u}_i}
\Big),
\quad a\in \Delta_k;
\quad 
\ell(a)=\infty,
\quad 
a\not \in \Delta_k
 $$
with a given 
lower semicontinuous function
$\mathfrak{f}:[0,\infty)\to \sR\cup \{\infty\}$ 
satisfying
 $\mathfrak{f}(0)=\lim_{x\to 0} \mathfrak{f}(x)$, $\mathfrak{f}(1)=0$, and 
 being $\kappa_{\mathfrak{u}}$-strongly convex on $[0,\tfrac{1}{\min_i\mathfrak{u}_i}]$ with some  $\kappa_{\mathfrak{u}}>0$.
 As shown in 
 \cite[Example 2.2]{guo2021reinforcement},
 $\ell$ 
  satisfies (H.\ref{assum:pgm})
with
$\nu=\tfrac{\kappa_{\mathfrak{u}}}{\max_i\mathfrak{u}_i}>0$.

Note that  an  $\mathfrak{f}$-divergence
$\ell$ 
is 
typically  non-differentiable
and may have non-closed 
  effective domain $\bA$
(see 
\cite{guo2021reinforcement} for 
concrete examples).
For  commonly used forms of $\mathfrak{f}$-divergence,
the  proximal map  $\prox_{\ell}$ can   
be computed by solving 
\eqref{eq:proximal} with
 Lagrange multipliers.
For instance, 
let 
$\ell$ be the relative entropy 
corresponding to  $\mathfrak{f}(s)=s\log s$, $s\in \sR$.
Then for each $\tau>0$
and $a=(a_i)_{i=1}^k \in \sR^k$,
$\prox_{\tau \ell}(a)_i
=\tau W\left(\frac{\mathfrak{u}_i}{\tau}\exp\big(\frac{\lambda+a_i}{\tau}-1\big)\right)
$ for all $i=1,\ldots, k$,
where $W:[0,\infty)\to [0,\infty)$ is the Lambert W-function, and  $\lambda\in \sR$
is the unique solution to 
$\sum_{i=1}^k\tau W\left(\frac{\mathfrak{u}_i}{\tau}\exp\big(\frac{\lambda+a_i}{\tau}-1\big)\right)=1$.

\end{Example}

\subsection{{Well-posedness of the iterates}}\label{sec:main2}

In the sequel, we focus on Lipschitz continuous feedback controls such that the corresponding controlled state dynamics \eqref{eq:control_fwd} admits a strong solution.
Due to the (possible) unboundedness of  the action set $\bA$, these  controls  in general grow linearly with respect to the state variable.


\begin{Definition}\l{def:fb}
Let 
$\cB([0,T]\t \sR^n;\sR^k)$
be the space of  measurable functions $\phi:[0,T]\t\sR^n\to \sR^k$,
and let 
 $|\cdot|_{0},[\cdot]_{1}:\cB([0,T]\t \sR^n;\sR^k)\to [0,\infty]$
be 
such that for all $\phi\in\cB([0,T]\t \sR^n;\sR^k)$, 
$$
|\phi|_{0}=\sup_{(t,x)\in [0,T]\t \sR^n}\frac{|\phi_t(x)|}{1+|x|},
\quad \q 
[\phi]_{1}=\sup_{t\in [0,T], x,y\in  \sR^n,x\not = y}\frac{|\phi_t(x)-\phi_t(y)|}{|x-y|}.
$$
We define  the  
following space of 
 feedback controls:
 \begin{align}
 \l{eq:lipschitz_feedback}
\cV_\bA 
&\coloneqq 
\left\{ 
\phi \in \cB( [0,T]\t \sR^n; \sR^k)
\,\middle\vert\, 
|\phi|_0+[\phi]_1<\infty,
\textnormal{$\phi_t(x)\in \bA$ for a.e.~$(t,x)\in [0,T]\t \sR^n$}
\right\},
\end{align}
and for each
  $\phi\in \cV_\bA$,
 define
 the  associated control process
 $\alpha^\phi\in \cH^2(\sR^k)$
  by
 $\alpha^\phi_t=\phi_t(X^{\xi_0,\phi}_t)$
 $\d t\otimes \d \sP$-a.e.,
 where  $X^{\xi_0,\phi}\in \cS^2(\sR^n)$ 
 is the solution to the  following SDE (cf.~\eqref{eq:control_fwd}):
\bb\l{eq:control_fwd_fb}
\d X_t=
b_t(X_t,\phi_t(X_t))\, \d t
+\sigma_t(X_t)\, \d W_t,
\q t\in [0,T];
\q 
X_0=\xi_0.
\ee

\end{Definition}

{
The Lipschitz regularity of $\phi\in \cV_\bA$ ensures that  
  the system
\eqref{sde_feedback}-\eqref{bsde_feedback}
 admits a unique strong solution. 
We refer the reader to 
\cite{kerimkulov2020exponential, guo2021reinforcement}
for sufficient conditions under which 
the control problem admits an optimal feedback control in the class $\cV_\bA$.
However, we emphasise that 
 in this work we 
 do not require  the control problem 
\eqref{eq:control_fwd}-\eqref{eq:control_value} to have an optimal feedback control. 
 Instead, we   focus on constructing a policy $\phi\in\cV_\bA$ whose associated (open-loop)
control process is a stationary point of 
$J(\cdot;\xi_0)$;
see Section \ref{sec:main3} for details. 

}

Under (H.\ref{assum:pgm}), 
the iterative scheme \eqref{eq:phi_update}
is well-defined 
for any given guess 
$\phi^0\in \cV_\bA$ and stepsize $\tau>0$.
The proof of this
relies on the well-posedness and stability of the
FBSDEs \eqref{sde_feedback}-\eqref{bsde_feedback}, 
with extra difficulties arising from possibly
  non-Lipschitz and unbounded coefficients, i.e.,
  $ \hat{b}_t$ may be non-Lipschitz  in $x$, and $ \p_x H$ non-Lipschitz  in $(x,y)$ and unbounded in $x$. The detailed arguments can be found in Appendix 
 \ref{appendix:technical}.

\begin{Proposition}
\label{prop:well_posed_phi}
 Suppose (H.\ref{assum:pgm}) holds. Then for all 
 $\phi^0\in \cV_\bA$ and $\tau>0$,
 the iterates 
$(\phi^{m})_{m\in \sN_0}$ are well-defined functions in $\cV_\bA$.
\end{Proposition}

{
\paragraph{Regularity of the gradient direction in \eqref{eq:phi_update}.}
 Here we emphasise the   importance of the  gradient direction in \eqref{eq:phi_update}
on the well-posedness of the iterates $(\phi^{m})_{m\in \sN_0}$.
Observe that  if $\phi^m\in \cV_\bA$,
then 
classical stability results of    \eqref{sde_feedback}-\eqref{bsde_feedback}
imply that 
the map $x\mapsto Y^{t,x,\phi^m}_t$ in 
\eqref{eq:phi_update}
is Lipschitz continuous uniformly in $t$,
which subsequently ensures that $\phi^{m+1}\in \cV_\bA$.
Such a Lipschitz regularity holds even if the diffusion coefficient of \eqref{eq:control_fwd} degenerates, which includes   deterministic control problems as   special cases.
 
The above regularity estimate    in general does not hold if 
one updates   a feedback control using    the gradient of $J$  at the feedback map itself,
especially when the diffusion coefficient of  \eqref{eq:control_fwd} degenerates. 
To see this, we assume for simplicity that  all variables are one-dimensional, 
and consider minimising the following cost
(with 
$\xi_0=x_0$ and $\sigma=0$ in \eqref{eq:control_fwd}
and 
$\ell= 0$ and $\rho=0$  in \eqref{eq:control_value})
over all $\phi\in \cV_\bA$:
\bb
\label{eq:cost_feedback}
J(\phi;x_0)=
\int_0^T 
 f_t(X^{x_0,\phi}_t,\phi_t(X^{x_0,\phi}_t))   \, \d t+    g(X^{x_0,\phi}_T),
\ee
where for each $\phi\in \cV_\bA$,
 $X^{x_0, \phi}_t = x_0 +\int_0^t b_s(X^{x_0, \phi}_s,\phi_s(X^{x_0, \phi}_s))\, \d s $
for all $t\in [0,T]$.
By \cite[Section 4.1]{bensoussan2013mean},
 for any given   $\psi \in \cV_\bA$, the derivative of $J$ at $\phi \in \cV_\bA$ in the direction $\psi$ is  
\begin{equation}
\label{eq:derivative_feedback}
    \frac{\mathrm{d} J(\phi + \varepsilon \psi)}{\mathrm{d} \varepsilon} \Big |_{\varepsilon =0} = 
    \int_{0}^{T} \partial_{a} H^{\textrm{re}}_t(X^{x_0,\phi}_t,\phi_t(X^{x_0,\phi}_t),\partial_x \mathfrak{u}^\phi_t(X^{x_0,\phi}_t)) \psi_t(X^{x_0,\phi}_t) \, \mathrm{d}t, 
\end{equation}
where  $H^{\textrm{re}}$ is defined in 
\eqref{eq:Hamiltonian_re}, 
and $ \mathfrak{u}^\phi:[0,T] \times \mathbb{R} \to \mathbb{R}$ satisfies for all $(t,x)\in [0,T]\t \sR$, 
 \begin{align}\label{eq:linear_pde}
\partial_t   \mathfrak{u}_t(x)+
H^{\textrm re}_t\big(x,\phi_t(x), \p_{x}  \mathfrak{u}_t(x)\big)=0; \quad  \mathfrak{u}_T(x)=g(x).
\end{align}
Observe that \eqref{eq:derivative_feedback} is an $L^2$-inner product between 
$(t,x)\mapsto \partial_{a} H^{\textrm{re}}_t(x,\phi_t(x),\partial_x \mathfrak{u}^\phi_t(x)) $
and $\psi$ with respect to the law of $X^{x_0, \phi}$.
This leads to the following iterative scheme,  which is a direct application of the gradient descent algorithm  for feedback controls:
\bb\label{eq:gd_phi}
\phi^{m+1}_t(x)=\phi^{m}_t(x)-\tau 
\p_a H^{\textrm re}_t \big(x,\phi^m_t(x),
\partial_x \mathfrak{u}^{\phi^m}_t  (x)\big),
\quad (t,x)\in [0,T]\t \sR.
\ee
However,  the iteration \eqref{eq:gd_phi} 
in general does not preserve the   regularity of the   iterates $(\phi^m)_{m\in \mathbb{N}}$,
and hence may not be well-defined.
To see this, assume that $\phi^m\in \cV_\bA$  
for some $m\in \mathbb{N}$. Then by \eqref{eq:gd_phi},
the regularity of $\phi^{m+1}$ depends on the regularity of $\partial_x \mathfrak{u}^{\phi^m}$. Formally taking derivatives of \eqref{eq:linear_pde} with respect to $x$ implies that $w^{m} \coloneqq   \partial_x \mathfrak{u}^{\phi^m}$ 
satisfies the following PDE:
for all $(t,x)\in [0,T]\t \sR$,
\begin{align}\label{PDEsystem_gd}
\begin{split}
&\partial_t w_t(x)
+
b_t(x,\phi^{m}_t(x)) \partial_x w_t(x)
+\p _x H^{\textrm re}_t\big(x,\phi^m_t(x), w_t(x)\big)
 = 
-\p_a H^{\textrm re}_t\big(x,\phi^m_t(x), w_t(x) \big)
\p_x \phi^m_t(x),
\end{split}
\end{align}
with $w_T(x)=\p_x g(x)$. The  term $\p_a H^{\textrm{re}}\,\p_x \phi^m$ on the right-hand side of \eqref{PDEsystem_gd} appears due to  the application of the chain rule. 
The Lipschitz continuity of $\phi^m$ only implies the boundedness of 
$\,\p_x \phi^m$, and consequently
both    
$\partial_x \mathfrak{u}^{\phi^m}$
and   $\phi^{m+1}$   are in general  not Lipschitz continuous.
More crucially, it indicates  that  estimating the derivatives of 
$\phi^{m+1}$ requires bounds on higher order derivatives of $\phi^{m}$, and it is unclear how to close this norm gap.

In contrast, 
such a loss of regularity does not occur in 
  \eqref{eq:phi_update}. Indeed, in  the setting of
\eqref{eq:cost_feedback},  
$Z^{t,x,\phi}\equiv 0$ in \eqref{bsde_feedback}, 
and hence by the Feynman-Kac formula, \eqref{eq:phi_update} is   equivalently to  
\begin{align*}
 \phi_t^{m+1}(x)&=   \phi_t^m(x)-\tau {
  \partial_{a}} H_t^{\textrm{re}}(x,\phi_t^m(x),u^{m}_t(x)), 
 \q 
 \fa (t,x)\in [0,T]\t \sR,
\end{align*}
where $u^{m}$ is the unique  continuous viscosity solution to the following  PDE:
for all $(t,x)\in [0,T]\t \sR$,
\begin{equation}
\label{PDE_feymankac}
    \partial_t u_t(x) +b_t(x,\phi^{m}_t(x)) \partial_x u_t(x) +\partial_x H_t^{\textrm{re}}(x,\phi_t^{m}(x),u_t(x))=0, \qquad u_T(x) = \partial_x g(x).
\end{equation}
Note that \eqref{PDE_feymankac} does not involve  the term  $\partial_x H_t^{\textrm{re}}\p_x \phi^m$,
and   under (H.\ref{assum:pgm}), 
all coefficients of \eqref{PDE_feymankac} are 
sufficiently regular
such that 
$u^{m}$ is indeed Lipschitz continuous in $x$ (uniformly in $t$),
according to  classical Lipschitz estimates of viscosity solution
(see e.g., \cite{barles2002convergence}).

 }

\subsection{Linear convergence of the iterates}\label{sec:main3}

The main contribution of this article is to identify conditions under which
$(\phi^m)_{m\in \sN_0}\subset \cV_\bA$ converge linearly to 
a stationary point of
the control problem 
\eqref{eq:control_fwd}-\eqref{eq:control_value}.
As the functional $J(\cdot;\xi_0):\cH^2(\sR^k)\to \sR\cup\{\infty\}$ is typically nonsmooth and nonconvex, 
we first recall a notion of stationary points for nonsmooth nonconvex functionals on  Hilbert spaces, defined as in
 \cite{mordukhovich2006variational}. By \cite[Proposition 1.114]{mordukhovich2006variational}, every local minimizer $\alpha^\star\in \operatorname{dom} J(\cdot;\xi_0)$ is a stationary point in the sense of Definition \ref{def:stationary}.
In practice, a stationary point found in this way often gives a good solution candidate  \cite{li2018maximum}.
 

 \begin{Definition}
 \label{def:stationary}
 Let $X$ be a Hilbert space
 equipped with the norm $\|\cdot\|_X$ and the inner product $\la \cdot,\cdot\ra_X$,  $F:X\to \sR\cup \{\infty\}$, and ${x}^\star\in \operatorname{dom} F=\{x\in X\mid F(x)<\infty\}$.
 The Fr\'{e}chet subdifferential of $F$ at ${x}^\star$ is defined by
 $$
 \p F({x}^\star)=
 \bigg\{\bar{x}\in X\bigg\vert 
 \liminf_{x\to {x}^\star}
 \frac{F(x)-F({x}^\star)-\la \bar{x}, x-{x}^\star \ra_X}{\|x-{x}^\star\|_X}\ge 0\bigg\}.
 $$
 We say ${x}^\star\in \operatorname{dom} F$ is a stationary point of $F$ if $0\in \p F({x}^\star)$.
 \end{Definition}

As alluded to earlier,  the map 
 $\cV_\bA\ni \phi\mapsto J(\a^\phi;\xi_0)\in \sR\cup\{\infty\}$ is typically nonconvex and 
may not satisfy  
 the Polyak-{\L}ojasiewicz condition 
 as in the setting 
with parametric policies
(\cite{fazel2018global,
 wang2019neural,
 mei2020global,
 gu2021mean,
 hambly2021policy,
 kerimkulov2022convergence}).
 Hence to ensure the linear convergence of the PPGM \eqref{eq:phi_update},
 we impose further conditions
 on the coefficients which guarantee that we are in \emph{one of the following {six} cases}:
\begin{enumerate}[(i)]

    \item
    \label{item:T_small}
    Time horizon $T$ is small. 
    \item 
    \label{item:rho} {
    Discount factor $\rho$ is large.} 
    \item 
    \label{item:large_mu_nu}
    Running cost is sufficiently convex in control,  
    i.e., $\mu+\nu$ is sufficiently large.
     \item 
     \label{item:small_cost}
     Costs depend weakly on state, i.e., $C_{fx}, L_{fx},  C_g$ and $L_g$ are  small.
      \item 
    \label{item:small_Cb}
    Control affects  state dynamics  weakly, i.e., $C_{\bar{b}}$ is  small.
    \item 
    \label{item:small_k}
    State dynamics is strongly dissipative, i.e., $\kappa_{\hat{b}}$ is sufficiently negative.
   
\end{enumerate}
The above conditions will be made precise in  \eqref{eq:condition_lipschitz} 
 and 
\eqref{eq:constraction_condition}.
Here we give some 
practical implications 
of these conditions. 

\begin{Remark}
\label{rml:convergence_condition}
Conditions \ref{item:T_small} 
and \ref{item:rho}
are commonly used conditions to ensure  the convergence of iterative  algorithms  for nonconvex problems (see e.g., \cite{bender2008time,
bayraktar2018numerical,
hu2019deep,ito2021neural}). 
{
 Condition \ref{item:rho}   also justifies the use of a fictitious discount factor 
to accelerate the convergence of PGMs
 for continuous-time control problems
(see \cite{guo2021theoretical} and references therein).
}

Conditions \ref{item:large_mu_nu}-\ref{item:small_Cb}
help to 
ease the nonconvexity 
of $\phi\mapsto J(\a^\phi;\xi_0)$
and to reduce the oscillation of the loss function's curvature,
which subsequently
promotes  the convergence of gradient-based algorithms
(see  \cite{nesterov2003introductory}).
Condition \ref{item:large_mu_nu},
along with Example \ref{example:entropy},
also justifies recent reinforcement learning heuristics that adding   $\mathfrak{f}$-divergences, such as the relative entropy, to the optimization objective can  accelerate the convergence of PGMs
(see e.g., \cite{vsivska2020gradient, kerimkulov2022convergence}).

Condition \ref{item:small_k} indicates that a
strong dissipativity of the state dynamics enhances the efficiency of  learning algorithms.
Such a phenomenon has already been observed in the LQ setting
with  $\hat{b}_t(x)=A_t x$ in \eqref{eq:affine},
where the desired 
dissipativity   can be ensured
if  eigenvalues of $A_t$ are sufficiently negative
(see  \cite{
hu2019deep,hambly2021policy}).
Condition \ref{item:small_k}
also motivates a residual correction method 
for  solving nonlinear control problems.
Consider  a  control problem \eqref{eq:control_fwd}-\eqref{eq:control_value} whose drift involves non-dissipative coefficient $\hat{b}$. Then 
one can search feedback controls of the form
$\phi=\ol{\phi}+\widetilde{\phi}$,
where $\ol{\phi}$
is a precomputed candidate policy, and $\tilde{\phi}$  is an unknown  residual correction. Observe that 
the state dynamics now has the drift coefficient
 $b=(\hat{b}+\bar{b}\ol{\phi})+\bar{b}\widetilde{\phi}$, 
and the function $\hat{b}+\bar{b}\ol{\phi}$ may be dissipative for suitably chosen policy $\ol{\phi}$;
see \cite{reisinger2021fast} and references therein 
for computing  $\ol{\phi}$ via linearization
and the efficiency improvement of 
the residual correction method
over plain PGMs.

\end{Remark}

Now we present   the main theorem
on the linear convergence of the PPGM  
\eqref{eq:phi_update}
as $m$ tends to infinity.
 The  precise statement and proof will be given in 
 Section 
 \ref{sec:conv_stationary}
 (see Theorem \ref{thm:convergence_stationary}).

 \begin{Theorem}\label{thm:convergence_stationary_formal}
Suppose (H.\ref{assum:pgm}) holds. For all
 $\phi^0 \in \cV_\bA$ 
 and 
  $\tau \in (0, \frac{2}{\mu + L_{fa}}\wedge \frac{1}{\nu}]$, if one of  conditions \ref{item:T_small}-\ref{item:small_k} holds,
then there exist 
$\phi^\star\in \cV_\bA$ 
and constants  $c \in [0,1)$ and $\widetilde{C}\ge 0$ 
such that 
\begin{enumerate}[(1)]

    \item
    
$\alpha^{\phi^\star}$ 
is a stationary point of
$J(\cdot;\xi_0):\cH^2(\sR^k)\to \sR\cup\{\infty\}$
defined as in \eqref{eq:control_value},
 \item
  for all $m\in \sN_0$,   
    $  |\phi^{m+1}-\phi^{\star}|_0\le  
c|\phi^{m}-\phi^\star|_0$
and  $\|\a^{\phi^{m}}-\a^{\phi^\star}\|_{\cH^2}
\le \widetilde{C}
c^m$.
   
\end{enumerate}

 \end{Theorem}

The precise constant $c$, which determines the rate of convergence, 
 will be given in the proof
based on conditions 
  \ref{item:T_small}-\ref{item:small_k}.
  Roughly speaking, the stronger the cost  convexity  (resp.~the stronger the state dissipativity, 
  the weaker the state/control coupling, the smaller the time horizon, the larger the discount factor), the smaller one can choose $c$ and, hence, the faster the iteration converges.

Note that   Theorem \ref{thm:convergence_stationary_formal} does not require nondegeneracy of 
$\xi_0$ and $\sigma$, and can be extended to quadratically growing cost functions (see Remark \ref{rmk:assum}).
As \eqref{eq:phi_update} concerns iterations of unbounded and nonlinear feedback controls,
the proof of convergence  is rather technical.
Here we  outline the key steps  for the reader's convenience. 

\begin{proof}[\textbf{Sketched proof of Theorem \ref{thm:convergence_stationary_formal}}]
 Observe that a necessary condition on the convergence of $(\a^{\phi^m})_{m\in \sN_0}$ 
 is 
 that 
 $(\|X^{\xi_0,\phi^m}\|_{\cH^2})_{m\in \sN_0}$ are uniformly bounded  in $m$. By standard moment estimates of SDEs, 
  it seems unavoidable to control the Lipschitz constant of 
$( \phi^m)_{m\in \sN}$
in order to obtain the desired convergence result. This uniform regularity estimate is the main technical difficulty in analyzing  \eqref{eq:phi_update},  compared with the analyses of iterative algorithms for open-loop controls in \cite{
   li2018maximum,vsivska2020gradient, kerimkulov2021modified}.
 
 To this end,
 suppose that $\phi^m\in \cV_\bA$ for a  given $m\in \sN_0$.
By exploiting  \eqref{eq:phi_update} and the convexity of $f$ and $\ell$, for  all sufficiently small $\tau>0$,
\begin{align*}
[\phi^{m+1}]_1
&\leq (
1 - \tau C
)
[\phi^m]_1 + \tau C   \left(
\sup_{t,x,x'}\frac{|Y^{t,x,\phi^m}_t -Y^{t,x',\phi^m}_t|}{|x-x'|} +
\sup_{t,x}|Y^{t,x,{\phi}^m}_t| + 1\right),
\end{align*}
where the constant  $C>0$ depends only on  coefficients (see Lemma \ref{Lemma_iteration_1}).
An a-priori estimate of \eqref{bsde_feedback}
and 
the boundedness of $\p_x f$ and $\p_x g$ imply $\sup_{m,t,x}|Y^{t,x,{\phi}^m}_t|<\infty$, 
while Lipschitz estimates of 
\eqref{sde_feedback} and 
\eqref{bsde_feedback}
imply that 
$x\mapsto Y^{t,x,\phi^m}_t$
is Lipschitz continuous uniformly in $t$, where 
the Lipschitz constant
$L_Y([\phi^m]_1)$
depends  \textit{exponentially} on $[\phi^m]_1$
due to the feedback controlled dynamics \eqref{sde_feedback}
(see Proposition \ref{prop:Y_stability}).
Combining these estimates 
gives 
$[\phi^{m+1}]_1
\leq (
1 - \tau C
)
[\phi^m]_1 + \tau C  (L_Y([\phi^m]_1)+1)$.
We then show in Theorem 
\ref{thm:lipschitz_iterates}
that 
under suitable conditions on the coefficients, 
such an exponential dependence can be controlled, and further deduce that 
$\sup_{m}[\phi^m]_1<\infty$.

We then proceed to prove the linear convergence of $(\phi^m)_{m\in \sN}$.
Using the strong convexity of costs, 
for sufficiently small $\tau>0$, 
\begin{align}
\label{eq:linear_conv_sketch}
 |\phi^{m+1}-\phi^{m}|_0    \le 
(
1 - \tau C)
|\phi^{m}-\phi^{{
m-1}}|_0
+\tau C \sup_{t,x}\frac{|Y_t^{t,x,\phi^m}-Y_t^{t,x,\phi^{m-1}}|}{1+|x|},
\quad \fa m\in \sN.
\end{align}
Based on  $\sup_{m}[\phi^m]_1<\infty$, 
we prove by Malliavin calculus that 
$\sup_{m,t,x,s}|Z^{t,x,\phi^m}_s|<\infty$ (see Lemma \ref{lem:MD})
and further 
by stability estimates of \eqref{sde_feedback} and \eqref{bsde_feedback}
that 
$|Y_t^{t,x,\phi^m}-Y_t^{t,x,\phi^{m-1}}|\le \tilde{C} (1+|x|)|\phi^m-\phi^{m-1}|_0$,
for some constant $\tilde{C}$ independent of $t,x,m$ (see Proposition  \ref{prop:Y_stability2}).
By quantifying  $C$ in \eqref{eq:linear_conv_sketch}
and $\tilde{C}$ precisely,
we prove  
under each of the conditions \ref{item:T_small}-\ref{item:small_k}
that  there exists $c\in [0,1)$ such that  
$|\phi^{m+1}-\phi^{m}|_0    \le 
c
|\phi^{m}-\phi^{m-1}|_0$ 
for all $m$, which subsequently implies the convergence of  $(\phi^m)_{m\in \sN}$
due to 
Banach's fixed point theorem. Finally,
we show that 
the limit of 
$(\phi^m)_{m\in \sN}$
induces   
a stationary point of $J(\cdot; \xi_0)$,
based on an equivalent  characterization of stationary points of $J(\cdot; \xi_0)$
in terms of adjoint processes and proximal map of $\ell$
(see Theorem  \ref{thm:convergence_stationary}).
\end{proof}

In practice,  \eqref{bsde_feedback}
can only be solved
approximately and  the update step \eqref{eq:phi_update} for the feedback controls can only be performed with this approximate solution, which  feeds the errors into subsequent iterations.
Hence we 
further quantify this effect by 
establishing a stability property of \eqref{eq:phi_update} under perturbations of solutions to \eqref{bsde_feedback}. 
 For clarity, we only carry out perturbation analysis for the computation of $Y^{t,x,\phi}_t$, 
but similar analysis can be performed 
for \eqref{eq:phi_update}
with inexact computation of the proximal map $\prox_{\tau \ell}$.
{
Our analysis    allows for  stochastic approximations of 
$Y^{t,x,\phi}_t$ resulting from     applying  
   probabilistic numerical methods
to solve   \eqref{bsde_feedback}
(see e.g.~\cite{gobet2005regression,han2018solving}).
}

More precisely, 
let 
  ${\phi}^0\in \cV_\bA$
  be an initial guess 
  and 
    $\tau>0$ be a stepsize.
{
At the $m$-th iteration with $m\in \sN_0$,
   let $\widetilde{\phi}^m$
   be the (random) feedback control obtained at the previous iteration
    (with $\widetilde{\phi}^0=\phi^0$).
    That is, $\widetilde{\phi}^m:[0,T]\t \sR^n\t \Omega\to \bA$ is a measurable function such that 
     $\widetilde{\phi}^m_\cdot(\cdot,\om)\in \cV_\bA$ for a.s.~$\om\in \Om$.
    Consider 
%
%
%
%
  a {
    measurable} function 
$\widetilde{\mathcal{Y}}^{\widetilde{\phi}^m}:[0,T] \t \sR^n\t {
\Omega}\to \sR^n$
such that for a.s.~$\om\in \Om$,
$(t,x)\mapsto \widetilde{\mathcal{Y}}^{\widetilde{\phi}^m}_t(x,\omega)$
approximates  $
 (t,x)\mapsto   \cY_t^{\widetilde{\phi}^m}(x,\omega)\coloneqq  Y^{t,x,\widetilde{\phi}^m_\cdot(\cdot,\om)}_t$,
 where
$Y^{t,x,\widetilde{\phi}^m_\cdot(\cdot,\om)}_t\in \sR^n$
satisfies  \eqref{bsde_feedback} with the realised control $\widetilde{\phi}^m_\cdot(\cdot,\om)\in \cV_\bA$.
}%
The {
(random)} feedback control 
for the next iteration  is then obtained via a proximal gradient update
\eqref{eq:phi_update} 
based on $\widetilde{\mathcal{Y}}^{\widetilde{\phi}^m}$:
 \begin{align}
\label{eq:phi_update22}
 \widetilde{\phi}_t^{m+1}(x)&= \prox_{\tau\ell} \big(\widetilde{\phi}_t^m(x)-\tau {
  \partial_{a}} H_t^{\textrm{re}}(x,\widetilde{\phi}_t^m(x),\widetilde{\mathcal{Y}}_t^{\widetilde{\phi}^m}(x)) \big),
 \q (t,x)\in [0,T]\t \sR^n,
\end{align}
{
where the identity is understood in an almost sure sense.
}

The following theorem  
shows the accuracy of 
\eqref{eq:phi_update22}, 
whose precise statement and proof will  be given in Section \ref{sec:conv_stationary} (see  Theorem \ref{thm:conv_approximate}).
Here we  assume that $\widetilde{\mathcal{Y}}^{\widetilde{\phi}^m}$
approximates the 
 function $
{\mathcal{Y}}^{\widetilde{\phi}^m}$
well enough such that 
the resulting 
 controls $\widetilde{\phi}^m$ are uniformly bounded in time and uniformly Lipschitz in space.

\begin{Theorem}\l{thm:conv_approximate_formal}

Suppose (H.\ref{assum:pgm}) holds.
For all
 $\phi^0 \in \cV_\bA$ 
 and 
  $\tau \in (0, \frac{2}{\mu + L_{fa}}\wedge \frac{1}{\nu}]$, if 
  {
  $\sup_{m\in \sN, \om\in \Omega}(|\widetilde{\phi}^m_\cdot(\cdot,\om)|_0+[\widetilde{\phi}^m_\cdot(\cdot,\om)]_1)<\infty$},
  and 
  one of the conditions \ref{item:T_small}-\ref{item:small_k}  holds,
then there exist constants $c\in [0,1)$ and $C\ge 0$
such that 
{
for a.s.~$\om\in \Omega$ and
for all  $m\in\sN_0$,
\begin{align*}
    |\widetilde{\phi}^{m}_\cdot(\cdot,\om)-\phi^{\star} |_0 \leq c^{m}|\phi^0-\phi^\star|_0 +
    C
     \sum_{j=0}^{m-1} c^{m-1-j} 
     |\cY^{\widetilde{\phi}^j}_\cdot(\cdot,\om)- \widetilde{\mathcal{Y}}^{\widetilde{\phi}^j}_\cdot(\cdot,\om)|_0,
\end{align*}
where 
$\phi^\star\in \cV_\bA$ 
is the limit function  in Theorem \ref{thm:convergence_stationary_formal}.
Consequently, for all $p\ge 1$ and $m\in \sN_0$, 
\begin{align*}
  \sE[  |\widetilde{\phi}^{m}-\phi^{\star} |^p_0]^{\frac{1}{p}} \leq c^{m}|\phi^0-\phi^\star|_0 +
    C
     \sum_{j=0}^{m-1} c^{m-1-j} 
     \sE[|\cY^{\widetilde{\phi}^j}- \widetilde{\mathcal{Y}}^{\widetilde{\phi}^j}|^p_0]^{\frac{1}{p}}.
     \end{align*}

}%
\end{Theorem}

{
  
 By Lemma \ref{Lemma_iteration_1},
the condition $\sup_{m\in \sN}(|\widetilde{\phi}_t^m|_0+[\widetilde{\phi}^m]_1)<\infty$
holds if 
 there exists $C>0$ such that for all  
$m\in \sN$, $\om\in \Om$,
 and $(t,x,x') \in [0,T]\t \sR^n \t \sR^n$,
 $
|\widetilde{\mathcal{Y}}^{\widetilde{\phi}^m}_t(x,\om)|
\le C$ 
and 
$  
|\widetilde{\mathcal{Y}}^{\widetilde{\phi}^m}_t(x,\om)-\widetilde{\mathcal{Y}}^{\widetilde{\phi}^m}_t(x',\om)|
\le C{|x-x'|}$.
This highlights the fact that 
the numerical  approximations of the gradient directions 
must be sufficiently regular in space,
to prevent a spatial oscillation 
of the iterates 
and   to ensure the  convergence of  the iterates.
}%
This is a reasonable assumption 
as  the exact gradient directions  $( {\mathcal{Y}}^{ {\phi}^m})_{m\in \sN_0} $
enjoy these properties 
(see Proposition \ref{prop:Y_stability}
and  Theorems
\ref{prop:Y_bound} and 
 \ref{thm:lipschitz_iterates}),
and 
any reasonable
approximation 
$\widetilde{\mathcal{Y}}^{\widetilde{\phi}^m}$
of 
$
{\mathcal{Y}}^{\widetilde{\phi}^m}$
should retain these properties;
see e.g., 
\cite{barles2002convergence} for  approximation schemes that preserve
boundedness and Lipschitz continuity of 
exact solutions. 
{
 It would be interesting to derive explicit conditions on   model coefficients 
to ensure the required regularity of 
$(\widetilde{\phi}^m)_{m\in \sN_0}$. 
This would entail imposing precise dependencies of 
 the Lipschitz regularity of 
$\widetilde{\mathcal{Y}}^{\widetilde{\phi}^m}$
on  the semi-norm  $[\widetilde{\phi}^m]_1$,
and is left for future research.

}

\section{Proofs}

Throughout the rest of this work, we establish 
estimates with explicit dependence  on the constants 
$T, \rho,C_{fx}, L_{fx},  L_{fa}, \mu, \nu, C_g, L_g, \kappa_{\hat{b}}, C_{\bar{b}}$,
which are important for the convergence of  \eqref{eq:phi_update}. 
For notational simplicity, we 
{
write $(x)_{+}=\max(0,x)$ for all $x \in \mathbb{R}$,}
and  denote by  $C>0$ 
a generic constant, which depends  on the remaining constants appearing in  (H.\ref{assum:pgm}), and may take a different value at each occurrence.
We shall refer to $C>0$ as an absolute constant if its value is independent of the constants in (H.\ref{assum:pgm}).
Dependence of $C$ on important quantities will be indicated explicitly by $C_{(\cdot)}$, e.g., $C_{(\phi)}$ for $\phi\in \cV_\bA$.


\subsection{Auxiliary lemmas}

In this section, we present some technical lemmas  used in the subsequent analysis.  The following lemma establishes 
stability of  SDEs with non-Lipschitz drift coefficients.
The upper bounds involve explicit dependence on  relevant constants, whose proof is given in Appendix 
\ref{appendix:technical}.

\begin{Lemma}\label{forward:apriori_p}
Let   
$T>0$,
and for each $i=1,2$,
let $\mu_{i}\in \sR$,
  $\nu_{i}\ge 0$, let  
  { $b^i:[0,T]\t \sR^n\to \sR^{
  n}$}
 and ${\sigma}^i:[0,T]\t \sR^n\to \sR^{n\t d}$ 
  be measurable functions such that for all $t\in [0,T]$ and $x,x'\in \sR^n$,
$\sup_{(t,x)\in [0,T]\t \sR^n}\frac{|b^i_t(x)|+|\sigma^i_t(x)|}{1+|x|}<\infty$, 
 $\la x-x',b^i_t(x)-b^i_t(x')\ra  \le \mu_i |x-x'|^2$,
 and 
 $|\sigma^i_t(x)-\sigma^i_t(x')| \le \nu_i |x-x'|$, 
 and for each  $(t,x)\in [0,T]\t  \sR^n$,
 let $X^{t,x,i}\in \cS^2(t,T; \sR^n)$ satisfy
\begin{equation}
\d X_s= b^i_s(X_s)\, \d s +\sigma^i_s(X_s)\, \d W_s, 
\q s\in [t,T]; 
\quad X_t = x.  
\end{equation}
Then for all $p  \geq 2$ there exists an absolute constant $C_{(p)}$
 such that   for all $t\in [0,T]$, $x_1,x_2\in \sR^n$,
\begin{align*}
\|X^{t,x_1,1} - {X}^{t,x_2,2}\|_{\cS^p }
& \le 
C_{(p)} e^{T(2\mu_1+C_{(p)}\nu_1^2)_{+}}
\Big(|x_1-x_2|+\sqrt{T} \| b^1({X}^{t,x_2,2})-{b}^2({X}^{t,x_2,2}) \|_{\cH^p }
\\
&\quad 
+\| \sigma^1({X}^{t,x_2,2})-\sigma^2({X}^{t,x_2,2}) \|_{\cH^p }\Big). 
\end{align*}
If we further assume that  $\sigma^1 \equiv \sigma^2$, then   for all $(t,x) \in
[0,T]\t \sR^n$,  
\begin{align}
\label{eq:X_b_difference}
    \mathbb{E} \left[ |X^{t,x,1}_T - X^{t,x,2}_T|^2 \right] \leq \mathbb{E} \left[ \int_{t}^{T} |b_s^{1}(X^{t,x,2}_{s})-b^{2}_{s}(X^{t,x,2}_{s})|^2 e^{(T-s)(2\mu_1 +\nu_1^2 +1)} \, \mathrm{d}s \right].
\end{align}
\end{Lemma}

The following lemma 
establishes 
stability of  BSDEs with monotone nonlinearity. 
It 
 has been proved 
 in 
\cite{pardoux1998backward} for $p=1$
and in 
\cite[Proposition 3.2]{briand2000bsdes}
for $p>1$. 


\begin{Lemma}\label{BSDE:apriori}
For each $i=1,2$
and $t\in [0,T]$, 
let 
$\xi^i\in L^2(\cF_T;\sR^n)$,
$\gamma_i\ge 0$, $\mu_i\in \sR$, 
 $f^i:[t,T]\t \Om\t  \sR^n\t \sR^{n\t d}\to \sR^n$ 
 be 
such that 
for all  $(y,z)\in  \sR^n\t \sR^{n\t d}$,
$(f^i_s(\cdot,y,z))_{s\in [t,T]}$ is progressively measurable, and
for all  $(s,\om)\in [t,T]\t \Om$, $y,y'\in  \sR^n$ and $z,z'\in \sR^{n\t d}$,
$|f^i_s(\om,y,z)-f^i_s(\om,y,z')|\le \gamma_i|z-z'|$
and $\la y-y',f^i_s(\om,y,z)-f^i_s(\om,y',z)\ra \le \mu_i|y-y'|^2$, and let $(Y^i,Z^i)\in \cS^2(t,T;\sR^n)\t  \cH^2(t,T;\sR^{n\t d})$ satisfy
$$
\d Y_s
 = - f^i_s(\cdot, Y_s,Z_s)  \, \mathrm{d}s+ Z_s \, \mathrm{d}W_s, \q s\in [t,T]; \quad Y_T = \xi^i.
$$
Then for all $p\ge 1$ and  $\eps\in (0,1)$, there exists an absolute constant $C_{(p,\eps)}>0$
such that 
for  all 
 $t\in [0,T]$ and
 $\a\ge \eps^{-1}\gamma_1^2+2\mu_1$,
\begin{align*}
&\sE\bigg[\sup_{s\in [t,T]}e^{p\a s}|Y^1_s-Y^2_s|^{2p}
+\bigg(\int_t^Te^{\a s}|Z^1_s-Z^2_s|^2\, \d s\bigg)^p\bigg]
\\
&\le
C_{(p,\eps)}
\sE\bigg[ e^{p\a T}|\xi^1-\xi^2|^{2p}
+\bigg(\int_t^Te^{\frac{\a}{2}s}
|f^1_s(\cdot, Y^2_s,Z^2_s)-f^2_s(\cdot, Y^2_s,Z^2_s)|\,\d s\bigg)^{2p}\bigg].
\end{align*}
\end{Lemma}

The next lemma estimates the monotonicity and Lipschitz continuity of $\p_x H$. The proof follows 
directly from \eqref{eq:Hamiltonian}
 and (H.\ref{assum:pgm}),
 and is given in Appendix \ref{appendix:technical}.
Due to 
the presence of $(\p_x \hat{b}_t(x))^\top y$ and the unboundedness  of $\p_x \hat{b}$,   $\p_x H$  is not  globally Lipschitz continuous in $y$.

\begin{Lemma}
\label{lemma:H_gradient_estimate}
Suppose (H.\ref{assum:pgm}) holds,
and let $H$ be defined by 
\eqref{eq:Hamiltonian}.
Then for all $t\in [0,T]$, $x,x'\in \sR^n$, $a,a'\in \bA$, $y,y' \in \sR^n$ and $z,z'\in \sR^{n\t d}$,
\begin{align}
&    \la y-y',
\partial_x H_t(x, a,y,z)
-\partial_x H_t(x, a,y',z)\ra 
\le (\kappa_{\hat{b}}-{
\rho}+L_{\bar{b}})|y-y'|^2,
\\
&|\partial_x H_t(x, a,y,z)
-\partial_x H_t(x', a',y,z')|
\nb
\\
& \q 
\le 
\big(L_{\hat{b}}|x-x'|+
L_{\bar{b}}(|x-x'|+|a-a'|)\big)| y| 
\nb
+ L_{\sigma}|x-x'|
|z|
\\
&\q
\q +L_{\sigma}|z-z'|
+L_{fx}(|x-x'|+|a-a'|).
\end{align}
\end{Lemma}

We then  present a Lipschitz estimate for the proximal 
gradient mapping \eqref{eq:phi_update}.

\begin{Lemma}\label{Lemma_iteration_1}
Suppose (H.\ref{assum:pgm}) holds,
and 
let 
$H^{\textrm{re}}
$ be defined as in \eqref{eq:Hamiltonian_re}.
Then for all $t\in [0,T]$, $x,x'\in \sR^n$, $a,a'\in\bA$, $y,y'\in \sR^n$
and 
$\tau \in (0, \frac{2}{\mu + L_{fa}}\wedge \frac{1}{\nu}]$,
\begin{align*}
&|\prox_{\tau \ell}(a-\tau\p_a H_t^{\textrm{re}}(x,a,y))-
\prox_{\tau \ell}(a'-\tau\p_a H_t^{\textrm{re}}(x',a',y'))|
\\
&\le 
\left(
1 - \tau 
\frac{1}{2}
\left(\frac{\mu L_{fa}}{\mu + L_{fa}}
+\nu
\right)
\right)
|a - a'|
+\tau C_{\bar{b}}|y-y'|+ \tau
(L_{\bar{b}} | y'|
+L_{fa})|x-x'|.
\end{align*}

\end{Lemma}

\begin{proof}
For each $\tau>0$,
since
 $\tau \ell$ is proper, lower-semicontinuous and $\tau\nu$-strongly convex (cf.~\eqref{eq:strong_convex_l}),
 by Theorem 12.56
 and Exercise 12.59 in \cite{rockafellar2009variational},
$$|\prox_{\tau \ell}(x)-\prox_{\tau \ell}(y)|\le \frac{1}{1+\tau \nu}|x-y|,
\q \fa x,y\in \sR^k.
$$
 Hence
for any 
$t\in [0,T]$, $x,x'\in \sR^n$, $a,a'\in \bA$ and $y,y'\in \sR^n$,
\begin{align}
\label{eq:prox_difference}
\begin{split}
&|\prox_{\tau \ell}(a-\tau\p_a H_t^{\textrm{re}}(x,a,y))-
\prox_{\tau \ell}(a'-\tau\p_a H_t^{\textrm{re}}(x',a',y'))|
\\
&
\le
\tfrac{1}{1+\tau \nu}
|(a-\tau\p_a H_t^{\textrm{re}}(x,a,y))-
(a'-\tau\p_a H_t^{\textrm{re}}(x',a',y'))|
\\
&\le
\tfrac{1}{1+\tau \nu}
|(a-\tau\p_a H_t^{\textrm{re}}(x,a,y))-
(a'-\tau\p_a H_t^{\textrm{re}}(x,a',y))|
\\
&\q +\tfrac{\tau}{1+\tau \nu} |\p_a H_t^{\textrm{re}}(x,a',y)-\p_a H_t^{\textrm{re}}(x',a',y')|.
\end{split}
\end{align}
We now estimate the two terms in \eqref{eq:prox_difference}
 separately. 
Observe that $\p_a H_t^{\textrm{re}}(x,a,y)=\bar{b}_t(x)^\top y+\p_a f_t(x,a)$
for all $(t,x,a,y)\in [0,T]\t \sR^n\t \bA\t \sR^n$.
Then 
by 
\eqref{eq:b_bound},
\eqref{eq:f_x_lipschitz} and \eqref{eq:f_a_lipschitz},
 the second term 
 in \eqref{eq:prox_difference} can be bounded by:
\begin{align}
\label{eq:prox_difference_term2}
\begin{split}
 |\p_a H_t^{\textrm{re}}(x,a',y)-\p_a H_t^{\textrm{re}}(x',a',y')|
&\le
|\bar{b}_t(x)^\top y-\bar{b}_t(x')^\top y'|
+
|\p_a f_t(x,a')-\p_a f_t(x',a')|
\\
&\le 
C_{\bar{b}}|y-y'|+
(L_{\bar{b}} | y'|
+L_{fa})|x-x'|.
\end{split}
\end{align}
To estimate the first term in \eqref{eq:prox_difference},
observe that 
for all $(t,x,y)\in [0,T]\t\sR^n\t \sR^n$,
by 
\eqref{eq:affine},
\eqref{eq:b_bound},
\eqref{eq:f_a_lipschitz}
and \eqref{eq:strong_convex_f},
$\bA \ni a \mapsto H_t^{\textrm{re}}(x,a,y)\in \sR$ is $\mu$-strongly convex, and  $\bA \ni a \mapsto \partial_a H_t^{\textrm{re}}(x,a,y)\in \sR^k$ is $L_{fa}$-Lipschitz continuous,
which along with 
 \cite[Theorem 2.1.12]{nesterov2003introductory},
implies for all $a, a' \in \bA$,
\begin{align*}
& \left \langle \partial_a H_t^{\textrm{re}}(x,a,y) - \partial_a H_t^{\textrm{re}}(x,a',y), a-a' \right \rangle \\
& \geq \frac{\mu L_{fa}}{\mu +L_{fa}}|a-a'|^2 + \frac{1}{\mu + L_{fa}} | \partial_a H_t^{\textrm{re}}(x,a,y) -  \partial_a H_t^{\textrm{re}}(x,a',y)|^2.
\end{align*}
Hence for all $a,a' \in \bA$ and $(t,x,y)\in [0,T]\t\sR^n\t \sR^n$,
and $\tau\in (0, \frac{2}{\mu + L_{fa}}]$,
\begin{align*}
\begin{split}
    &|(a-\tau\p_a H_t^{\textrm{re}}(x,a,y))-
(a'-\tau\p_a H_t^{\textrm{re}}(x,a',y))|^2
\\
&=
|a-a'|^2- 2\tau \la a-a', \p_a H_t^{\textrm{re}}(x,a,y)-\p_a H_t^{\textrm{re}}(x,a',y)\ra +
\tau^2 | \p_a H_t^{\textrm{re}}(x,a,y)-\p_a H_t^{\textrm{re}}(x,a',y)|^2
\\
&\le 
 \left(1 - 2\tau \frac{\mu L_{fa}}{\mu + L_{fa}} \right)|a - a'|^2
+\tau\left(\tau -\frac{2}{\mu + L_{fa}} \right)| \p_a H_t^{\textrm{re}}(x,a,y)-\p_a H_t^{\textrm{re}}(x,a',y)|^2
\\
&\le \left(1 - 2\tau \frac{\mu L_{fa}}{\mu + L_{fa}} \right)|a - a'|^2.
\end{split}
\end{align*}
Taking the square root of both sides of the above estimate 
and  using   the inequality
 $\sqrt{1- {
 \gamma} \tau}\le 1- {
  \gamma} \tau/2$ for all $ {
  \gamma},\tau\ge 0$ and ${
   \gamma}\tau \le 1$ give that
 \begin{align*}
\begin{split}
    &|(a-\tau\p_a H_t^{\textrm{re}}(x,a,y))-
(a'-\tau\p_a H_t^{\textrm{re}}(x,a',y))|
\le \left(1 - \tau \frac{\mu L_{fa}}{\mu + L_{fa}} \right)|a - a'|.
\end{split}
\end{align*}
This along with 
\eqref{eq:prox_difference}, \eqref{eq:prox_difference_term2}
and 
$\tfrac{\tau}{1+\tau \nu}\le \tau$
shows that 
for all
 $\tau\in (0, \frac{2}{\mu + L_{fa}}]$,
\begin{align*}
&|\prox_{\tau \ell}(a-\tau\p_a H_t^{\textrm{re}}(x,a,y))-
\prox_{\tau \ell}(a'-\tau\p_a H_t^{\textrm{re}}(x',a',y'))|
\\
&\le 
\frac{1}{1+\tau \nu}
\left(
1 - \tau \frac{\mu L_{fa}}{\mu + L_{fa}}
\right)
|a - a'|
+\tau C_{\bar{b}}|y-y'|+ \tau
(L_{\bar{b}} | y'|
+L_{fa})|x-x'|.
\end{align*}
Observe that 
for all  $a,b,\tau \ge 0$ with $0\le \tau b\le 1$, 
$1-\tau a\le (1+\tau b)(1-\tau \frac{a+b}{2})$.
Then setting 
$a=\frac{\mu L_{fa}}{\mu + L_{fa}}$ and $b=\nu$
in the inequality 
shows that the desired estimate holds with 
$\tau\in 
(0, \frac{2}{\mu + L_{fa}}\wedge \frac{1}{\nu}]
$.
\end{proof}

\subsection{Uniform boundedness  
in time}
 

To establish the boundedness of $\phi^m_t(0)$,
we first
 prove  the adjoint processes 
$(Y^{t,x,\phi},Z^{t,x,\phi})$
defined in \eqref{bsde_feedback}
have bounded $p$-th moments.

\begin{Proposition}\label{prop:Y_bound}
Suppose (H.\ref{assum:pgm}) holds.
For each
 $\phi\in \cV_\bA$ and 
$(t,x)\in [0,T]\t \sR^n$, 
let
$(Y^{t,x,\phi},Z^{t,x,\phi})\in \cS^2(t,T;\sR^n)\t \cH^2(t,T;\sR^{n\t d})$ 
be defined by 
\eqref{bsde_feedback}.
Then 
for all $p\ge 1$
there exists $C_{(p)}\ge 0$,
such that
for all 
$\phi\in \cV_\bA$,
$(t,x)\in [0,T]\t \sR^n$,
\begin{align}
\label{eq:u_bound}
&
\sE\bigg[\sup_{s\in [t,T]}e^{p\widetilde{\a} s}|Y^{t,x,\phi}_s|^{2p}
+\bigg(\int_t^Te^{\widetilde{\a} s}|Z^{t,x,\phi}_s|^2\, \d s\bigg)^p\bigg]
\le
C_{(p)}\bigg(
e^{p\widetilde{\a} T}C_g^{2p}
+C_{fx}^{2p}
\bigg(\int_t^Te^{\frac{\widetilde{\a}}{2} s}
\,\d s
\bigg)^{2p}
\bigg),
\end{align}
with $\widetilde{\a}=2(\kappa_{\hat{b}} -{
\rho}+L_{\bar{b}}+L_{\sigma}^2)$.
Consequently,
there exists an absolute constant  $C\ge 0$ such that 
for all  $\phi\in \cV_\bA$ and $(t,x)\in [0,T]\t \sR^n$,
\begin{align}\label{eq:Y_uniform_bdd}
|Y^{t,x,\phi}_t|
\le
C_Y\coloneqq 
C 
(C_g
+C_{fx}T)
e^{(\kappa_{\hat{b}}-{
\rho} +C )_+T}.
\end{align}
\end{Proposition}
 
 \begin{proof}
Let 
$\bar{f}^1:[t,T]\t \Om \t 
 \sR^n  \t \sR^{n\t d}
\to \sR^n$ be such that
for all $(s,\om,y,z)\in [t,T]\t \Om \t 
 \sR^n  \t \sR^{n\t d}$,
$\bar{f}^1_s(\om,y,z)= \partial_x H_s(X^{t,x,\phi}_s(\om), \phi_s(X^{t,x,\phi}_s(\om)),y,z)$,
where $H$ is defined in \eqref{eq:Hamiltonian},
and $X^{t,x,\phi}
$
is defined by 
\eqref{sde_feedback}.
Then by Lemma \ref{lemma:H_gradient_estimate},
$|\bar{f}^1_t(\om,y,z)-\bar{f}^1_t(\om,y,z')|\le L_{\sigma}|z-z'|$,
and 
\begin{align}
\label{eq:bar_f_1}
\begin{split}
&\la y-y',\bar{f}^1_t(\om,y,z)-\bar{f}^1_t(\om,y',z)\ra 
\le (\kappa_{\hat{b}}-{
\rho}+L_{\bar{b}})|y-y'|^2.
\end{split}
\end{align}
By applying Lemma \ref{BSDE:apriori}
with $f^1=\bar{f}^1$,  $\xi^1=\p_x g(X^{t,x,\phi}_T)$, $f^2=0$, $\xi^2=0$,
$Y^2=Z^2=0$,
 $\eps=1/2$ and $\a=2(\kappa_{\hat{b}}-{
 \rho}+L_{\bar{b}}+L_{\sigma}^2)$, it holds with some constant $C\ge 0$ that, for all $p\ge 1$,
\begin{align*}
&
\sE\bigg[\sup_{s\in [t,T]}e^{p\a s}|Y^{t,x,\phi}_s|^{2p}
+\bigg(\int_t^Te^{\a s}|Z^{t,x,\phi}_s|^2\, \d s\bigg)^p\bigg]
\\
&\le 
C_{(p)}
\sE\bigg[ e^{p\a T}|\p_x g(X^{t,x,\phi}_T)|^{2p}
+\bigg(\int_t^Te^{\frac{\a}{2}s}
|\p_x f_s(X^{t,x,\phi}_s, \partial_x \phi(X^{t,x,\phi}_s))|\,\d s\bigg)^{2p}\bigg]
\\
&\le
C_{(p)}\bigg(
e^{p\a T}C_g^{2p}
+C_{fx}^{2p}
\bigg(\int_t^Te^{\frac{\a}{2} s}
\,\d s
\bigg)^{2p}
\bigg),
\end{align*}
where the last inequality follows from 
\eqref{eq:f_x_lipschitz}
and \eqref{eq:g_bound}.

Consequently, by setting $p=1$ in the above estimate and taking the square root of both sides,
there exists an absolute constant  $C\ge 0$ such that 
for all $(t,x)\in [0,T]\t \sR^n$,
\begin{align}
|Y^{t,x,\phi}_t|
&\le 
C
e^{-\frac{\a}{2}t}
\bigg(
e^{\f{\a}{2}T}C_g
+C_{fx}\int_t^T e^{\f{\a}{2}s}\,\d s \bigg)
\le
C 
(C_g
+C_{fx}T)
e^{(\frac{\a}{2} T)_+}.
\end{align}
This finishes the proof of the proposition.
 \end{proof}
 
Based on 
Lemma \ref{Lemma_iteration_1} and 
Proposition \ref{prop:Y_bound},
we now establish the 
uniform boundedness of $\phi^m_t(0)$.

\begin{Theorem}\label{thm:phi_bound}
Suppose (H.\ref{assum:pgm}) holds.
Let 
$a_0\in \bA$
and $z^{a_0}\in \sR^k$
 such that $z^{a_0}\in \p^s \ell(a_0)$.\footnotemark
 \footnotetext{
 For any $a\in \bA=\operatorname{dom}\ell$,  
 the convex subdifferential
 of $\ell$ at $a$
 is defined as 
 $\p^s \ell(a)\coloneqq\{z\in \sR^k\mid \ell(a')-\ell(a)\ge \la z, a'-a\ra, 
\; \fa a'\in \sR^k
\}$.
As $\ell$ is proper, lower semicontinuous and convex,
$\p^s \ell(\cdot)$ is nonempty on a dense subset of
 $\bA$
 by \cite[Corollary 2.44]{barbu2012convexity}.
 }
For each
 $\phi^0\in \cV_\bA$,
 $\tau >0$
 and 
$m\in \mathbb{N}$,
let $\phi^m$ be defined by \eqref{eq:phi_update}.
Then for all $\phi^0\in \cV_\bA$
and
{ 
$\tau \in (0, \frac{2}{\mu + L_{fa}} \wedge \frac{1}{\nu}]$},
\begin{align}
\label{eq:phi_bound_constant}
\begin{split}
 \sup_{m\in\sN_0,t\in [0,T]}|\phi^{m}_t(0)|
\le 
C_{(\phi^0)}
&\coloneqq 
\sup_{t\in [0,T]}|\phi_t^0(0)|
+2\Big( \tfrac{1}{{ \mu+\nu}}C_{\bar{b}} C_Y+
 \tfrac{2}{{ \mu+\nu}}(C_{fa}+L_{fa}|a_0|+
|z^{a_0}|)+|a_0|
\Big)
\\
&\quad
+{4C_YC_{\bar{b}}\frac{\mu + L_{fa}}{(\mu+\nu) L_{fa} + \mu\nu}},
\end{split}
\end{align}
where the constant $C_Y \ge 0$  is defined by \eqref{eq:Y_uniform_bdd}.

\end{Theorem}
\begin{proof}
For each $(t,x, a)\in [0,T]\t  \sR^n\t  \sR^k$ and $u\in \sR^n$, 
 let 
$h_t(x,a)= f_t(x,a)+\ell(a)$
and 
$\phi^\star_t[u]=\argmin_{a\in \sR^k} (H^\textrm{re}_t(0,a,u)+\ell(a))$, with 
$H^{\textrm{re}}
$  defined as in \eqref{eq:Hamiltonian_re}.
By  \eqref{eq:Hamiltonian_re}
and \eqref{eq:affine},
\bb\label{eq:phi_optimal_u}
\phi^\star_t[u]=\argmin_{a\in \sR^k} 
\big(\la \bar{b}_t(0) {
 a},u\ra +h_t(0,a)\big)
=
\p_z h^*_t(0,-\bar{b}_t^\top(0) u),
\ee
where 
$h^*:[0,T]\t \sR^n\t \sR^k\to \sR^k$ is the convex conjugate function of $h$ defined by
\bb\l{eq:conjugate}
 h^*_t(x,z)\coloneqq \sup\{\la a,z\ra-h_t(x,a)\mid a\in \sR^k\}.
\ee
Note that by (H.\ref{assum:pgm}),
for each $(t,x)\in [0,T]\t \sR^n$,
the function
$a\mapsto h_t(x,a)$ is proper, lower semicontinuous
and 
{$(\mu+\nu)$-strongly} convex,
which implies that 
$z\mapsto h^*_t(x,z)$  is finite and differentiable on $\sR^k$,
  $z\mapsto \p_z h^*_t(x,z)$ 
  is  {$\tfrac{1}{\mu+\nu}$}-Lipschitz continuous, 
    and $\p_z h^*_t(x,z)=\argmax_{a\in \sR^k} \big(\la a, z\ra -h_t(x,a)\big)$.
    
Let $
a_0\in \operatorname{dom} \ell$
and $z^{a_0}\in \sR^k$
 such that $z^{a_0}\in \p^s \ell(a_0)\not = \emptyset$.
 Then  for all $t\in [0,T]$, by the differentiability and convexity of $a\mapsto f_t(0,a)$, $\p_a f_t(0,a_0)+
z^{a_0} \in \p^s h_t(0, a_0)$
(see \cite[Corollary 10.9]{rockafellar2009variational}).
Hence, by 
the fact that
$\p_z h^*_t(0,0)=\argmin_{a\in \sR^k} h_t(0,a)$
and  
the {$(\mu+\nu)$-strong} convexity of $a\mapsto h_t(0,a)$,
\begin{align*}
h_t(0,a_0)\ge h_t(0,\p_z h^*_t(0,0))
\ge h_t(0,a_0)+\la \p_a f_t(0,a_0)+
z^{a_0}, \p_z h^*_t(0,0)-a_0\ra +\tfrac{{ \mu+\nu}}{2}|\p_z h^*_t(0,0)-a_0|^2,
\end{align*}
which implies that 
$$
|\p_z h^*_t(0,0)-a_0|\le \tfrac{2}{{ \mu+\nu}}|  \p_a f_t(0,a_0)+
z^{a_0}|\le
 \tfrac{2}{{ \mu+\nu}}(C_{fa}+L_{fa}|a_0|+
|z^{a_0}|),
$$
where the last inequality follows from \eqref{eq:f_a_lipschitz}.
Hence by using 
\eqref{eq:affine}  and 
the {$\tfrac{1}{\mu +\nu}$}-Lipschitz continuity of $z\mapsto \p_zh^*_t(0,z)$, for all $t\in [0,T]$ and $u\in \sR^n$, 
\begin{align}\label{eq:phi_star_u}
\begin{split}
|\phi_t^\star[u]|
&\le |\p_z h^*_t(0,-\bar{b}_t^\top(0) u)
-\p_z h^*_t(0,0)|+|\p_z h^*_t(0,0)|
\\
&\le \tfrac{1}{{ \mu+\nu}}C_{\bar{b}} |u|+
 \tfrac{2}{{ \mu+\nu}}(C_{fa}+L_{fa}|a_0|+
|z^{a_0}|)+|a_0|.
\end{split}
\end{align}
Let $\phi^0\in \cV_\bA$ and { $\tau \in (0, \frac{2}{\mu + L_{fa}}\wedge \frac{1}{\nu}]$} be fixed in the subsequent analysis. 
For each  $t\in [0,T]$, let $c^0_t\coloneqq \phi_t^\star[Y^{t,0,\phi^0}_t]$.
{
Observe that 
for all $t\in [0,T]$, $c\in \sR^k$ and $u\in \sR^n$, 
by the definition of $\prox_{\tau\ell} $ in \eqref{eq:proximal},
\begin{align*}
&
c=\argmin_{a\in \sR^k} (H^\textrm{re}_t(0,a,u)+\ell(a))
\Longleftrightarrow
0\in \p_a H^\textrm{re}_t(0,c ,u)+\p^s \ell(c) 
\\
&\quad \Longleftrightarrow
0\in \big(c-(c-\tau \p_a H^\textrm{re}_t(0,c ,u)) \big) +\p^s (\tau   \ell)(c) 
\Longleftrightarrow
c=\prox_{\tau\ell} (c-\tau \p_a H^\textrm{re}_t(0,c ,u)). 
\end{align*}
}%
Then 
for all $t\in [0,T]$, 
the fact that 
$c^0_t=\argmin_{a\in \sR^k} (H^\textrm{re}_t(0,a,Y^{t,0,\phi^0}_t)+\ell(a))$
implies that
$   c^0_t=\prox_{\tau \ell}( c^0_t-\tau \p_a H^{\textrm{re}}_t(0, c^0_t,Y^{t,0,\phi^0}_t))$.
Hence 
for all $m\in \sN_0$ and  $t\in [0,T]$, 
\eqref{eq:phi_update} and 
 Lemma \ref{Lemma_iteration_1}
 imply that 
 \begin{align*}
& |\phi^{m+1}_t(0)-c^{0}_t|
\\
& =\big|\prox_{\tau\ell} \big(\phi_t^m(0)-\tau {
 \partial_{a}} H_t^{\textrm{re}}(0,\phi_t^m(0),Y^{t,0,\phi^m}_t) \big)
 -\prox_{\tau \ell}( c^0_t-\tau \p_a H^{\textrm{re}}_t(0, c^0_t,Y^{t,0,\phi^0}_t)) \big|
 \\
&\le { \left(
1 - \tau \frac{1}{2}
\left(\frac{\mu L_{fa}}{\mu + L_{fa}}
+\nu
\right)
\right)}|\phi_t^m(0) - c^{0}_t|
+\tau C_{\bar{b}}|Y^{t,0,\phi^m}_t-Y^{t,0,\phi^0}_t|.
\end{align*}
By Proposition \ref{prop:Y_bound}, there exists an absolute constant $C\ge 0$ such that 
for all $t\in [0,T]$ and $\phi\in \cV_\bA$, 
$|Y^{t,0,\phi}_t|
\le C_Y\coloneqq C 
(C_g
+C_{fx}T)
(e^{\a_+ T}),$
with 
$\a=\kappa_{\hat{b}}-{
\rho} +L_{\bar{b}}+L_{\sigma}^2$,
which implies that for all $m\in \sN_0$, 
\begin{align*}
 |\phi^{m}_t(0)-c^{0}_t|
&\le |\phi_t^0(0) - c^{0}_t|
+
{2C_{\bar{b}}\frac{\mu + L_{fa}}{(\mu+\nu) L_{fa} + \mu\nu}}
\sup_{m\in \sN_0}|Y^{t,0,\phi^m}_t-Y^{t,0,\phi^0}_t|
\\
&\le |\phi_t^0(0)| +| c^{0}_t|
+{4C_YC_{\bar{b}}\frac{\mu + L_{fa}}{(\mu+\nu) L_{fa} + \mu\nu}}.
\end{align*}
By \eqref{eq:phi_star_u}, 
for all $t\in [0,T]$,
$|c_t^0|\le  \tfrac{1}{{ \mu+\nu}}C_{\bar{b}} C_Y+
 \tfrac{2}{{ \mu+\nu}}(C_{fa}+L_{fa}|a_0|+
|z^{a_0}|)+|a_0|$,
from which for all $m\in \sN_0$ and $t\in [0,T]$,
\begin{align*}
 |\phi^{m}_t(0)|
&\le |\phi_t^0(0)| +2| c^{0}_t|
+{4C_YC_{\bar{b}}\frac{\mu + L_{fa}}{(\mu+\nu) L_{fa} + \mu\nu}}
\\
&\le \sup_{t\in [0,T]}|\phi_t^0(0)|
+2\Big( \tfrac{1}{{ \mu+\nu}}C_{\bar{b}} C_Y+
 \tfrac{2}{{ \mu+\nu}}(C_{fa}+L_{fa}|a_0|+
|z^{a_0}|)+|a_0|
\Big)
+{4C_YC_{\bar{b}}\tfrac{\mu + L_{fa}}{(\mu+\nu) L_{fa} + \mu\nu}}.
\end{align*}
This finishes the proof of the uniform boundedness of $\phi^m_t(0)$.
\end{proof}

\subsection{Uniform Lipschitz continuity 
in space
}

This section proves that 
the iterates  $(\phi^m)_{m\in \sN_0}$ 
satisfy $\sup_{m\in \sN_0}[\phi^m]_1<\infty$
if one of the conditions \ref{item:T_small}-\ref{item:small_k} 
 holds. 
The following proposition estimates the Lipschitz continuity of the function $x\mapsto Y^{t,x,\phi}_t$
in terms of the Lipschitz continuity of a given feedback control $\phi\in \cV_\bA$.

\begin{Proposition}\label{prop:Y_stability}
Suppose (H.\ref{assum:pgm}) holds.
For each 
  $\phi\in \cV_\bA$ and 
$(t,x)\in [0,T]\t \sR^n$,
let 
$(Y^{t,x,\phi},Z^{t,x,\phi})\in \cS^2(t,T;\sR^n)\t \cH^2(t,T;\sR^{n\t d})$ 
be defined by 
\eqref{bsde_feedback}.
Then 
there exists a constant 
$C\ge 0$  such that
for all
$\phi\in \cV_\bA$,
$t\in [0,T], x,x'\in \sR^n$,
\begin{align}
\label{eq:Y_Lipschitz_estimate_statement}
 \begin{split}
|Y^{t,x,\phi}_t-Y^{t,x',\phi}_t|
&\le 
 L_Y([\phi]_1)|x-x'|,
\end{split}
 \end{align}
 where 
 for each $M\ge 0 $,
the constant $L_Y(M)\ge 0$ is  defined by
\begin{align}
\label{eq:L_Y_phi}
\begin{split}
 L_Y(M) &\coloneqq  C\bigg[
L_g
e^{ 
\big(2L_{\bar{b}}M
+ 2\kappa_{\hat{b}} + C \big)_+
T}
 e^{\alpha_+ T} +\bigg(
\Big(
(1+
L_{\bar{b}}M)
(C_g
+C_{fx}T)
e^{(\kappa_{\hat{b}}-{
\rho}+C )_+T}
\\
&\q 
+L_{fx}(1+M)
 \Big)
 \tfrac{e^{\a T} -1}{\a}
 +
\sqrt{T}
\Big(
e^{\a_+ T}C_g
+C_{fx}
\tfrac{e^{\a T} -1}{\a}
\Big)\bigg)
e^{
\big(2L_{\bar{b}}M
+ 2\kappa_{\hat{b}} + C \big)_+T
}
\bigg],
\\
\alpha&\coloneqq 
\kappa_{\hat{b}} - {
\rho }+L_{\bar{b}}+L_{\sigma}^2.
\end{split}
\end{align}

\end{Proposition}

\begin{proof}
Let 
$\bar{f}^1,\bar{f}^2:[t,T]\t \Om \t 
 \sR^n  \t \sR^{n\t d}
\to \sR^n$ be such that
for all $(s,\om,y,z)\in [t,T]\t \Om \t 
 \sR^n  \t \sR^{n\t d}$,
 \begin{align*}
   \bar{f}^1_s(\om,y,z)
   &= \partial_x H_s(X^{t,x,\phi}_s(\om), \phi_s(X^{t,x,\phi}_s(\om)),y,z),
   \\
   \bar{f}^2_s(\om,y,z)
   &= \partial_x H_s(X^{t,x',\phi}_s(\om), \phi_s(X^{t,x',\phi}_s(\om)),y,z),
 \end{align*}
where $H$ is defined by \eqref{eq:Hamiltonian},
and for each   $\phi\in \cV_\bA$,
$X^{t,x,\phi}\in \cS^2(t,T;\sR^n)$
is defined by 
\eqref{sde_feedback}.
By using 
\eqref{eq:bar_f_1} and 
 applying Lemma \ref{BSDE:apriori}
with $p=1$, $f^1=\bar{f}^1$,  $\xi^1=\p_x g(X^{t,x,\phi}_T)$, $f^2=\bar{f}^2$, $\xi^2=\p_x g(X^{t,x',\phi}_T)$,
$(Y^2,Z^2)=(Y^{t,x',\phi},Z^{t,x',\phi})$
and  $\eps=1/2$,
%
there exists a absolute constant  $C\ge 0$ such that
 \begin{align}
 \label{eq:Y_lipschitz}
 \begin{split}
&\sE\bigg[\sup_{s\in [t,T]}e^{\widetilde{\a} s}
|Y^{t,x,\phi}_s-Y^{t,x',\phi}_s|^2
+\int_t^Te^{\widetilde{\a} s}
|Z^{t,x,\phi}_s-Z^{t,x',\phi}_s|^2
\, \d s\bigg]
\\
&\le
C\sE\bigg[ 
e^{\widetilde{\a} T}
|\p_x g(X^{t,x,\phi}_T)-\p_x g(X^{t,x',\phi}_T)|^2
\\
&\q 
+\bigg(\int_t^Te^{\frac{\widetilde{\a}}{2}s}
|\bar{f}^1_s(\cdot, Y^{t,x',\phi}_s,Z^{t,x',\phi}_s)-\bar{f}^2_s(\cdot, Y^{t,x',\phi}_s,Z^{t,x',\phi}_s)|\,\d s\bigg)^2\bigg].
\end{split}
\end{align}
where we defined $\widetilde{\alpha} \coloneqq 2(\kappa_{\hat{b}}-{
\rho}+L_{\bar{b}}+L_{\sigma}^2)$
above and hereafter. 
Recall that
in the subsequent analysis,
$C$ denotes a generic constant independent of 
$T,\rho,\kappa_{\hat{b}}, C_{\bar{b}},  C_{fx}, L_{fx}, \mu, \nu, L_{fa}, C_g, L_g$.

We now estimate the two terms on the right-hand side of 
\eqref{eq:Y_lipschitz}.
Let $b^1,b^2:[0,T]\t \sR^n\to \sR^n$ be such that 
for all $(t,x)\in [0,T]\t \sR^n$,
$b^1_t(x)=b^2_t(x)=b_t(x,\phi_t(x))$,
and let $\Delta X^{t,x,x'}=X^{t,x,\phi} - {X}^{t,x',\phi}$.
Then by \eqref{eq:hat_b_lipschitz}
and
\eqref{eq:bar_b_lipschitz},
 for all $t\in [0,T]$ and $x,x'\in \sR^n$,
 \begin{align*}
 \la x-x',b^1_t(x)-b^1_t(x')\ra 
 & \le
 \la x-x',\hat{b}_t(x)-\hat{b}_t(x')+
 \bar{b}_t(x)\phi_t(x)- \bar{b}_t(x')\phi_t(x')\ra 
 \\
&  \le
  \kappa_{\hat{b}} |x-x'|^2
  +L_{\bar{b}}|x-x'|(|x-x'|+|\phi_t(x)-\phi_t(x')|)
  \\
  &\le 
\big(  \kappa_{\hat{b}} 
  +L_{\bar{b}}(1+[\phi]_1)\big)|x-x'|^2.
   \end{align*}
By  Lemma \ref{forward:apriori_p} for $p \geq 2$ and 
\eqref{eq:sigma},
 for all $x,x'\in \sR^n$,
\begin{align}
\label{eq:X_lipschitz_phi}
\|\Delta X^{t,x,x'}\|_{\cS^p }
& \le 
C_{(p)}e^{T\big(2(
\kappa_{\hat{b}} 
  +L_{\bar{b}}(1+[\phi]_1))
+C_{(p)} L_{\sigma}^2\big)_{+}}
|x-x'|,
\end{align}
from which by setting $p=2$ and using \eqref{eq:g_bound}, we obtain 
\begin{align}
 \label{eq:Y_lipschitz_term1}
\begin{split}
&
e^{-\frac{\widetilde{\a}}{2} t}
\sE\Big[ e^{\widetilde{\a} T}
|\p_x g(X^{t,x,\phi}_T)-\p_x g(X^{t,x',\phi}_T)|^2
\Big]^{1/2}
\le 
C
L_g 
e^{ 
(T-t)\frac{\widetilde{\a}}{2}+ T\big(
2(\kappa_{\hat{b}} 
  +L_{\bar{b}}[\phi]_1)
  +C
\big)_{+}}
|x-x'|.
\end{split}
\end{align}

We then proceed to estimate the second term in 
\eqref{eq:Y_lipschitz}.
Note that for all 
$t\in [0,T]$, $x,x'\in \sR^n$,
$\phi\in \cV_\bA$,
$(y,z)\in \sR^n\t \sR^{n\t d}$,
by 
Lemma \ref{lemma:H_gradient_estimate},
\begin{align*}
&|\partial_x H_t(x, \phi_t(x),y,z)
-\partial_x H_t(x', \phi_t(x'),y,z)|
\\
&\le 
\big(L_{\hat{b}}|x-x'|+
L_{\bar{b}}(|x-x'|+|\phi_t(x)-\phi_t(x')|)\big)| y| + L_{\sigma}|x-x'| |z|
+L_{fx}(|x-x'|+|\phi_t(x)-\phi_t(x')|)
\\
&\le 
\Big(
\big(L_{\hat{b}}+
L_{\bar{b}}(1+[\phi]_1)\big)  | y|
+L_{fx}(1+[\phi]_1) + L_{\sigma} |z|
\Big)|x-x'|,
\end{align*}
which 
along with the Cauchy-Schwarz inequality
implies that 
\begin{align}
 \begin{split}
&
\sE\bigg[ \bigg(\int_t^Te^{\frac{\widetilde{\a}}{2}s}
|\bar{f}^1_s(\cdot, Y^{t,x',\phi}_s,Z^{t,x',\phi}_s)-\bar{f}^2_s(\cdot, Y^{t,x',\phi}_s,Z^{t,x',\phi}_s)|\,\d s\bigg)^2\bigg]^{1/2}
\\
&\le
\sE\bigg[ \bigg(\int_t^Te^{\frac{\widetilde{\a}}{2}s}
\Big(
(C+
L_{\bar{b}}[\phi]_1) |Y^{t,x',\phi}_s|
+L_{fx}(1+[\phi]_1) + L_{\sigma} |Z^{t,x',\phi}_s|
\Big)|\Delta X^{t,x,x'}_s|
\,\d s\bigg)^2\bigg]^{1/2}
\\
&\le
\sE\bigg[ \bigg(\int_t^Te^{\frac{\widetilde{\a}}{2}s}
\Big(
(C+
L_{\bar{b}}[\phi]_1) |Y^{t,x',\phi}_s|
+L_{fx}(1+[\phi]_1)
+C |Z^{t,x',\phi}_s| \Big)\d s
\bigg)^4
\bigg]^{1/4} 
 \|\Delta X^{t,x,x'}\|_{\cS^4 }
\\
&\le
\bigg(
\sE\bigg[ \bigg(\int_t^Te^{\frac{\widetilde{\a}}{2}s}
\Big(
(C+
L_{\bar{b}}[\phi]_1) |Y^{t,x',\phi}_s|
+L_{fx}(1+[\phi]_1)
 \Big)\d s
\bigg)^4
\bigg]^{1/4} 
\\
&\q 
+C
\sE\bigg[ \bigg(\int_t^Te^{\f{\widetilde{\a}}{2} s}
|Z^{t,x',\phi}_s| \,\d s
\bigg)^4
\bigg]^{1/4}\bigg)
 \|\Delta X^{t,x,x'}\|_{\cS^4 }.
\label{eq:Y_lipschitz_term2}
\end{split}
\end{align}
By
 Proposition \ref{prop:Y_bound}, 
 there exists an absolute constant  $C\ge 0$ such that 
for all $(t,x')\in [0,T]\t \sR^n$,
\begin{align}\label{eq:C^Y_s}
|Y^{t,x',\phi}_t|
\le
 C_Y\coloneqq
C 
(C_g
+C_{fx}T)
e^{(\kappa_{\hat{b}} - {
\rho } +C )_+T}.
\end{align}
The Markovian property of $Y^{t,x',\phi}$ implies that 
$Y^{t,x',\phi}_s=Y^{s,X^{t,x',\phi}_s,\phi}_s$
(see e.g., \cite[Theorem 5.1.3]{zhang2017backward}),
which subsequently shows that 
\begin{align}
\label{eq:Y_lipschitz_Y_estimate}
 \begin{split}
 &
\sE\bigg[ \bigg(\int_t^Te^{\frac{\widetilde{\a}}{2}s}
\Big(
(C+
L_{\bar{b}}[\phi]_1) |Y^{t,x',\phi}_s|
+L_{fx}(1+[\phi]_1)
 \Big)\d s
\bigg)^4
\bigg]^{1/4} 
\\
&\le 
\Big(
(C+
L_{\bar{b}}[\phi]_1) C_Y
+L_{fx}(1+[\phi]_1)
 \Big)\int_t^Te^{\frac{\widetilde{\a}}{2}s}\,\d s.
\end{split}
\end{align}
On the other hand,
by
the Cauchy-Schwarz inequality
and 
 Proposition \ref{prop:Y_bound}
 with $p=2$,
there exists
an absolute constant 
$C\ge 0$ such that 
\begin{align}\l{eq:Y_lipschitz_Z_estimate}
 \begin{split}
\sE\bigg[ \bigg(\int_t^T e^{\f{\widetilde{\a}}{2} s}
|Z^{t,x',\phi}_s| \,\d s
\bigg)^4
\bigg]^{1/4}
&\le
\sqrt{T-t}
\sE\bigg[
\bigg(\int_t^Te^{\widetilde{\a} s}
|Z^{t,x',\phi}_s|^2 \,\d s
\bigg)^2
\bigg]^{1/4} 
\\
&\le
C\sqrt{T-t}
\bigg(
e^{\frac{\widetilde{\a} }{2}T}C_g
+C_{fx}
\int_t^Te^{\frac{\widetilde{\a}}{2} s}
\,\d s
\bigg).
\end{split}
\end{align}
Combining 
\eqref{eq:X_lipschitz_phi} (with $p=4$),
\eqref{eq:Y_lipschitz_term2},
\eqref{eq:Y_lipschitz_Y_estimate},
and 
\eqref{eq:Y_lipschitz_Z_estimate}
gives  that 
\begin{align*}
 \begin{split}
&
\sE\bigg[ \bigg(\int_t^Te^{\frac{\widetilde{\a}}{2}s}
|\bar{f}^1_s(\cdot, Y^{t,x',\phi}_s,Z^{t,x',\phi}_s)-\bar{f}^2_s(\cdot, Y^{t,x',\phi}_s,Z^{t,x',\phi}_s)|\,\d s\bigg)^2\bigg]^{1/2}
\\
&\le
C
\bigg(
\Big(
(1+
L_{\bar{b}}[\phi]_1) C_Y
+L_{fx}(1+[\phi]_1)
 \Big)\int_t^Te^{\frac{\widetilde{\a}}{2}s}\,\d s
\\
&\q 
+
\sqrt{T}
\bigg(
e^{\frac{\widetilde{\a} }{2}T}C_g
+C_{fx}
\int_t^Te^{\frac{\widetilde{\a}}{2} s}
\,\d s
\bigg)\bigg)
e^{T\big(2(
\kappa_{\hat{b}} 
  +L_{\bar{b}}[\phi]_1)
+C\big)_{+}}
|x-x'|,
\end{split}
\end{align*}
with $C_Y$ defined in \eqref{eq:C^Y_s}.
Consequently, by using \eqref{eq:Y_lipschitz}, and \eqref{eq:Y_lipschitz_term1}
and the identity
that
$e^{-\f{\widetilde{\a}}{2} t}
\int_t^T e^{\frac{\widetilde{\a}}{2} s}
\,\d s 
=\f{2}{\widetilde{\a}}
(e^{\frac{\widetilde{\a}}{2} (T-t)} -1)
$, 
\begin{align*}
 \begin{split}
|Y^{t,x,\phi}_t-Y^{t,x',\phi}_t|
&\le 
C\bigg[
L_g
e^{ (T-t) \tfrac{\widetilde{\alpha}}{2} +
T\big(
2(\kappa_{\hat{b}} 
  +L_{\bar{b}}[\phi]_1)
  +C\big)_{+}}
\\
&\q 
+\bigg(
\Big(
(1+
L_{\bar{b}}[\phi]_1)
(C_g
+C_{fx}T)
e^{(\kappa_{\hat{b}} -{
\rho} + C )_+T}
+L_{fx}(1+[\phi]_1)
 \Big)\f{2}{\widetilde{\a}}
(e^{\frac{\widetilde{\a}}{2} (T-t)} -1)
\\
&\q 
+
\sqrt{T}
\Big(
e^{\frac{\widetilde{\a} }{2}(T-t)}C_g
+C_{fx}
\f{2}{\widetilde{\a}}
(e^{\frac{\widetilde{\a}}{2} (T-t)} -1)
\Big)\bigg)
e^{T\big(2(
\kappa_{\hat{b}} 
  +L_{\bar{b}}[\phi]_1)
+C\big)_{+}}
\bigg]|x-x'|.
\end{split}
\end{align*}
Then the facts that 
for all $\widetilde{\a}\in \sR$,  $[0,T]\ni t\mapsto \f{2}{\widetilde{\a}}
(e^{\frac{\widetilde{\a}}{2} (T-t)} -1)\in (0,\infty)$ is maximized at $t=0$, 
and $\widetilde{\alpha}=2(\kappa_{\hat{b}} - {
\rho}+L_{\bar{b}}+L_{\sigma}^2)$  
lead to the desired Lipschitz estimate uniformly in $t$. 
\end{proof}

With Proposition \ref{prop:Y_stability} in hand, 
we   prove that under suitable assumptions,  for any initial guess $\phi^0\in \cV_\bA$,
the sequence of feedback controls $(\phi^m)_{m\in\sN_0}$
generated by 
\eqref{eq:phi_update}
is uniformly Lipschitz continuous. 
For notational simplicity, 
let $C>0$ be a  constant such that 
\eqref{eq:Y_uniform_bdd} 
and 
\eqref{eq:Y_Lipschitz_estimate_statement}
hold,
  $C_Y$ and 
  $\alpha$ be  defined
in 
 \eqref{eq:Y_uniform_bdd} and 
\eqref{eq:L_Y_phi},
respectively,
and for each 
 $\phi^0\in \cV_\bA$,
 define 
\begin{align}
\label{eq:constants_lipschitz_statement}
\begin{split}
   A_1 & \coloneqq  C_{\bar{b}} C 
   \bigg(
  (L_g + \sqrt{T}
 C_g) e^{\a_+ T}
+ \tfrac{e^{\a T}-1}{\a}
\Big(
(C_g
+C_{fx}T)
e^{\a_+ T}
 +L_{fx} + \sqrt{T} C_{fx}\Big)
\bigg)
\\
A_2 & \coloneqq C_{\bar{b}} C\tfrac{e^{\a T}-1}{\a}
\Big(
(C_g
+C_{fx}T)
e^{\a_+ T} L_{\bar{b}}+L_{fx}
\Big),
 \\
\mu_0 &  \coloneqq { \frac{1}{4} \left( \frac{\mu L_{fa}}{\mu + L_{fa}} +\nu \right),} 
     \q 
K  \coloneqq 
\max\left\{2TL_{\bar{b}}[\phi^0]_1,
\frac{  2 T L_{\bar{b}}   A_1   +2T L_{\bar{b}} (L_{\bar{b}}C_Y + L_{fa})}{\mu_0}+1
\right\}.
\end{split}
\end{align}

\begin{Theorem}\label{thm:lipschitz_iterates}
Suppose (H.\ref{assum:pgm}) holds.
For each
 $\phi^0\in \cV_\bA$,
 $\tau >0$
 and 
$m\in \mathbb{N}$,
let $\phi^m$ be defined by \eqref{eq:phi_update}
with the initial guess $\phi^0$ and stepsize $\tau$.
Let $C\ge 0$ be a  constant such that 
\eqref{eq:Y_uniform_bdd} and 
\eqref{eq:Y_Lipschitz_estimate_statement}
hold,
let 
 $C_Y \ge 0$ be
  defined  in \eqref{eq:Y_uniform_bdd},
 let $\alpha \in \sR$ be  defined  in  
\eqref{eq:L_Y_phi},
and let $A_1, A_2,
\mu_0, K\ge 0$
be defined  in 
\eqref{eq:constants_lipschitz_statement}.
Then 
for all 
 $\phi^0 \in \cV_\bA$
 satisfying
\begin{align}
\label{eq:condition_lipschitz}
      & 2TL_{\bar{b}}{A}_1   e^{ (2\kappa_{\hat{b}} + C) T + K } 
      \leq \mu_0,
      \q
      \textnormal{and}\q
     A_2 ( e^{(2\kappa_{\hat{b}} + C) T + K } +1 ) \leq \mu_0,
\end{align}
and 
for all 
  { $\tau \in (0, \frac{2}{\mu + L_{fa}}\wedge \frac{1}{\nu}]$}
and
    $m\in \sN_0$,
   \begin{align}
   \label{eq:lipschitz_iterates_constant}
\begin{split}
[{\phi^{m}}]_1
& \leq
L_{(\phi^0)}\coloneqq 
 [{\phi^{0}}]_1  + 
 { \frac{1}{\mu_0}
 \Big(
 L_{\bar{b}}
 C_Y + L_{fa} 
+ {A}_1
\big(
e^{(2\kappa_{\hat{b}} +C) T +  K}
+ 1\big) \Big)}.  
 \end{split}
\end{align}

\end{Theorem}

\begin{proof}

Throughout this proof, let
$\phi^0\in \cV_\bA$ and { $\tau \in (0, \frac{2}{\mu + L_{fa}}\wedge \frac{1}{\nu}]$} be  fixed. 
Suppose that  $\phi^m\in \cV_\bA$
for some $m\in \sN_0$.
For all $(t,x,x') \in [0,T] \times \mathbb{R}^n \times \mathbb{R}^n$,
by applying Lemma \ref{Lemma_iteration_1} with $a=\phi_t^{m}(x)$, $a'=\phi_t^{m}(x')$, $y=Y^{t,x,\phi^m}_t$ and $y'=Y^{t,x',\phi^m}_t$, 
\begin{align*}
|\phi_t^{m+1}(x) - \phi^{m+1}_t(x')| 
&\leq \left(
1 - { \tau\frac{1}{2}
\left(\frac{\mu L_{fa}}{\mu + L_{fa}}
+\nu
\right)}
\right)
|\phi_t^m(x) - \phi_t^{m}(x')| + \tau C_{\bar{b}}  |Y^{t,x,\phi^m}_t -Y^{t,x',\phi^m}_t| \\
&\quad + \tau(L_{\bar{b}}|Y^{t,x',\phi^m}_t| + L_{fa})|x-x'|.
\end{align*}
By using the Lipschitz continuity of $\phi^{m}$,  and the
Lipschitz continuity and uniform boundedness of 
the mapping 
$(t,x)\mapsto  Y^{t,x,\phi^m}_t$ (see Propositions
\ref{prop:Y_bound}
and  \ref{prop:Y_stability}), we further deduce
\begin{align*}
    |\phi_t^{m+1}(x) - \phi^{m+1}_t(x')| & \leq \left( [\phi^{m}]_1 \left(
1 -  { \tau\frac{1}{2}
\left(\frac{\mu L_{fa}}{\mu + L_{fa}}
+\nu
\right)}
\right) + 
    \tau C_{\bar{b}} L_Y([\phi^m]_1)
    +\tau(L_{\bar{b}}C_Y + L_{fa}) \right)|x-x'|,
\end{align*}
where 
$C_Y$
is defined by \eqref{eq:Y_uniform_bdd}
and 
 $L_Y([\phi^m]_1)$
is defined by 
\eqref{eq:L_Y_phi}.
Consequently, we have 
\begin{align}
\label{eq:Lipschitz_induction}
\begin{split}
  [\phi^{m+1}]_1 
& \leq    [\phi^{m}]_1 \left(1- { \tau\frac{1}{2}
\left(\frac{\mu L_{fa}}{\mu + L_{fa}}
+\nu
\right)}\right) +  \tau(L_{\bar{b}}C_Y + L_{fa})  + \tau C_{\bar{b}} L_Y([\phi^m]_1).
   \end{split}
\end{align}

 In the sequel, we aim to establish a uniform bound of 
 $([\phi^m]_1)_{m\in \sN_0}$
 based on \eqref{eq:Lipschitz_induction}.
 Observe from the definition of $L_Y([\phi^m]_1)$
in 
\eqref{eq:L_Y_phi}
 that 
\begin{align}
 L_Y([\phi^m]_1) &\coloneqq  C\bigg[
L_g
e^{ 
\big(2L_{\bar{b}}[\phi^m]_1
+ 2\kappa_{\hat{b}} + C \big)_+
T}
 e^{\alpha_+ T} +\bigg(
\Big(
(1+
L_{\bar{b}}[\phi^m]_1)
(C_g
+C_{fx}T)
e^{(\kappa_{\hat{b}}-{
\rho}+C )_+T}
\nb
\\
&\q 
+L_{fx}(1+[\phi^m]_1)
 \Big)
 \tfrac{e^{\a T} -1}{\a}
 +
\sqrt{T}
\Big(
e^{\a_+ T}C_g
+C_{fx}
\tfrac{e^{\a T} -1}{\a}
\Big)\bigg)
e^{
\big(2L_{\bar{b}}[\phi^m]_1
+ 2\kappa_{\hat{b}} + C \big)_+T
}
\bigg]
\nb
\\
&=
C\bigg[
\bigg(
\big(
L_g
+
\sqrt{T}
C_g
\big)e^{\a_+ T}
+
\Big(
(C_g
+C_{fx}T)
e^{(\kappa_{\hat{b}}-{
\rho}+C )_+T}
+L_{fx}
+\sqrt{T}
 C_{fx}
\Big)
\tfrac{e^{\a T} -1}{\a}
\bigg)
\nb
\\
&\q \times
\Big(
e^{
(2L_{\bar{b}}[\phi^m]_1
+ 2\kappa_{\hat{b}}+C) T
}
+ 1\Big)
\nb
\\
&\q +
\Big(
(C_g
+C_{fx}T)
e^{(\kappa_{\hat{b}}-{
\rho} +C )_+T}
L_{\bar{b}}
+L_{fx}
\Big)
\tfrac{e^{\a T} -1}{\a}
[\phi]_1
\Big(
e^{(2L_{\bar{b}}[\phi^m]_1
+ 2\kappa_{\hat{b}} +C)T}
+ 1\Big)
\bigg].
\label{eq:LC_Y}
\end{align}
Let 
$A_1, A_2, \mu_0, K$ be defined as in 
\eqref{eq:constants_lipschitz_statement}.
Then by writing 
$ \widetilde{A}_1\coloneqq  2 T L_{\bar{b}}   A_1$
and 
$
[\widetilde{\phi^{m}}]_1
  \coloneqq 2 T L_{\bar{b}} [\phi^{m}]_1$ for all $m\in \sN_0$,
   multiplying  both sides of 
\eqref{eq:Lipschitz_induction} by $2T L_{\bar{b}}$
and using \eqref{eq:LC_Y}, we have
\begin{align}
\label{eq:Lipschitz_induction_tilde}
\begin{split}
[\widetilde{\phi^{m+1}}]_1
& \leq   [\widetilde{\phi^{m}}]_1 \left(1-2\mu_0 \tau\right) + 2 \tau
T L_{\bar{b}}(L_{\bar{b}}C_Y + L_{fa})  \\
&\q 
+ \tau 
 \Big( \widetilde{A}_1
\big(
e^{(2\kappa_{\hat{b}} + C) T +  [\widetilde{\phi^{m}}]_1}
+ 1\big) +
A_2[\widetilde{\phi^{m}}]_1
\big(e^{(2\kappa_{\hat{b}} + C) T
+  [\widetilde{\phi^{m}}]_1}
+ 1\big)\Big).
   \end{split}
\end{align}
Now we prove by induction  that 
 $\sup_{m \in \sN_0 }[\widetilde{\phi^{m}}]_1 \leq K$ under the conditions that 
 \begin{align}
 \label{eq:assum_lipschitz}
      & \widetilde{A}_1   e^{ (2\kappa_{\hat{b}} +C) T+ K } 
      \leq \mu_0,
      \q
      \textnormal{and}\q
     A_2 ( e^{ (2\kappa_{\hat{b}} +C)  T + K } +1 ) \leq \mu_0.
\end{align}
 The statement holds  for $m=0$ by the definition of $K$. 
 Suppose that $[\widetilde{\phi^{m}}]_1 \leq K$ for some 
 $m\in \sN_0$.
Then by 
the induction hypothesis, 
\eqref{eq:Lipschitz_induction_tilde} and 
\eqref{eq:assum_lipschitz},
\begin{align}
\begin{split}
[\widetilde{\phi^{m+1}}]_1
& \leq   [\widetilde{\phi^{m}}]_1 \left(1-2\mu_0 \tau\right) + 2 \tau
T L_{\bar{b}}(L_{\bar{b}}C_Y + L_{fa})  
+ \tau 
 \Big( \widetilde{A}_1
+
\mu_0 +\mu_0 [\widetilde{\phi^{m}}]_1
\Big)
\\
&=
  [\widetilde{\phi^{m}}]_1 \left(1-\mu_0 \tau\right) + 
  \tau
  \Big( 2 
T L_{\bar{b}}(L_{\bar{b}}C_Y + L_{fa})  
+ \widetilde{A}_1
+
 \mu_0
\Big)
\\
&\le 
K \left(1-\mu_0 \tau\right) + 
  \tau \mu_0 K
  \le K.
   \end{split}
\end{align}
This finishes the proof of the fact 
 $\sup_{m \in \sN_0 }[\widetilde{\phi^{m}}]_1 \leq K$. 
Substituting  this a-priori bound into  \eqref{eq:Lipschitz_induction_tilde} 
and using \eqref{eq:assum_lipschitz}
give that
 \begin{align*}
\begin{split}
[\widetilde{\phi^{m+1}}]_1
& \leq   [\widetilde{\phi^{m}}]_1 \left(1-2\mu_0 \tau\right) + 2 \tau
T L_{\bar{b}}(L_{\bar{b}}C_Y + L_{fa}) 
+ \tau 
 \Big( \widetilde{A}_1
\big(
e^{ (2\kappa_{\hat{b}} +C)  T + K }
+ 1\big)  \\
& \qquad +
A_2[\widetilde{\phi^{m}}]_1
\big(e^{(2\kappa_{\hat{b}} +C)  T + K  }
+ 1\big)\Big)
\\
& \leq
 [\widetilde{\phi^{m}}]_1 \left(1-2\mu_0 \tau\right) + 2 \tau
T L_{\bar{b}}(L_{\bar{b}}C_Y + L_{fa})  
+ \tau 
 \Big( \widetilde{A}_1
\big(
e^{(2\kappa_{\hat{b}} +C)  T + K }
+ 1\big) 
+\mu_0 [\widetilde{\phi^{m}}]_1 \Big)
\\
& \leq
 [\widetilde{\phi^{m}}]_1 \left(1-\mu_0 \tau\right) + 
  \tau 
 \Big(
 2 T L_{\bar{b}}(L_{\bar{b}}C_Y + L_{fa})   
+  \widetilde{A}_1
\big(
e^{(2\kappa_{\hat{b}} +C)  T + K }
+ 1\big) \Big),
   \end{split}
\end{align*}
 from which one can deduce that 
 for all $m\in \sN_0$,
  \begin{align*}
\begin{split}
[\widetilde{\phi^{m}}]_1
& \leq
 [\widetilde{\phi^{0}}]_1  + 
  \frac{1}{\mu_0}
 \Big(
 2 T L_{\bar{b}}(L_{\bar{b}}C_Y + L_{fa})   
+ \widetilde{A}_1
\big(
e^{ (2\kappa_{\hat{b}} +C)  T + K }
+ 1\big) \Big).
   \end{split}
\end{align*}
Dividing both sides of the above inequality by $2TL_{\bar{b}}$ shows 
  \begin{align*}
\begin{split}
[{\phi^{m}}]_1
& \leq
 [{\phi^{0}}]_1  + 
 { \frac{1}{\mu_0}
 \Big(
 L_{\bar{b}}
 C_Y + L_{fa} 
+ {A}_1
\big(
e^{ (2\kappa_{\hat{b}} +C)  T + K  }
+ 1\big) \Big)}.
   \end{split}
\end{align*}
with 
constants 
 $A_1,K,\mu_0$ defined as in \eqref{eq:constants_lipschitz_statement}.
\end{proof}

\subsection{Contraction   
in a weighted sup-norm}

Based on the uniform Lipschitz continuity of 
$(\phi^m)_{m\in \sN_0}$ in Theorem \ref{thm:lipschitz_iterates},
we prove 
the contractivity of 
the iterates  $(\phi^m)_{m\in \sN_0}$ with respect to the  weighted sup-norm norm
$|\cdot|_0$ (see Definition \ref{def:fb}).

The following proposition estimates  the Lipschitz stability of the adjoint process $ Y^{t,x,\phi}$
with respect to the  feedback control $\phi\in \cV_\bA$. 

\begin{Proposition}\label{prop:Y_stability2}
Suppose (H.\ref{assum:pgm}) holds.
For each
 $\phi\in \cV_\bA$ and 
$(t,x)\in [0,T]\t \sR^n$,
let $(Y^{t,x,\phi},Z^{t,x,\phi})\in \cS^2(t,T;\sR^n)\t \cH^2(t,T;\sR^{n\t d})$ 
be defined by 
\eqref{bsde_feedback}. Suppose that for all $\phi' \in \cV_\bA$ and $(t,s,x)\in [0,T] \t [t,T] \t  \sR^n$, it holds with some constant  $C_Z^{\phi'}\ge 0$  that 
$|Z_s^{t,x,\phi'}|\le C_Z^{\phi'}$
for $\d t \otimes \d \sP$-a.e.
Then 
there exists a  constant 
$C\ge 0$  such that
for all 
 $\phi, \phi' \in \cV_\bA$ and
$(t,x)\in [0,T]\t  \sR^n$,
\begin{align}\label{eq:Y_difference_phi_statement}
\begin{split}
&|Y^{t,x,\phi}_t-Y^{t,x,\phi'}_t| \le B[\phi,\phi',C_Z^{\phi'}]
(1+|x|)|\phi-\phi'|_0,
  \end{split}
\end{align}
where the constant $B[\phi,\phi',C_Z^{\phi'}]$ is defined by 
\begin{align}\label{eq:Y_difference_phi_constant}
\begin{split}
B[\phi,\phi',C_Z^{\phi'}]
& \coloneqq 
C C_{\bar{b}} \Big(1+ {T}
+ T  C_{\bar{b}} \sup_{t\in [0,T]}|\phi'_t(0)|
\Big)     e^{T
\beta_{+}}
\bigg( L_g
{
\mathfrak{m}^{1/2}_{(\alpha, \beta)}}
\\
&\q + \tfrac{e^{T \a } -1}{{\a}}
\Big[\Big(
( C_Y
+L_{fx})(1+[\phi]_1)
+
C_Z^{\phi'} 
\Big)
T
e^{T
\beta_{+}}+L_{fx} +  C_Y 
 \Big]\bigg),
\\
\beta
&\coloneqq
2\kappa_{\hat{b}} 
  + 2L_{\bar{b}}\max \{ [\phi]_1,[\phi']_1 \}+C, \quad { 
  \mathfrak{m}_{(\alpha, \beta)} \coloneqq \sup_{ t \in [0,T]} e^{2 \alpha (T-t)}    \int_{t}^{T} e^{(T-s)\b} \, \mathrm{d}s,}
  \end{split}
\end{align}
with  $C_Y$ and $\alpha $
 defined as in 
\eqref{eq:Y_uniform_bdd}
and 
 \eqref{eq:L_Y_phi},
 respectively. 
\end{Proposition}

\begin{proof}
Let 
$\bar{f}^1,\bar{f}^2:[t,T]\t \Om \t 
 \sR^n  \t \sR^{n\t d}
\to \sR^n$ be such that
for all $(s,\om,y,z)\in [t,T]\t \Om \t 
 \sR^n  \t \sR^{n\t d}$,
 \begin{align*}
 \bar{f}^1_s(\om,y,z)&= \partial_x H_s(X^{t,x,\phi}_s(\om), \phi_s(X^{t,x,\phi}_s(\om)),y,z),
 \\
 \bar{f}^2_s(\om,y,z) &= \partial_x H_s(X^{t,x,\phi'}_s(\om), \phi'_s(X^{t,x,\phi'}_s(\om)),y,z),
 \end{align*}
where $H$ is defined in \eqref{eq:Hamiltonian},
and 
for each
 $\phi\in \cV_\bA$ and 
$(t,x)\in [0,T]\t \sR^n$,
 $X^{t,x,\phi}\in \cS^2(t,T;\sR^n)$
is defined by 
\eqref{sde_feedback}.
By using 
\eqref{eq:bar_f_1} and 
 applying Lemma \ref{BSDE:apriori}
with $p=1$, $f^1=\bar{f}^1$,  $\xi^1=\p_x g(X^{t,x,\phi}_T)$, $f^2=\bar{f}^2$, $\xi^2=\p_x g(X^{t,x,\phi'}_T)$,
$(Y^2,Z^2)=(Y^{t,x,\phi'},Z^{t,x,\phi'})$
and 
 $\eps=1/2$,
  it holds with  an absolute constant $C\ge 0$ that
 \begin{align}
 \label{eq:Y_growth}
 \begin{split}
&\sE\bigg[\sup_{s\in [t,T]}e^{\widetilde{\a} s}
|Y^{t,x,\phi}_s-Y^{t,x,\phi'}_s|^2
+\int_t^Te^{\widetilde{\a} s}
|Z^{t,x,\phi}_s-Z^{t,x,\phi'}_s|^2
\, \d s\bigg]
\\
&\le
C\sE\bigg[ e^{\widetilde{\a} T}
|\p_x g(X^{t,x,\phi}_T)-\p_x g(X^{t,x,\phi'}_T)|^2
\\
&\q 
+\bigg(\int_t^Te^{\frac{\widetilde{\a}}{2}s}
|\bar{f}^1_s(\cdot, Y^{t,x,\phi'}_s,Z^{t,x,\phi'}_s)-\bar{f}^2_s(\cdot, Y^{t,x,\phi'}_s,Z^{t,x,\phi'}_s)|\,\d s\bigg)^2\bigg],
\end{split}
\end{align}
where we defined $\widetilde{\alpha} \coloneqq 2(\kappa_{\hat{b}} - {
\rho }+L_{\bar{b}}+L_{\sigma}^2)$
above and hereafter.
In the subsequent analysis, we denote by 
$C\ge 1$ a generic constant independent of 
$T,\rho, \kappa_{\hat{b}},
C_{\bar{b}},
C_{fx}, L_{fx}, \mu,\nu, L_{fa}, C_g, L_g$.

To estimate the right-hand side of 
\eqref{eq:Y_growth},
we first quantify the dependence of $X^{t,x,\phi}$
on $\phi$.
For all $(t,x)\in [0,T]\t \sR^n$,
let $\Delta X^{t,x}=X^{t,x,\phi} - {X}^{t,x,\phi'}$.
Similar to \eqref{eq:X_lipschitz_phi},
by Lemma \ref{forward:apriori_p} with $p = 2$,
$b^1_t(x)=b_t(x,\phi_t(x))$ and 
$b^2_t(x)=b_t(x,\phi'_t(x))$,
 for all $(t,x) \in [0,T]\t  \sR^n$, $\phi, \phi' \in \cV_\bA$,
\begin{align}
\label{eq:X_lipschitz_phi2}
\| \Delta X^{t,x}
\|_{\cS^2 }
& \le 
C_{\bar{b}} \sqrt{T} 
M_{([\phi]_1)}
\| \phi({X}^{t,x,\phi'})- \phi'({X}^{t,x,\phi'}) \|_{\cH^2 }
\nb
\\
&\le 
C_{\bar{b}} T
M_{([\phi]_1)}
\| \phi({X}^{t,x,\phi'})- \phi'({X}^{t,x,\phi'}) \|_{\cS^2 },
\end{align}
where the constant $M_{([\phi]_1)}$ is defined by
\bb \label{eq:MT}
M_{([\phi]_1)}
\coloneqq 
C
e^{T\big(2(
\kappa_{\hat{b}} 
  +L_{\bar{b}}(1+[\phi]_1))
+C  L_{\sigma}^2\big)_{+}}
\le 
C 
e^{2T\big(
\kappa_{\hat{b}} 
  +L_{\bar{b}}[\phi]_1
+C \big)_{+}}.
\ee
Moreover, by using the fact that $|\phi(x)-\phi'(x)|\le  |\phi-\phi'|_0(1+|x|)$,
\begin{align}
\label{eq:control_linear}
\| \phi({X}^{t,x,\phi'})- \phi'({X}^{t,x,\phi'}) \|_{\cS^2 }
\le  (1+\|{X}^{t,x,\phi'}\|_{\cS^2})|\phi-\phi'|_0.
\end{align}
We estimate $\|{X}^{t,x,\phi'}\|_{\cS^2}$,
by setting $\|\phi'(0)\|_\infty=\sup_{t\in [0,T]}|\phi'_t(0)|$
and applying
Lemma \ref{forward:apriori_p} with $p = 2$, 
$x_1=x$, 
$b^1_t(x)=b_t(x,\phi'_t(x))$,
$\sigma^1_t(x)=\sigma_t(x)$,
$x_2=0$,
$b^2_t(x)=0$
and $\sigma^2_t(x)=0$,
\begin{align}
\label{eq:X_moment_bound_p}
\begin{split}
  \| {X}^{t,x,\phi'} \|_{\cS^2 }
               &\le 
   M_{([\phi']_1)}
\big(|x|+\sqrt{T}
\|b(0,\phi'(0))\|_{\cH^2}
+\|\sigma(0)\|_{\cH^2}\big)
\\
 &\le
C     M_{([\phi']_1)}
\big(|x|+ T+ {T}C_{\bar{b}} \|\phi'(0)\|_\infty+\sqrt{T}
\big)
\\
&\le 
C     M_{([\phi']_1)}
\big(1+ {T}
+TC_{\bar{b}} \|\phi'(0)\|_\infty
\big)(1+|x|),
\end{split}
\end{align}
which along with 
\eqref{eq:X_lipschitz_phi2},
 \eqref{eq:control_linear}
 and $ M_{([\phi']_1)} \ge 1$
shows   
\begin{align}
\label{eq:X_difference_phi}
\begin{split}
\| \phi({X}^{t,x,\phi'})- \phi'({X}^{t,x,\phi'}) \|_{\cS^2 }
&\le  C     M_{([\phi']_1)}
\big(1+ {T}
+TC_{\bar{b}} (1+|x|) \|\phi'(0)\|_\infty
\big)|\phi-\phi'|_0,
\\
\|\Delta X^{t,x}\|_{\cS^2 }
& \le 
C C_{\bar{b}} T
 M_{([\phi]_1)}
 M_{([\phi']_1)}
\big(1+ {T}
+T C_{\bar{b}} \|\phi'(0)\|_\infty
\big)
(1+|x|) |\phi-\phi'|_0.
\end{split}
\end{align}
Similarly, by setting $\widetilde{\b}=2 (
\kappa_{\hat{b}} 
  +L_{\bar{b}}(1+[\phi]_1)) +L_\sigma^2 +1$
  and using \eqref{eq:X_b_difference}
with  $b^1_t(x)=b_t(x,\phi_t(x))$
and  $b^2_t(x)=b_t(x,\phi'_t(x))$,
   we have
\begin{align}\label{eq:X_phi_terminal}
\begin{split}
    \mathbb{E} \left[ |\Delta X^{t,x}_T |^2 \right]^{\frac{1}{2}} 
   & \le
C  C_{\bar{b}}  \left( \int_{t}^{T} e^{(T-s)\widetilde{\b}} \, \mathrm{d}s \right)^{\frac{1}{2}}  
    \| \phi({X}^{t,x,\phi'})- \phi'({X}^{t,x,\phi'}) \|_{\cS^2 }
    \\& \leq 
    C C_{\bar{b}}
   {
   \left( \int_{t}^{T} e^{(T-s)\b} \, \mathrm{d}s \right)^{\frac{1}{2}}  }
     M_{([\phi']_1)}
\big(1+ {T}
+ TC_{\bar{b}} \|\phi'(0)\|_\infty
\big)
(1+|x|)
 |\phi-\phi'|_0,
\end{split}
\end{align}
with $\beta\coloneqq 
2
(\kappa_{\hat{b}} 
  +L_{\bar{b}}[\phi]_1) +C$,
  where the last inequality used \eqref{eq:X_difference_phi}.
 
Now we are ready to estimate the right-hand side of 
\eqref{eq:Y_growth}.
By using \eqref{eq:X_phi_terminal}, 
\begin{align}
 \label{eq:Y_growth_terminal}
\begin{split}
&e^{-\frac{\widetilde{\alpha}}{2}t} \sE\Big[ e^{\widetilde{\a} T}
|\p_x g(X^{t,x,\phi}_T)-\p_x g(X^{t,x,\phi'}_T)|^2
\Big]^{\frac{1}{2}}
\le 
L_g { 
e^{\frac{\widetilde{\alpha}}{2}(T-t)} }
\mathbb{E} \left[|\Delta X^{t,x}_T|^2 \right]^{\frac{1}{2}}
\\
& \leq C C_{\bar{b}}  L_g \mathfrak{m}^{1/2}_{(\alpha, \beta)} M_{([\phi']_1)}
\big(1+ {T}
+ T C_{\bar{b}}\|\phi'(0)\|_\infty
\big)
(1+|x|)
 |\phi-\phi'|_0,
\end{split}
\end{align}
{
where we recall  $\mathfrak{m}_{(\alpha, \beta)} = \sup_{t \in [0,T]} e^{ \widetilde{\alpha}(T-t)}      \int_{t}^{T} e^{(T-s)\b} \, \mathrm{d}s$}.
On the other hand, by 
Lemma \ref{lemma:H_gradient_estimate},
for all 
$t\in [0,T]$, $x,x'\in \sR^n$,
$\phi, \phi' \in \cV_\bA$,
$(y,z)\in \sR^n\t \sR^{n\t d}$,
\begin{align*}
&|\partial_x H_t(x, \phi_t(x),y,z)
-\partial_x H_t(x', \phi'_t(x'),y,z)|
\\
&\le 
\big(L_{\hat{b}}|x-x'|+
L_{\bar{b}}(|x-x'|+|\phi_t(x)-\phi'_t(x')|)\big)| y|
+ L_{\sigma}|x-x'||z|
+L_{fx}(|x-x'|+|\phi_t(x)-\phi'_t(x')|)
\\
&\le 
\Big(
\big(L_{\hat{b}}+
L_{\bar{b}}(1+[\phi]_1)\big)  | y|
+ L_{\sigma}|z|+L_{fx}(1+[\phi]_1)
\Big)|x-x'| + (L_{fx} + L_{\bar{b}} |y|) |\phi_t(x')-\phi'_t(x')|.
\end{align*}
This  
along with \eqref{eq:Y_uniform_bdd}
and the assumption that 
$|Z_s^{t,x,\phi}|\le C^{\phi}_Z$
 implies that 
\begin{align*}
 \begin{split}
&
\sE\bigg[ \bigg(\int_t^Te^{\frac{\widetilde{\a}}{2}s}
|\bar{f}^1_s(\cdot, Y^{t,x,\phi'}_s,Z^{t,x,\phi'}_s)-\bar{f}^2_s(\cdot, Y^{t,x,\phi'}_s,Z^{t,x,\phi'}_s)|\,\d s\bigg)^2\bigg]^{\frac{1}{2}}
\\
&\le
\sE\bigg[ \bigg(\int_t^Te^{\frac{\widetilde{\a}}{2}s}
\Big(
(C+
L_{\bar{b}}[\phi]_1) |Y^{t,x,\phi'}_s|
+L_{fx}(1+[\phi]_1) 
+ C |Z^{t,x,\phi'}_s|
\Big)|\Delta X^{t,x}_s|
\,\d s\bigg)^2\bigg]^{\frac{1}{2}}
\\
&\q 
+\sE\bigg[ \bigg(\int_t^Te^{\frac{\widetilde{\a}}{2}s}
\Big(
(L_{fx} + L_{\bar{b}} |Y^{t,x,\phi'}_s|) |\phi_s(X^{t,x,\phi'}_s)-\phi'_s(X^{t,x,\phi'}_s)|
\Big)\, \d s
\bigg)^2\bigg]^{\frac{1}{2}}
\\
&\le
C \bigg(\int_t^Te^{\frac{\widetilde{\a}}{2}s}\, \d s\bigg)
\bigg[\Big(
(1+
[\phi]_1) C_Y
+L_{fx}(1+[\phi]_1)
+  C_Z^{\phi'} 
\Big)
\|\Delta X^{t,x}\|_{\cS^2 }
\\
&\q 
+(L_{fx} +  C_Y) 
 \|\phi(X^{t,x,\phi'})-\phi'(X^{t,x,\phi'})\|_{\cS^2}\bigg].
\end{split}
\end{align*}
 Substituting 
  \eqref{eq:X_difference_phi} 
into the above estimate yields 
\begin{align*}
 \begin{split}
&
\sE\bigg[ \bigg(\int_t^Te^{\frac{\widetilde{\a}}{2}s}
|\bar{f}^1_s(\cdot, Y^{t,x,\phi'}_s,Z^{t,x,\phi'}_s)-\bar{f}^2_s(\cdot, Y^{t,x,\phi'}_s,Z^{t,x,\phi'}_s)|\,\d s\bigg)^2\bigg]^{\frac{1}{2}}
\\
&\le
C C_{\bar{b}} \bigg(\int_t^Te^{\frac{\widetilde{\a}}{2}s}\, \d s\bigg)
\bigg[\Big(
(C_Y
+L_{fx})(1+[\phi]_1)
+  C_Z^{\phi'} 
\Big)
T
 M_{([\phi]_1)}
 M_{([\phi']_1)}
\big(1+ {T}
+T C_{\bar{b}} \|\phi'(0)\|_\infty
\big)
\\
&\q 
+(L_{fx} +  C_Y) 
     M_{([\phi']_1)}
(1+ {T}
+ T C_{\bar{b}}\|\phi'(0)\|_\infty)
 \bigg](1+|x|)|\phi-\phi'|_0,
\end{split}
\end{align*}
 Combining with above estimate with 
 \eqref{eq:Y_growth} and \eqref{eq:Y_growth_terminal},
and using 
$e^{-\f{\widetilde{\a}}{2} t}
\int_t^T e^{\frac{\widetilde{\a}}{2} s}
\,\d s 
\le \f{2}{\widetilde{\a}}
(e^{\frac{\widetilde{\a}}{2}T} -1)
$ with $\widetilde{\alpha} = 2(\kappa_{\hat{b}} -{
\rho} +L_{\bar{b}}+L_{\sigma}^2)$,
we conclude the desired estimate 
\begin{align*}
\begin{split}
|Y^{t,x,\phi}_t-Y^{t,x,\phi'}_t| 
&\le 
C C_{\bar{b}}
(1+|x|)|\phi-\phi'|_0
\big(1+ {T}
+ T C_{\bar{b}} \|\phi'(0)\|_\infty
\big)     M_{([\phi']_1)}
\bigg( L_g { 
\mathfrak{m}^{1/2}_{(\alpha, \beta)}}
\\
&\q + \f{e^{\a T} -1}{{\a}}
\bigg[\Big(
( C_Y
+L_{fx})(1+[\phi]_1)
+ C_Z^{\phi'}
\Big)
T
 M_{([\phi]_1)}
+(L_{fx} +  C_Y) 
 \bigg]\bigg)
  \end{split}
\end{align*}
with $ {\alpha} = \kappa_{\hat{b}} -{
\rho}+L_{\bar{b}}+L_{\sigma}^2$, 
$\beta= 
2
(\kappa_{\hat{b}} 
  +L_{\bar{b}}[\phi]_1) +C$,
  and $M_{([\phi]_1)}$ defined in \eqref{eq:MT}.
\end{proof}
 
The next lemma 
establishes an upper bound of the adjoint process $Z^{t,x,\phi}$
in terms of the Lipschitz constant of $x\mapsto Y^{t,x,\phi}$.
The proof  is given in Appendix \ref{appendix:technical}
and extends 
the arguments of  \cite[Proposition 3.7]{lionnet2015time}
to the present setting where
\eqref{sde_feedback}
has non-Lipschitz drift coefficients and 
multiplicative noises,
and 
\eqref{bsde_feedback} has 
unbounded coefficients in front of $Y$.

\begin{Lemma}\label{lem:MD}
Suppose (H.\ref{assum:pgm}) holds.
For each 
  $\phi\in \cV_\bA$ and 
$(t,x)\in [0,T]\t \sR^n$,
let 
$(Y^{t,x,\phi},Z^{t,x,\phi})\in \cS^2(t,T;\sR^n)\t \cH^2(t,T;\sR^{n\t d})$ 
be defined by 
\eqref{bsde_feedback}.
Then 
for all
$\phi\in \cV_\bA$ and
$(t,x)\in [0,T]\t \sR^n$,
\begin{align}
    |Z^{t,x,\phi}|\le C_Z^{\phi}
    \coloneqq C_\sigma L_Y([\phi]_1),
    \q 
\textnormal{for $\d t \otimes \d \sP$-a.e.,}
\end{align}
where the constant $L_Y([\phi]_1) \ge 0$
is defined by \eqref{eq:L_Y_phi}.

\end{Lemma}

Armed with
 Theorem \ref{thm:phi_bound}, Theorem \ref{thm:lipschitz_iterates} and Proposition \ref{prop:Y_stability2}, 
we   prove that under suitable assumptions,  for any initial guess $\phi^0\in \cV_\bA$,
the sequence of feedback controls $(\phi^m)_{m \in \sN_0}$
generated by 
\eqref{eq:phi_update}
is a contraction with respect to the norm $|\cdot|_0$.

\begin{Theorem}\l{thm:iterate_contration}

Suppose (H.\ref{assum:pgm}) holds.
For each
 $\phi^0\in \cV_\bA$,
 $\tau >0$
 and 
$m\in \mathbb{N}$,
let $\phi^m$ be defined by \eqref{eq:phi_update}
with the initial guess $\phi^0$ and stepsize $\tau$.
Let $C\ge 0$ be a  constant such that 
\eqref{eq:Y_uniform_bdd}, 
\eqref{eq:Y_Lipschitz_estimate_statement}
and 
\eqref{eq:Y_difference_phi_statement}
hold,
let 
 $C_Y \ge 0 $ be  defined  in \eqref{eq:Y_uniform_bdd},
 and let
   $\alpha \in \sR$
 be defined in
\eqref{eq:L_Y_phi}.
For each 
$\phi^0\in \cV_\bA$ and $M\ge 0$,
let  $C_{(\phi^0)}\ge 0$
be defined in
\eqref{eq:phi_bound_constant},
let $L_{(\phi^0)}\ge 0$
be defined in 
\eqref{eq:lipschitz_iterates_constant}, 
and let 
$L_Y(M)\ge 0$ be defined in \eqref{eq:L_Y_phi}.
Then 
for all 
 $\phi^0 \in \cV_\bA$,
 if we assume further that  \eqref{eq:condition_lipschitz} 
 holds and 
\begin{align}
\label{eq:constraction_condition}
  C 
    (1+ {T}
+ T C_{\bar{b}} C_{(\phi^0)}
)     e^{  T
\beta_{+}} 
(T
e^{ T
 \beta_{+}}+1) 
B_{(\phi^0)}
    < { \frac{1}{2}
\left(\frac{\mu L_{fa}}{\mu + L_{fa}}
+\nu
\right)},
\end{align}
with the constants
$\beta\in \sR$,
$\mathfrak{m}_{(\alpha, \beta)}>0$
and 
$B_{(\phi^0)}\ge 0$  defined by  
\begin{align}\label{const:B}
\begin{split}
 \beta &\coloneqq 2 \kappa_{\hat{b}} 
  +2L_{\bar{b}}L_{(\phi^0)}+C, \quad { 
  \mathfrak{m}_{(\alpha, \beta)}  \coloneqq \sup_{t \in [0,T]} e^{2\alpha (T-t)}      \int_{t}^{T} e^{(T-s)\b} \, \mathrm{d}s}, \\
{B}_{(\phi^0)} & 
\coloneqq   
 C^2_{\bar{b}} 
 \Big[
 L_g {
 \mathfrak{m}^{1/2}_{(\alpha, \beta)}} 
 + \tfrac{e^{T \a} -1}{{\a}}
\Big(
( C_Y
+L_{fx})(1+L_{(\phi^0)})
+  C_\sigma L_Y\big(L_{(\phi^0)}\big)
\Big)
 \Big],
 \end{split}
\end{align}
then for  all
 { $\tau \in (0, \frac{2}{\mu + L_{fa}}\wedge \frac{1}{\nu}]$},
 there exists a constant $c \in [0,1)$ such that
\begin{align}\label{contraction}
 |\phi^{m+1}-\phi^{m}|_0 \leq c |\phi^m-\phi^{m-1}|_0,   
 \q \fa m\in \sN.
\end{align} 
\end{Theorem}
\begin{Remark}\label{remark:items}
Theorem \ref{thm:iterate_contration}
shows that 
if 
\eqref{eq:condition_lipschitz} 
and \eqref{eq:constraction_condition}
hold, then 
the iterates $(\phi^m)_{m\in \sN_0}$
form a Cauchy sequence,
whose limit will be characterized in Theorem \ref{thm:convergence_stationary}.
We now observe that 
the inequalities
\eqref{eq:condition_lipschitz} 
and \eqref{eq:constraction_condition}
can be ensured if one of the conditions 
\ref{item:T_small}-\ref{item:small_k} 
holds. 
To this end, we  focus on 
\eqref{eq:constraction_condition},
as \eqref{eq:condition_lipschitz} 
can be analyzed similarly. 
Suppose  all remaining parameters are  fixed.
Then one can clearly see that 
\eqref{eq:constraction_condition}
holds if
(a)
${ \frac{\mu L_{fa}}{\mu + L_{fa}}
+\nu}$ is sufficiently large or 
(b) ${B}_{(\phi^0)}$
is sufficiently small.
The former case holds 
if either $\mu$ or $\nu$ is sufficiently large
(note that 
\eqref{eq:f_a_lipschitz}
and \eqref{eq:strong_convex_f}
imply that  $\mu\le L_{fa}$).
{
The latter case holds for (a) small $C_{\bar{b}}$, 
or (b) small 
${\tfrac{e^{ T\alpha }-1}{ \alpha}}$
and 
$\mathfrak{m}_{(\alpha, \beta)}$,
or (c) small
$L_g, C_Y,L_{fx}$ and $L_Y(L_{(\phi^0)})$.
By the definitions of $\alpha$ and $\beta$,
${\tfrac{e^{ T\alpha }-1}{ \alpha}}$ and $\mathfrak{m}_{(\alpha, \beta)}$
tend to 0,
as $T\to 0$ or $\kappa_{\hat{b}}\to -\infty$ or $\rho \to \infty$ 
(see Lemma \ref{lemma:m_ab} in Appendix \ref{appendix:technical}),
while 
by 
\eqref{eq:Y_uniform_bdd}
and 
\eqref{eq:Y_Lipschitz_estimate_statement},
 $C_Y$ and $L_Y(L_{(\phi^0)})$ scale linearly in $C_g, L_g, C_{fx},L_{fx}$,
 and hence  ${B}_{(\phi^0)}$ is close to zero if 
$C_g, L_g, C_{fx},L_{fx}$ are sufficiently small.}

\end{Remark}

\begin{proof}
For any $(t,x) \in [0,T] \times \mathbb{R}^n$, Lemma \ref{Lemma_iteration_1} with $x=x'$, $y=Y_t^{t,x,\phi^m}$, $y' =Y_t^{t,x,\phi^{m-1}}$, $a=\phi_t^m(x) $ and $a' = \phi_t^{m-1}(x)$ immediately yields that
for all
 { $\tau \in (0, \frac{2}{\mu + L_{fa}}\wedge \frac{1}{\nu}]$},
\begin{align}
\label{eq:phi_m+1-phi_m_contraction}
 |\phi_t^{m+1}(x)-\phi^{m}_t(x)|    \le 
\left(
1 - { \tau\frac{1}{2}
\left(\frac{\mu L_{fa}}{\mu + L_{fa}}
+\nu
\right)}
\right)
|\phi_t^{m}(x)-\phi^{m-1}_t(x)|
+\tau C_{\bar{b}}|Y_t^{t,x,\phi^m}-Y_t^{t,x,\phi^{m-1}}|.
\end{align}
Applying Proposition \ref{prop:Y_stability2} further gives
\begin{align*}
 \frac{
 |\phi_t^{m+1}(x)-\phi^{m}_t(x)|}
 {1+|x|}
 &\le 
\left(
1 - { \tau\frac{1}{2}
\left(\frac{\mu L_{fa}}{\mu + L_{fa}}
+\nu
\right)}
\right)
\frac{|\phi_t^{m}(x)-\phi^{m-1}_t(x)|}{1+|x|}
\\
&\q 
+\tau C_{\bar{b}}B[\phi^m,\phi^{m-1},C_Z^{\phi^{m-1}}]
|\phi^m-\phi^{m-1}|_0
\end{align*}
Hence,  taking supremum over $(t,x)\in [0,T]\t \sR^n$ results in 
\begin{align*}
     |\phi^{m+1}-\phi^{m}|_0
    \leq \left( 1 + \tau \left( C_{\bar{b}} B[\phi^m,\phi^{m-1},C_Z^{\phi^{m-1}}] -{ \frac{1}{2}
\left(\frac{\mu L_{fa}}{\mu + L_{fa}}
+\nu
\right)}  \right) \right) |\phi^m-\phi^{m-1}|_0,
\end{align*}
with 
the constant 
$B[\phi^m,\phi^{m-1},C_Z^{\phi^{m-1}}]$
 defined in \eqref{eq:Y_difference_phi_constant}.
Observe that 
under  \eqref{eq:condition_lipschitz}, 
Theorem \ref{thm:lipschitz_iterates} shows that 
for all { $\tau \in (0, \frac{2}{\mu + L_{fa}}\wedge \frac{1}{\nu}]$}
  and  $m\in \sN_0$,
$[\phi^m]_1\le L_{(\phi^0)}$,
which along with 
Proposition \ref{prop:Y_stability2} implies that 
$C_Z^{\phi^{m}}\le C_{\sigma} L_Y(L_{(\phi^0)})$.
Hence, by Theorem \ref{thm:phi_bound} and  \eqref{eq:Y_difference_phi_constant}, 
\begin{align*}
C_{\bar{b}} B[\phi^m,\phi^{m-1},C_Z^{\phi^{m-1}}]
& \le 
C C_{\bar{b}}^2 \Big(1+ {T}
+ T  C_{\bar{b}} \sup_{t\in [0,T]}|\phi^{m-1}_t(0)|
\Big)     e^{ T
\beta_{+}}
\bigg( L_g {
\mathfrak{m}^{1/2}_{(\alpha, \beta)}}  
\\
&\q + \tfrac{e^{T \a } -1}{{\a}}
\Big(
( C_Y
+L_{fx})(1+[\phi^{m}]_1)
+  C_Z^{\phi^{m-1}} 
\Big)
(T
e^{ T
 \beta_{+}}+1) 
 \bigg)
 \\
 &\le 
  C 
    (1+ {T}
+ T C_{\bar{b}} C_{(\phi^0)}
)     e^{T
\beta_{+}} 
(T
e^{ T
 \beta_{+}}+1) 
B_{(\phi^0)},
 \end{align*}
  with $\beta$, $\mathfrak{m}_{(\alpha, \beta)}$
  and $B_{(\phi^0)}$
  defined as in \eqref{const:B}.
 Then 
 under \eqref{eq:constraction_condition},
 the desired estimate holds with
 \begin{equation*}
 c=1 + \tau \left( C 
    (1+ {T}
+ T C_{\bar{b}} C_{(\phi^0)}
)     e^{ T
\beta_{+}} 
(T
e^{ T
 \beta_{+}}+1) 
B_{(\phi^0)} -{ \frac{1}{2}
\left(\frac{\mu L_{fa}}{\mu + L_{fa}}
+\nu
\right)}\right)
\in [0,1).
\end{equation*}
Note that  \eqref{eq:constraction_condition}
implies $c<1$, and
$\tau \le \frac{2}{\mu + L_{fa}}\wedge \frac{1}{\nu}$
implies that 
$c\ge 1 -  { \frac{ \tau}{2}
\left(\frac{\mu L_{fa}}{\mu + L_{fa}}
+\nu
\right)} \ge 0$.
 \end{proof}

\subsection{Linear convergence   to stationary points}\label{sec:conv_stationary}
Based on Theorem \ref{thm:iterate_contration}, we 
prove the linear convergence of the iterates 
$(\phi^m)_{m\in \sN_0}$ in the weighted sup-norm $|\cdot|_0$ (see Definition \ref{def:fb})
and the  associated control processes $(\a^{\phi^m})_{m\in \sN_0}$ to stationary points of $J(\cdot;\xi_0)$.

The following proposition  characterizes stationary points of the summation of a nonconvex differentiable function and a convex nonsmooth function.
\begin{Proposition}\label{prop:stationary}
Let $X$ be a Hilbert space
 equipped with the norm $\|\cdot\|_X$,  
 $F:X\to \sR$ be a Fr\'{e}chet differentiable function, 
 $G:X\to \sR\cup \{\infty\}$
 be a proper, lower semicontinuous, convex function, and $x^\star\in \operatorname{dom} G$. Then $x^\star$ is a stationary point of $F+G$
 if and only if  
 {
 for some  $\tau>0$,}
 $$
 x^\star=\prox_{\tau G}(x^\star-\tau \nabla F(x^\star)),
 $$
 where for all $x\in X$,
 $\prox_{\tau G}(x)={
 \argmin}_{z\in X}\big(
 \tfrac{1}{2}\|z-x\|^2_X+\tau G(z)\big)$.
\end{Proposition}
 \begin{proof}
 By
 \cite[Proposition 1.107]{mordukhovich2006variational}, the Fr\'{e}chet differentiability of $F$ implies that
 $\p (F+G)(x^\star)=\nabla F(x^\star)+\p G(x^\star)$.
 Hence $x^\star$ is a stationary point of $F+G$
 if and only if $-\nabla F(x^\star)\in \p G(x^\star)$.
 By the properties of $G$,
 $\p G$ agrees with  the  convex subdifferential of $G$ (see \cite[Equation 1.51 and Theorem 1.93]{mordukhovich2006variational}), which along with
 the definition of $\prox$ shows  that for all $x,u\in X$ and $\tau>0$,
 $u=\prox_{\tau G}(x)$ if and only if $x-u\in \p (\tau G)(u)$.
 Hence by $-\nabla F(x^\star)\in \p G(x^\star)$, for all $\tau>0$,
 $\big(x^\star-\tau \nabla F(x^\star)\big)-x^\star\in \p (\tau G)(x^\star)$,
 which leads to the desired result. \end{proof}

The following theorem presents a precise statement of  Theorem \ref{thm:convergence_stationary_formal}, which establishes the linear convergence 
of the iterates 
$(\phi^m)_{m\in \sN_0}$,
and 
  characterizes the  limit
of   the associated control processes
$(\a^{\phi^{m}})_{m\in \sN_0}$
based on Proposition \ref{prop:stationary}.

\begin{Theorem}
\label{thm:convergence_stationary}
Assume the same notation as in Theorem \ref{thm:iterate_contration}.
For each $\phi\in \cV_\bA$,
let $\a^\phi\in \cH^2(\sR^k)$ be the associated control process.
Then for all 
 $\phi^0 \in \cV_\bA$ 
 satisfying
 \eqref{eq:condition_lipschitz} 
 and 
\eqref{eq:constraction_condition},
and for all   { $\tau \in (0, \frac{2}{\mu + L_{fa}}\wedge \frac{1}{\nu}]$},
there exists  $c \in [0,1)$, $\widetilde{C}\ge 0$ and $\phi^\star\in \cV_\bA$ such that 
\begin{enumerate}[(1)]
    \item \l{item:q_linear_phi}
    for all $m\in \sN_0$,
    $  |\phi^{m+1}-\phi^{\star}|_0\le  
c|\phi^{m}-\phi^\star|_0$,
    \item \l{item:r_linear_alpha}
     for all $m\in \sN_0$,
    $\|\a^{\phi^{m}}-\a^{\phi^\star}\|_{\cH^2}
\le \widetilde{C}
c^m |\phi^{0}-\phi^\star|_0$,

    \item
    \l{item:stationary_alpha_star}
$\alpha^{\phi^\star}$ 
is a stationary point of
$J(\cdot;\xi_0):\cH^2(\sR^k)\to \sR\cup\{\infty\}$
defined as in \eqref{eq:control_value}.
\end{enumerate}

\end{Theorem}

\begin{proof}
Throughout the proof,
let 
 $\phi^0 \in \cV_\bA$ satisfy \eqref{eq:condition_lipschitz} 
 and 
\eqref{eq:constraction_condition},
and   { $\tau \in (0, \frac{2}{\mu + L_{fa}}\wedge \frac{1}{\nu}]$}.
In the present  setting,
Theorem \ref{thm:lipschitz_iterates} implies  that
$\sup_{m\in \sN_0}[\phi^m]_1\le L_{(\phi^0)}$,
and 
 Theorem \ref{thm:iterate_contration}
 shows that 
$(\phi^m)_{m\in \sN_0}$
is a Cauchy sequence 
in  $(\cB([0,T]\t \sR^n;\sR^k),|\cdot|_0)$. 
As 
$(\cB([0,T]\t \sR^n;\sR^k),|\cdot|_0)$ is a Banach space,
 the Banach fixed point theorem shows that 
there exists  $\phi^\star\in \cB([0,T]\t \sR^n;\sR^k)$
such that 
$\lim_{m\to \infty} |\phi^m-\phi^\star|_0=0$.
The  convergence of $(\phi^m)_{m\in \sN_0}$
in the $|\cdot|_0$-norm
and $\sup_{m\in \sN_0}[\phi^m]_1\le L_{(\phi^0)}$
 imply that 
$[\phi^\star]_1\le L_{(\phi^0)}$.
Hence, to show $\phi^\star\in \cV_\bA$, it remains to prove 
$\phi^\star$ takes values in $\bA$ a.e.

By
Proposition \ref{prop:Y_stability2}
and 
$\sup_{m\in \sN_0,t\in [0,T]}
\big(
|\phi_t^m(0)|+
[\phi^m]_1\big)
<\infty$, there exists  $C\ge 0$ such that for all $(t,x)\in [0,T]\t \sR^n$ and $m,m'\in \sN_0$,
$|Y^{t,x,\phi^m}_t-
Y^{t,x,\phi^{m'}}_t| 
\le C(1+|x|) |\phi^m-\phi^{m'}|_0$.
This  along with the fact that $(\phi^{m})_{m\in \sN}$ is a Cauchy sequence in
 $(\cB([0,T]\t \sR^n;\sR^k),|\cdot|_0)$
shows that 
for all $(t,x)\in [0,T]\t \sR^n$,
$(Y^{t,x,\phi^m}_t)_{m\in \sN_0}$
is a Cauchy sequence in $\sR^n$.
Hence there exists a function 
$\cY:[0,T]\t \sR^n\to \sR^n$
such that 
$\lim_{m\to \infty} Y^{t,x,\phi^m}_t=\cY_t(x)$
for all $(t,x)\in [0,T]\t \sR^n$.
Then for any $(t,x)\in [0,T]\t \sR^n$,
by the continuity of
$\prox_{\tau \ell} $
and  ${
\partial_{a}} H_t^{\textrm{re}}$
and the pointwise convergence of 
$(\phi^m)_{m\in \sN_0}$
and $(Y^{t,x,\phi^m}_t)_{m\in \sN_0}$,
one can pass $m$ to infinity in 
\eqref{eq:phi_update}
and show
 for a.e.~$(t,x)\in [0,T]\t \sR^n$, 
\begin{align}\label{limit2_cY}
 \phi^\star_t(x) = \lim_{m \to \infty}  \phi_t^{m+1}(x)&= \lim_{m \to \infty} \prox_{\tau\ell} \big(\phi_t^m(x)-\tau {
   \partial_{a}} H_t^{\textrm{re}}(x,\phi_t^m(x),Y^{t,x,\phi^m}_t) \big) \notag \\
   & = \prox_{\tau\ell} \big(\phi^\star_t(x)-\tau {
   \partial_{a}} H_t^{\textrm{re}}(x,\phi^\star_t(x),\cY_t(x)) \big).
\end{align}
As $\prox_{\tau\ell}(z)\in \operatorname{dom} \ell=\bA$ for all $z\in \sR^k$, $\phi^\star_t(x)\in \bA$
for a.e.~$(t,x)\in [0,T]\t \sR^n$,
and hence  $\phi^\star\in \cV_\bA$.
Furthermore, 
by 
$\phi^\star\in \cV_{\bA}$, 
$\lim_{m\to \infty} |\phi^m-\phi^\star|_0=0$
and Proposition \ref{prop:Y_stability2},
$\lim_{m\to \infty} Y^{t,x,\phi^m}_t
=Y^{t,x,\phi^\star}_t
=\cY_t(x)$
for all $(t,x)\in [0,T]\t \sR^n$,
which along with 
\eqref{limit2_cY}  shows that
\begin{align}\label{limit2}
 \phi^\star_t(x) &= \prox_{\tau\ell} \big(\phi^\star_t(x)-\tau {
 \partial_{a}} H_t^{\textrm{re}}(x,\phi^\star_t(x),Y^{t,x,\phi^\star}_t) \big) 
\end{align}
We are now ready to establish the desired statements. 
To prove Item \ref{item:q_linear_phi},
for any $(t,x) \in [0,T] \times \mathbb{R}^n$, Lemma \ref{Lemma_iteration_1} with $x'=x$, $y=Y_t^{t,x,\phi^m}$, $y' =Y_t^{t,x,\phi^{\star}}$, $a=\phi_t^m(x) $ and $a' = \phi_t^{\star}(x)$ and \eqref{limit2} immediately yield that
for all
 { $\tau \in (0, \frac{2}{\mu + L_{fa}}\wedge \frac{1}{\nu}]$},
\begin{align*}
 |\phi_t^{m+1}(x)-\phi^{\star}(x)|    \le 
\left(
1 - { \tau\frac{1}{2}
\left(\frac{\mu L_{fa}}{\mu + L_{fa}}
+\nu
\right)}
\right)
|\phi_t^{m}(x)-\phi^{\star}_t(x)|
+\tau C_{\bar{b}}|Y_t^{t,x,\phi^m}-Y_t^{t,x,\phi^{\star}}|.
\end{align*}
Now, following the exact same lines as the proof of Theorem \ref{thm:iterate_contration} 
(cf.~\eqref{eq:phi_m+1-phi_m_contraction})
and using the above  facts that $\phi^\star \in \cV_\bA$, $\sup_{t\in [0,T]}|\phi_t^\star(0)|
\le C_{(\phi^0)}$ and 
$[\phi^\star]_1\le L_{(\phi^0)}$, we deduce $|\phi^{m+1}-\phi^{\star}|_0
\le  c|\phi^{m}-\phi^\star|_0$ with the same constant $c \in [0,1)$ as in Theorem \ref{thm:iterate_contration}.

 To
 prove Item
 \ref{item:r_linear_alpha},
 observe that for each $m\in \sN_0$, $\a^{\phi^m}=\phi^m(X^{\xi_0,\phi^m})$
and $\a^{\phi^\star}=\phi^\star(X^{\xi_0,\phi^\star})$, which implies that
\begin{align}
\label{eq:open_loop_conv2}
    \begin{split}
    \|\a^{\phi^{m+1}}-\a^{\phi^\star}\|_{\cH^2}
    &\le 
    \|\phi^{m+1}(X^{\xi_0,\phi^{m+1}})-\phi^{m+1}(X^{\xi_0,\phi^\star})\|_{\cH^2}
    +
    \|\phi^{m+1}(X^{\xi_0,\phi^\star})-\phi^\star(X^{\xi_0,\phi^\star})\|_{\cH^2}
    \\
    &\le 
   [\phi^{m+1}]_1 \|X^{\xi_0,\phi^{m+1}}-X^{\xi_0,\phi^\star}\|_{\cH^2}
    +
    |\phi^{m+1}- \phi^\star|_0(1+\|X^{\xi_0,\phi^\star}\|_{\cH^2}).
    \end{split}
\end{align}
By using $\phi^\star\in \cV_\bA$ and  $\sup_{m\in \sN}(|\phi^m_t(0)|
+
[\phi^m]_1)<\infty$
and  Lemma \ref{forward:apriori_p},
one can easily show that 
there exists $C\ge 0$ such that for all $m\in \sN_0$,
$\|X^{\xi_0,\phi^\star}\|_{\cH^2}
\le C$ and 
$\|X^{\xi_0,\phi^m}-X^{\xi_0,\phi^\star}\|_{\cH^2}\le C |\phi^m-\phi^\star|_0 $,
which along with 
\eqref{eq:open_loop_conv2} 
leads to the desired estimate 
 $\|\a^{\phi^{m+1}}-\a^{\phi^\star}\|_{\cH^2}\le \widetilde{C} c^m|\phi^0-\phi^\star|_0$
 for all $m\in \sN_0$,
 with  some constant 
$\widetilde{C}\ge 0$ independent of $m$.

{
It remains to prove 
Item \ref{item:stationary_alpha_star}. 
{
Let 
 $\tilde{H}^{\textrm{re}}:[0,T] \t \sR^n \t \sR^k \t \sR^n \to \sR$ 
 and  $\tilde{H}:[0,T] \t \sR^n \t \sR^k \t \sR^n\t \sR^{n\t d} \to \sR$
  be such that
  for all $(t,x,a,y,z)\in [0,T]\t \sR^n\t \sR^k\t \sR^n\t \sR^{n\t d}$,
$ 
\tilde{H}_t^{\textrm{re}}(x,a,y)\coloneqq \left \langle b_t(x,a),y \right \rangle + e^{-\rho t} f_t(x,a)
$ and 
$
 \tilde{H}_t(x,a,y,z)\coloneqq \tilde{H}_t^{\textrm{re}}(x,a,y) + \left \langle \sigma_t(x), z \right \rangle
 $.
For each $(t,x) \in [0,T]\t \sR^n$, let 
  $X^{t,x,\phi^\star}\in \cS^2(t,T;\sR^n)$ satisfy \eqref{sde_feedback} with $
  \phi=\phi^\star$, and 
  $(\tilde{Y}^{t,x,\phi^\star},\tilde{Z}^{t,x,\phi^\star})\in \cS^2(t,T;\sR^n)\t \cH^2(t,T;\sR^{n\t d})$ satisfy 
   $ \tilde{Y}^{t,x,\phi^\star}_T =  e^{-\rho T}  \partial_x g(X^{t,x,\phi^\star}_T)$,
   and 
  $$
  \mathrm{d}\tilde{Y}^{t,x,\phi^\star}_s
 = - \partial_x \tilde{H}_s(X^{t,x,\phi^\star}_s, \phi^\star_s(X^{t,x,\phi^\star}_s),\tilde{Y}^{t,x,\phi^\star}_s,\tilde{Z}^{t,x,\phi^\star}_s) \, \mathrm{d}s + \tilde{Z}^{t,x,\phi^\star}_s \, \mathrm{d}W_s,
 \quad \fa s\in [t,T).
 $$
The affineness of  $H$   and $\tilde{H}$ in $y$ and $z$ implies that 
$\tilde{Y}^{t,x,\phi^\star}_s= e^{-\rho s} {Y}^{t,x,\phi^\star}_s$ for all $(t,x)\in [0,T]\t \sR^n$ and $s\in [t,T]$. 
Moreover, by (H.\ref{assum:pgm}) and \eqref{eq:proximal}, for all $a,u\in \sR^k$ and $\eta>0$, 
\begin{align*}
a=\prox_{\ell} (a- \eta u) 
&\Longleftrightarrow
0\in (a-(a- \eta u))+\p \ell(a) 
\Longleftrightarrow
0\in u +\p (\eta^{-1} \ell)(a) 
\\
&\Longleftrightarrow
0\in (a-(a-u)) +\p (\eta^{-1} \ell)(a) 
\Longleftrightarrow
a=\prox_{\eta^{-1}\ell} (a-  u). 
\end{align*}
Hence by \eqref{limit2} and the affineness of $H^{\textrm{re}}$ and $\tilde{H}^{\textrm{re}}$ in $y$, for all $(t,x)\in [0,T]\t \sR^n$, 
\begin{align}\label{limit_scale}
\begin{split}
 \phi^\star_t(x) &= \prox_{\tau\ell} \big(\phi^\star_t(x)-\tau e^{\rho t} {
  \partial_{a}} \tilde{H}_t^{\textrm{re}}(x,\phi^\star_t(x),e^{-\rho t} Y^{t,x,\phi^\star}_t) \big) 
 \\
 & =  \prox_{\tau e^{-\rho t} \ell} \big(\phi^\star_t(x)-\tau {
   \partial_{a}} \tilde{H}^{\textrm{re}}_t(x,\phi^\star_t(x),\tilde{Y}^{t,x,\phi^\star}_t) \big).
\end{split}
\end{align}
Now consider the
solution 
$(X^{\xi_0,\phi^\star},
\tilde{Y}^{\xi_0, \phi^\star},
\tilde{Z}^{\xi_0, \phi^\star})
\in 
\cS^2(\sR^n)
\t \cS^2(\sR^n)
\t \cH^2(\sR^{n\t d})
$
to the 
following FBSDE:
for all $t\in [0,T]$,
\begin{align*}
\d X^{\xi_0, \phi^\star}_t
&= b_t(X^{\xi_0, \phi^\star}_t,\phi^\star_t(X_t^{\xi_0, \phi^\star}))\, \d t +\sigma_t(X_t^{\xi_0, \phi^\star})\, \d W_t, 
&&\q
 X_0^{\xi_0, \phi^\star} = \xi_0,
\\
\mathrm{d}\tilde{Y}^{\xi_0, \phi^\star}_t
&= - \partial_x \tilde{H}_t(X^{\xi_0, \phi^\star}_t, \phi^\star_t(X^{\xi_0, \phi^\star}_t),\tilde{Y}^{\xi_0, \phi^\star}_t,\tilde{Z}^{\xi_0, \phi^\star}_t) \, \mathrm{d}t + \tilde{Z}^{\xi_0, \phi^\star}_t \, \mathrm{d}W_t,
&&
\q \tilde{Y}^{\xi_0, \phi^\star}_T = e^{-\rho T} \partial_x g(X^{\xi_0, \phi^\star}_T).
\end{align*}
The Markov property
in \cite[Theorem 5.1.3]
{zhang2017backward} implies 
$\tilde{Y}^{\xi_0, \phi^\star}_t=\tilde{Y}^{t,X^{\xi_0,  \phi^\star}_t,\phi^\star}_t$ 
$\d t\otimes \d \sP$ a.e.,
which along with  \eqref{limit_scale} gives that 
for $\d t\otimes \d \sP$ a.e.,
\begin{align}
\label{eq:open_loop_stationary2}
\begin{split}
 \phi^\star_t(X^{\xi_0,  \phi^\star}_t)
 &= \prox_{\tau e^{-\rho t} \ell} \big(\phi^\star_t(X^{\xi_0, \phi^\star}_t)-\tau {
 \partial_{a}} \tilde{H}_t^{\textrm{re}}(X^{\xi_0, \phi^\star}_t,\phi^\star_t(X^{\xi_0, \phi^\star}_t),
 \tilde{Y}^{\xi_0, \phi^\star}_t
 ) \big).
\end{split}
\end{align}
Observe that for all $\a\in \cH^2(\sR^k)$,
$J(\a;\xi_0)=F(\a)+G(\a)$,
where 
\begin{align*}
    F(\alpha) \coloneqq \sE\bigg[
\int_0^T e^{- \rho t} f_t(X^{\xi_0,\a}_t,\a_t)  \, \d t+e^{- \rho T} g(X^{\xi_0,\a}_T)
\bigg], \quad G(\alpha) \coloneqq \sE\bigg[
\int_0^T e^{-\rho t} \ell(\a_t)  \, \d t
\bigg],
\end{align*}
where $X^{\xi_0,\a}$ satisfies \eqref{eq:control_fwd}. 
Then the regularity of coefficients and  
\cite[Corollary 4.11]{carmona2016lectures}
imply that 
$F$ is Fr\'{e}chet differentiable,
and the derivative 
$\nabla F$ at $\alpha^{\phi^\star}_\cdot=\phi^\star_\cdot(X^{\xi_0, \phi^\star}_\cdot)$ is given by 
$$
\nabla F(\alpha^{\phi^\star} )_t
={
\partial_{a}} \tilde{H}_t^{\textrm{re}}(X^{\xi_0, \phi^\star}_t,\phi^\star_t(X^{\xi_0, \phi^\star}_t),
 \tilde{Y}^{\xi_0, \phi^\star}_t
 ),
 \q \textnormal{$\d t\otimes \d \sP$ a.e.}
$$
Moreover, 
by (H.\ref{assum:pgm}\ref{item:l}),
one can easily prove  that 
 $ G$ is 
 proper, lower semicontinuous and convex,
 and satisfies 
for all $\a\in \cH^2(\sR^k)$ and $\tau>0$,
$\prox_{\tau G}(\a)=\prox_{\tau e^{- \rho t}\ell} (\a) $ for $\d t\otimes \d \sP$ a.e.
Hence, Proposition \ref{prop:stationary} shows that $\alpha^{\phi^\star} \in \cH^2(\sR^k)$ is a stationary point of $J(\cdot;\xi_0)$.
}
}
\qedhere

\end{proof}

The following theorem presents a precise statement 
of Theorem \ref{thm:conv_approximate_formal},
where the feedback controls $(\phi^m)_{m\in \sN_0}$ are updated with approximate  gradients.

\begin{Theorem}\l{thm:conv_approximate}

Suppose (H.\ref{assum:pgm}) holds.
Let $C\ge 0$ be a  constant such that 
\eqref{eq:Y_uniform_bdd}, 
\eqref{eq:Y_Lipschitz_estimate_statement}
and 
\eqref{eq:Y_difference_phi_statement}
hold,
let 
 $C_Y \ge 0 $ be  defined  in \eqref{eq:Y_uniform_bdd},
 and let
   $\alpha \in \sR$
 be defined in
\eqref{eq:L_Y_phi}.
Let 
$\phi^0\in \cV_\bA$,
let  $C_{(\phi^0)}\ge 0$
be defined in
\eqref{eq:phi_bound_constant},
let $L_{(\phi^0)}\ge 0$
be defined in 
\eqref{eq:lipschitz_iterates_constant}, 
and let 
$L_Y(M)$, $M \geq 0$, be defined in \eqref{eq:L_Y_phi}.
Suppose that  \eqref{eq:condition_lipschitz} is satisfied, { $\tau \in (0, \frac{2}{\mu + L_{fa}}\wedge \frac{1}{\nu}]$},
and there exist constants $\widetilde{C}, \widetilde{L} \ge 0$
such that {
for all $\omega \in \Omega$}
$\sup_{t\in [0,T]}|\widetilde{\phi}_t^m(0,\om)|\le \widetilde{C}$
and $[\widetilde{\phi}^m_\cdot(\cdot,,\om)]_1\le \widetilde{L}$
for all 
$m\in \sN_0$,
and 
\begin{align}
\label{eq:constraction_condition2}
\mathfrak{D} 
\coloneqq
  { \frac{1}{2}
\left(\frac{\mu L_{fa}}{\mu + L_{fa}}
+\nu
\right)}- C 
    (1+ {T}
+ T C_{\bar{b}} \widetilde{C}
)     e^{  T
\beta_{+}} 
(T
e^{ T
 \beta_{+}}+1) 
\widetilde{B}
>0,
\end{align}
with the constants
$\beta\in \sR$,
$\mathfrak{m}_{(\alpha, \beta)}>0$
and 
$\widetilde{B}\ge 0$  defined by  
\begin{align}\label{const:B2}
\begin{split}
 \beta &\coloneqq 2 \kappa_{\hat{b}} 
  +2L_{\bar{b}} \max \{ L_{(\phi^0)}, \widetilde{L} \}+C,
  \quad 
  {
  \mathfrak{m}_{(\alpha, \beta)} \coloneqq \sup_{t \in [0,T]} e^{2\alpha (T-t)}  \int_{t}^{T} e^{(T-s)\b} \, \mathrm{d}s   } ,
  \\
\widetilde{B} &
\coloneqq   
 C^2_{\bar{b}} 
 \Big[
 L_g  {
 \mathfrak{m}^{1/2}_{(\alpha, \beta)}  }  
 + \tfrac{e^{T \a} -1}{{\a}}
\Big(
( C_Y
+L_{fx})(1+L_{(\phi^0)})
+  C_\sigma L_Y\big(\widetilde{L} \big)
\Big)
 \Big].
 \end{split}
\end{align}
Let 
$c=1-\tau \mathfrak{D}\in [0,1)$.
Then  
{
for a.s.~$\om\in \Omega$ and 
  for all 
$m\in\sN_0$,
\begin{align}
\label{eq:pointwise_stability}
    |\phi^{\star} - \widetilde{\phi}^{m}_\cdot(\cdot,\omega)|_0 \leq c^{m}|\phi^0-\phi^\star|_0 + \sum_{j=0}^{m-1} c^{m-1-j} \tau C_{\bar{b}}\sup_{(t,x)\in [0,T]\t \sR^n} \frac{|{\mathcal{Y}}_t^{\widetilde{\phi}^j}(x,\omega)  -\widetilde{\mathcal{Y}}_t^{\widetilde{\phi}^j}(x,\om)|}{1+|x|},
\end{align}
where 
$\phi^\star\in \cV_\bA$ 
is the limit function  in Theorem \ref{thm:convergence_stationary}.
Consequently, for all $p\ge 1$ 
 and for all 
$m\in\sN_0$,
\begin{align}
\label{eq:Lp_stability}
    \sE[|\phi^{\star} - \widetilde{\phi}^{m}|^p_0]^{\frac{1}{p}} \leq c^{m}|\phi^0-\phi^\star|_0 + \sum_{j=0}^{m-1} c^{m-1-j} \tau C_{\bar{b}}
     \sE[|\cY^{\widetilde{\phi}^j}- \widetilde{\mathcal{Y}}^{\widetilde{\phi}^j}|^p_0]^{\frac{1}{p}},
\end{align}

 }
\end{Theorem}

\begin{proof}
{
Throughout this proof, 
we fix  $\omega \in \Omega$ 
and omit the explicit dependence on $\om$ if no confusion occurs.} 
First, observe that the conditions $\sup_{t\in [0,T]}|\widetilde{\phi}_t^m(0)|\le \widetilde{C}$
and $[\widetilde{\phi}^m]_1\le \widetilde{L}$
for all $m\in \sN_0$ guarantee $\widetilde{\phi}^m \in \cV_\bA$.
We continue by quantifying $|\phi^{m+1}_t(x) - \widetilde{\phi}^{m+1}_t(x)|$ for any $(t,x) \in [0,T] \times \mathbb{R}^n$, where $\widetilde{\phi}^{m+1}_t(x)$ is defined by \eqref{eq:phi_update22}. By Lemma \ref{Lemma_iteration_1} with $x=x'$, $a=\phi^{m}_t(x)$, $a'=\widetilde{\phi}^{m}_t(x)$, $y=\mathcal{Y}_t^{\phi^m}(x)$ 
and $y'=\widetilde{\mathcal{Y}}_t^{\widetilde{\phi}^m}(x)$, 
\begin{align}\label{iteration}
 |\phi^{m+1}_t(x) - \widetilde{\phi}^{m+1}_t(x)|   \le 
\left(1 - { \tau\frac{1}{2}
\left(\frac{\mu L_{fa}}{\mu + L_{fa}}
+\nu
\right)}
\right) |\phi^{m}_t(x) - \widetilde{\phi}^{m}_t(x)| + \tau C_{\bar{b}}|\mathcal{Y}_t^{\phi^m}(x) -\widetilde{\mathcal{Y}}_t^{\widetilde{\phi}^m}(x) |.
\end{align}
Now, by
applying Proposition \ref{prop:Y_stability2} and 
using the definition of $\widetilde{B}$ given in \eqref{const:B2}, 
\begin{align*}
   & C_{\bar{b}}|\mathcal{Y}_t^{\phi^m}(x) -\widetilde{\mathcal{Y}}_t^{\widetilde{\phi}^m}(x) | \leq C_{\bar{b}} |\mathcal{Y}_t^{\phi^m}(x) -\mathcal{Y}_t^{\widetilde{\phi}^m}(x) | +C_{\bar{b}}|
   \mathcal{Y}_t^{\widetilde{\phi}^m}(x)
-\widetilde{\mathcal{Y}}_t^{\widetilde{\phi}^m}(x) | \\
    &\quad  \leq  C 
    (1+ {T}
+ T C_{\bar{b}} \widetilde{C}
)    
 e^{ T
\beta_{+}} 
(T
e^{ T
 \beta_{+}}+1) 
 \widetilde{B}
(1+|x|) |\phi^m  - \widetilde{\phi}^m|_0 
+ C_{\bar{b}} |\mathcal{Y}_t^{\widetilde{\phi}^m}(x)-\widetilde{\mathcal{Y}}_t^{\widetilde{\phi}^m}(x) |.
\end{align*}
Now let $c=1-\tau \mathfrak{D}$ with $ \mathfrak{D}>0$   defined   in 
\eqref{eq:constraction_condition2}. The condition \eqref{eq:constraction_condition2}
implies $c<1$, and
$\tau \le \frac{2}{\mu + L_{fa}}\wedge \frac{1}{\nu}$
implies that 
$c\ge 1 -  { \frac{ \tau}{2}
\left(\frac{\mu L_{fa}}{\mu + L_{fa}}
+\nu
\right)} \ge 0$.
Hence, 
{
for a.s.~$\om\in \Omega$,
\begin{align}
\label{eq:one_step_pointwise_stability}
   |\phi^{m+1}  - \widetilde{\phi}^{m+1}_\cdot(\cdot,\omega) |_0 \leq c |\phi^m  - \widetilde{\phi}^m_\cdot(\cdot,\omega) | _0+ \tau C_{\bar{b}} |
   \mathcal{Y}^{\widetilde{\phi}^m}_\cdot(\cdot,\omega)
    -\widetilde{\mathcal{Y}}^{\widetilde{\phi}^m}_\cdot(\cdot,\omega)  |_0,
\end{align}
which along with $\phi^0=\widetilde{\phi}^0$  implies that 
$
    |\phi^{m+1} - \widetilde{\phi}^{m+1}|_0 \leq 
    \sum_{j=0}^m c^{m-j} \tau C_{\bar{b}} | \mathcal{Y}^{\widetilde{\phi}^m}_\cdot(\cdot,\omega) -\widetilde{\mathcal{Y}}^{\widetilde{\phi}^m}_\cdot(\cdot,\omega) |_0.
$
The estimate  \eqref{eq:pointwise_stability} then follows from
Theorem
\ref{thm:convergence_stationary}
Item \ref{item:q_linear_phi},
and the estimate \eqref{eq:Lp_stability} follows by  taking the $L^p$-norm on both sides of  \eqref{eq:pointwise_stability}. 
}
 \qedhere

\end{proof}
\appendix

\section{Proofs of technical results}
\label{appendix:technical}

This section is devoted to the proofs of 
  Proposition \ref{prop:well_posed_phi}
  and Lemmas
\ref{forward:apriori_p}, 
\ref{lemma:H_gradient_estimate},
\ref{lem:MD}.

To prove Proposition \ref{prop:well_posed_phi},
we first establish a general well-posedness
result for BSDEs with non-Lipschitz and unbounded coefficients.
Although the result 
does not  follow directly from 
 \cite[Proposition 3.5]{briand2000bsdes}
due to the unboundedness of $A_t$
  (i.e., the condition (A1(3)) in \cite{briand2000bsdes} fails),
the  techniques there can be extended to the present setting.

\begin{Lemma}\label{lemma:adjoint_wellpose}
Let $T>0$,
$\kappa\in \sR$, $L\ge 0$, 
$\xi\in L^\infty(\cF_T;\sR^n)$,
let 
$A\in \cH^2(\sR^{n\t n})$
be such that 
 $y^\top A_t (\om) y\le \kappa |y|^2$
 for  all $(t,\om,y)\in [0,T]\t \Om\t \sR^n$,
 and let 
 $f:[0,T]\t \Om\t  \sR^{n\t d}\to \sR^n$ 
 satisfy
for all  $(t,\om)\in [0,T]\t \Om$ and  $z,z'\in \sR^{n\t d}$,
$(f_t(\cdot,z))_{t\in [0,T]}$ is progressively measurable,  
 $|f_t(\om,z)-f_t(\om,z')|\le L|z-z'|$
 and 
 $\sup_{(t,\om)\in [0,T]\t \Om} |f_t(\om,0)|<\infty$.
Then the following BSDE
\bb\label{eq:adjoint_unbounded}
\d Y_t
 = - (A_tY_t+f_t(\cdot, Z_t))  \, \mathrm{d}t+ Z_t \, \mathrm{d}W_t, \q t\in [0,T]; \quad Y_T = \xi,
\ee
admits a unique solution $(Y,Z)\in \mathcal{S}^2(\sR^n)
\t \mathcal{H}^2(\sR^{n\t d})$.
\end{Lemma}
\begin{proof}[Proof of Lemma \ref{lemma:adjoint_wellpose}]
Throughout this proof, let $h_t(\om,y,z)=A_t(\om)y+f_t(\om,z)$ for all $(t,\om,y,z)\in [0,T]\t \Om\t \sR^n\t   \sR^{n\t d}$. 
Then 
for all  $(t,\om)\in [0,T]\t \Om$, $y,y'\in  \sR^n$ and $z,z'\in \sR^{n\t d}$,
$\la y-y',h_t(\om,y,z)-h_t(\om,y',z)\ra \le \kappa |y-y'|^2$ and 
$|h_t(\om,y,z)-h_t(\om,y,z')|\le L|z-z'|$.
Hence the a-priori estimate in Lemma \ref{BSDE:apriori} shows that \eqref{eq:adjoint_unbounded} admits at most one solution in the space $\mathcal{S}^2(\sR^n)
\t \mathcal{H}^2(\sR^{n\t d})$. 

To establish the existence of solutions, we construct a sequence of Lipschitz functions $(h^m)_{m\in \sN}$ approximating $h$ as in \cite[Proposition 3.5]{briand2000bsdes}. 
Without  loss of generality, we assume that $\kappa=0$, which in general can be achieved with an exponential time scaling of the solution.
For each $m\in \sN$, let $h:[0,T]\t \Om\t \sR^n \t \sR^{n\t d}\to \sR^n$ be such that 
$$
h^m_t(\om,y,z) \coloneqq A^m_t(\om)y+f_t(\om,z),
\q \fa (t,\om,y,z)\in [0,T]\t \Om\t \sR^n  \t \sR^{n\t d},
$$
where 
$A^m_t(\om)\coloneqq 
\frac{\min(m,|A_t(\om)|)}{|A_t(\om)|} A_t(\om)$. 
It is clear that for each $m\in \sN$, $h^m$ is uniformly Lipschitz continuous in $y$ and $z$. 
Hence by \cite[Theorem 4.3.1]{zhang2017backward}, there exist unique processes $(Y^m, Z^m)\in \cH^2(\sR^n)\t \cH^2(\sR^{n\t d})$ satisfying 
$$
\d Y_t
 = - h^m_t(Y_t,Z_t)  \, \mathrm{d}t+ Z_t \, \mathrm{d}W_t, \q t\in [0,T]; \quad Y_T = \xi.
$$
Observe
from $\kappa=0$
that 
for all 
$(t,\om,y, z )\in [0,T]\t \Om\t 
\sR^n\t \sR^{n\t d}$,
\begin{align}
    \begin{split}
        \la y, h^m_t(\om,y,z)\ra
        =
        \la y, A^m_t(\om)y\ra
        +\la y, f_t(\om,z)\ra 
        \le 
        |y|(C_f+L|z|),
    \end{split}
\end{align}
where $C_f=\sup_{(t,\om)\in [0,T]\t \Om}|f_t(\om,0)|$.
Hence by \cite[Proposition 2.1]{briand2000bsdes},
for $\d t\otimes \d \sP$ a.e., 
$$
\sup_{t\in [0,T]}|Y^m|^2\le 
\|\xi\|_{L^\infty}^2e^{\beta T}+\tfrac{C_f^2}{\beta}(e^{\beta T}-1),
\q \textnormal{with $\beta=1+L^2$.}
$$
We now denote by $C$  a generic constant independent of $m$. 
By  Lemma \ref{BSDE:apriori}, for all $m,m'\in \sN$,
\begin{align*}
\|Y^{m'}-Y^{m}\|_{\cS^2}^2+
\|Z^{m'}-Z^{m}\|_{\cH^2}^2
&\le
C 
\sE\bigg[ \bigg(\int_0^T 
|h^{m'}_t(\cdot, Y^m_t,Z^m_t)-
h^{m}_t(\cdot, Y^m_t,Z^m_t)
|\,\d t\bigg)^{2}\bigg]
\\
&\le 
C 
\sE\bigg[ \bigg(\int_0^T 
|A^{m'}_t -A^{m}_t| | Y^m_t|
\,\d t\bigg)^{2}\bigg]
\le C 
\|
A^{m'} -A^{m}\|_{\cH^2}^{2},
\end{align*}
where the last inequality follows from the uniform bound of $(Y^m)_{m\in \sN}$.
Lebesgue's dominated convergence theorem shows that 
$\|
A^{m'} -A^{m}\|_{\cH^2}$ tends to zero as $m,m'\to \infty$ with $m'\ge m$.
This  implies that  
$(Y^m,Z^m)_{m\in \sN}$ is a Cauchy sequence in $\cS^2(\sR^n)\t \cH^2(\sR^{n\t d})$, 
and hence  converges to some processes 
$(Y,Z)$ in $\cS^2(\sR^n)\t \cH^2(\sR^{n\t d})$.
Moreover, by
Fatou's lemma, there exits $C\ge 0$ such that 
$|Y|\le C$ for $\d t\otimes \d \sP$ a.e.

It remains to verify that $(Y,Z)$ is a solution to \eqref{eq:adjoint_unbounded}.
As $(Z^m)_{m\in \sN}$ converges to $Z$ in $\cH^2(\sR^{n\t d})$, 
 for all $t\in [0,T]$, $\lim_{m\to \infty}\|\int_t^T Z^m_s\,\d W_s- \int_t^T Z_s\,\d W_s\|_{L^2}=0$. Moreover, 
 for all $t\in [0,T]$,
 \begin{align*}
&
\bigg\|\int_0^T 
\big(
h^{m}_t(\cdot, Y^m_t,Z^m_t)-
h_t(\cdot, Y_t,Z_t)
\big)
\,\d t\bigg\|_{L^1}
\\
&\le
\bigg\|\int_0^T 
(|A^m_t-A_t||Y^m_t|
+|A_t|Y^m_t-Y|
)\,\d t
\bigg\|_{L^1}
+ 
 \bigg\|\int_0^T 
|
f_t(\cdot, Z^m_t)-
f_t(\cdot, Z_t)
|
\,\d t\bigg\|_{L^1},
\end{align*}
which converge to zero as $m\to \infty$, 
due to the  Cauchy-Schwarz inequality,
Lebesgue's dominated convergence theorem,
the  boundedness 
of $(Y^m)_{m\in \sN}$
and the convergence of $(Y^m,Z^m)_{m\in \sN}$.
This proves the desired existence result.
\end{proof}

Armed with Lemma \ref{lemma:adjoint_wellpose},
we prove the well-posedness of the iterates $(\phi^m)_{m\in \sN_0}$
for general stepsizes $\tau>0$.

\begin{proof}[Proof of Proposition \ref{prop:well_posed_phi}]
For any given $\phi^0\in \cV_\bA$ and $\tau>0$,
we prove the statement with an induction argument. 
 Due to the assumption $\phi^0\in\cV_\bA$, 
the statement clearly holds for $m=0$.

Suppose that  $\phi^m\in \cV_\bA$ for some $m\in \sN_0$.
Then
by (H.\ref{assum:pgm}),
$(t,x)\mapsto (b_t(x,\phi^m_t(x)),\sigma_t(x))$ is locally Lipschitz  and monotone in $x$.
Hence, 
for each $(t,x)\in [0,T]\t \sR^n$,
by   \cite[
Theorem 3.6, p.~58]{mao2007stochastic},  \eqref{sde_feedback} admits a unique solution
$X^{t,x,\phi^m}\in  \mathcal{S}^2(t,T;\sR^n)$. 

We then apply Lemma \ref{lemma:adjoint_wellpose}
for 
 the well-posedness of 
$(Y^{t,x,\phi^m}, Z^{t,x,\phi^m})$. 
By Lemma \ref{lemma:H_gradient_estimate}
and the fact that $\phi^m$ takes values in $\bA$,
$|
\partial_x H_t(x, \phi_t(x),y,z)
-
\partial_x H_t(x, \phi_t(x),y,z')
|\le L_{\sigma}|z-z'|$,
and 
\begin{align*}
 \begin{split}
 &\la y-y',
 \partial_x H_t(x, \phi^m_t(x),y,z)
 -\partial_x H_t(x, \phi^m_t(x),y',z)
\ra  
 \le (\kappa_{\hat{b}} - {
 \rho} +L_{\bar{b}})|y-y'|^2.
 \end{split}
 \end{align*}
Moreover, by \eqref{eq:f_x_lipschitz} and \eqref{eq:g_bound}, 
$\partial_x g(X^{t,x,\phi^m}_T)$ and 
$\p_x f_s(X^{t,x,\phi^m}_s, \phi^m_s(X^{t,x,\phi^m}_s))$
are uniformly bounded $\d t\otimes \d \sP$-a.e.,
 and 
$$
\|\p_x b(X^{t,x,\phi^m},\phi^m(X^{t,x,\phi^m}))\|_{\cH^2}
\le C(1+\|X^{t,x,\phi^m}\|_{\cS^2})<\infty,
$$
where we  used the linear growth of $\p_x\hat{b}$ and $\phi^m$ and the boundedness of $\p_x \bar{b}$. 
Hence by Lemma \ref{lemma:adjoint_wellpose},
\eqref{bsde_feedback} admits a unique solution in the space $ \mathcal{S}^2(t,T;\sR^n)
\t \mathcal{H}^2(t,T;\sR^{n\t d})$.
The fact that $(t,x)\mapsto Y^{t,x,\phi^m}_t$ can be identified as a deterministic function follows from \cite[Remark 2.1]{pardoux1998backward}.
Finally, similar to Propositions 
\ref{prop:Y_bound}
and 
\ref{prop:Y_stability}, one can prove that
 $(t,x)\mapsto Y^{t,x,\phi^m}_t$ 
is bounded
and  $x\mapsto Y^{t,x,\phi^m}_t$ is Lipschitz continuous uniformly in $t$.
As $\tau \ell:\sR^k\to \sR\cup\{\infty\}$ is proper, lower semicontinuous and convex, 
 $z\mapsto \operatorname{prox}_{\tau \ell}(z)$
is Lipschitz continuous and takes values in $\bA$.
Hence, one can easily deduce from 
(H.\ref{assum:pgm}) that
 $\phi^{m+1}\in \cV_\bA$, which completes the induction argument. 
\end{proof}

\begin{proof}[Proof of Lemma \ref{forward:apriori_p}]
Throughout this proof, 
let $C$ be a generic constant depending only on $p$,
let $t\in [0,T]$, $x_1,x_2\in  \sR^n$
and write $X_s^{1}=X_s^{t,x_1,1}$,
$X_s^{2}=X_s^{t,x_2,2}$
and $\Delta X_s=X_s^{1}-X_s^{2}$. 
By applying It\^{o}'s formula
to 
$(|\Delta X_s|^p)_{s\in [t,T]}$, for $  t\le s\le T$, 
\begin{align}
\begin{split}
  |\Delta X_s|^p & \leq |x_1 - x_2|^p + \int_{t}^s  \big(
 p |\Delta X_r|^{p-2}   \langle \Delta X_r,  b^1_r(X^{1}_r)-{b}^2_r(X^{2}_r)  \rangle  
 \\
 &\q 
+\tfrac{p(p-1)}{2} |\Delta X_r|^{p-2}  |\sigma^1_r(X_r^{1})-\sigma^2_r({X}_r^{2})|^2\big)
 \, \d r 
 \\
 &\q 
  + p \int_t^s  |\Delta X_r|^{p-2}  \langle \Delta X_r,
 \sigma^1_r(X_r^{1})-\sigma^2_r(X^{2}_r)\ra \,\d W_r,
\label{eq:Delta_X2}
\end{split}
\end{align}
which along with  the assumptions of  $b^1$ and $\sigma^1$ gives 
\begin{align}
|\Delta X_s|^p 
& \le |x_1-x_2|^p+
  \int_{t}^s \big(
  (p\mu_1+
  p(p-1)\nu^2_1) |\Delta X_r|^p
    \notag \\
 & \quad + p |\Delta X_r|^{p-1} |  b_r^{1}(X^{2}_r)-b^{2}_r(X^{2}_r) | 
+ p(p-1)
|\Delta X_r|^{p-2}|\sigma^1_r(X_r^{2})-\sigma^2_r({X}_r^{2})|^2
\big)
   \, \d r  
 \nb \\
& \q + p \bigg|\int_t^s  |\Delta X_r|^{p-2} \langle \Delta X_r,
( \sigma^1_r(X_r^{1})-\sigma^2_r(X^{2}_r) )\,\d W_r\ra\bigg|.
\label{eq:Delta_X}
\end{align}
Observe that by the Burkholder-Davis-Gundy inequality  (see \cite[Theorem 2.4.1]{zhang2017backward})
and Young's inequality,  for all $\varepsilon > 0$,
\begin{align*}
& p \mathbb{E} \left[ \sup_{s \in [t,T]} \bigg|\int_t^s  |\Delta X_r|^{p-2} \langle \Delta X_r,
( \sigma^1_r(X_r^{1})-\sigma^2_r(X^{2}_r) )\,\d W_r\ra\bigg| \right]  \\
& \leq C p \mathbb{E} \left[ \bigg(\int_t^T  |\Delta X_r|^{2p-4} |\Delta X_r|^2 | \sigma^1_r(X_r^{1})-\sigma^2_r(X^{2}_r)|^2 \,\d r \bigg)^{1/2} \right] \\
& \leq C p \mathbb{E} \left[ \sup_{s \in [t,T]} |\Delta X_s|^{p/2} \bigg(\int_t^T  |\Delta X_r|^{p-2} | \sigma^1_r(X_r^{1})-\sigma^2_r(X^{2}_r)|^2 \,\d r \bigg)^{1/2} \right] \\
& \leq \varepsilon \|\Delta X\|^p_{\cS^p } +  C \nu_1^2\varepsilon^{-1} {p^2}
\mathbb{E}  \left[ \int_t^T  |\Delta X_r|^{p} \,\d r  \right] + C\varepsilon^{-1} 
{p^2} \int_t^T  |\Delta X_r|^{p-2} | \sigma^1_r(X_r^{2})-\sigma^2_r(X^{2}_r)|^2 \,\d r.
\end{align*}
Hence, 
by taking supremum over $s\in [t,T]$ and  expectations on both sides of \eqref{eq:Delta_X}, 
\begin{align*}
 &   \|\Delta X\|^p_{\cS^p } \\
 & \leq |x_1-x_2|^p+
    (p\mu_1+p(p-1)\nu^2_1)\mathbb{E} \left[ \int_{t}^T 
 |\Delta X_r|^p \,\d r \right] + \varepsilon \|\Delta X\|^p_{\cS^p } +  C \nu_1^2 \varepsilon^{-1} p^2\mathbb{E} \left[ \int_t^T  |\Delta X_r|^{p} \,\d r  \right] \\
 & \quad +  \mathbb{E} \left[  \int_t^T \big( p |\Delta X_r|^{p-1} |  b_r^{1}(X^{2}_r)-b^{2}_r(X^{2}_r) | 
+ (p(p-1) + C \varepsilon^{-1} p^2) |\Delta X_r|^{p-2}|\sigma^1_r(X_r^{2})-\sigma^2_r({X}_r^{2})|^2 \big) \, \mathrm{d}r \right].
\end{align*}
Then by Young's inequality, 
\allowdisplaybreaks
\begin{align*}
 & \mathbb{E} \left[  \int_t^T \big( p |\Delta X_r|^{p-1} |  b_r^{1}(X^{2}_r)-b^{2}_r(X^{2}_r) | 
+ (p(p-1) + C \varepsilon^{-1} p^2) |\Delta X_r|^{p-2}|\sigma^1_r(X_r^{2})-\sigma^2_r({X}_r^{2})|^2 \big) \, \mathrm{d}r \right] \\
& \leq \mathbb{E} \left[ \sup_{s \in [t,T]} |\Delta X_s|^{p-1} \int_t^T p  |b_r^{1}(X^{2}_r)-b^{2}_r(X^{2}_r) |  \, \mathrm{d}r \right]
 \\
& \quad +  \mathbb{E} \left[ \sup_{s \in [t,T]} |\Delta X_s|^{p-2} \int_t^T (p(p-1) + C \varepsilon^{-1} p^2)|\sigma^1_r(X_r^{2})-\sigma^2_r({X}_r^{2})|^2 \, \mathrm{d}r \right] \\
& \leq \varepsilon \|\Delta X\|^p_{\cS^p } + \varepsilon^{-1} T^{p/2} C_{(p)} \mathbb{E} \left[ \left(\int_t^T |b_r^{1}(X^{2}_r)-b^{2}_r(X^{2}_r) |^2 \, \mathrm{d}r  \right)^{p/2}  \right] \\
& \quad  + \varepsilon^{-1} C_{(p)} (p(p-1) + C \varepsilon^{-1} p^2)^{p/2}  \mathbb{E} \left[ \left(\int_t^T |\sigma_r^{1}(X^{2}_r)-\sigma^{2}_r(X^{2}_r) |^2  \, \mathrm{d}r \right)^{p/2}   \right],
\end{align*}
which gives us that
\begin{align*}
   \|\Delta X\|^p_{\cS^p }  &\leq |x_1-x_2|^p+
    (p\mu_1+p(p-1) \nu^2_1) \mathbb{E} \left[ \int_{t}^T 
 |\Delta X_r|^p \,\d r \right] + \varepsilon \|\Delta X\|^p_{\cS^p }\\
 & \quad  +  C\nu_1^2\varepsilon^{-1} p^2 \mathbb{E} \left[ \int_{t}^T 
 |\Delta X_r|^p \,\d r \right] + \varepsilon^{-1} T^{p/2} C_{(p)} \| b^1({X}^{t,x_2,2})-{b}^2({X}^{t,x_2,2}) \|^p_{\cH^p } \\
 & \quad +  \varepsilon^{-1}
 C_{(p)}(p(p-1) + C \varepsilon^{-1} p^2)^{p/2}  \| \sigma^1({X}^{t,x_2,2})-\sigma^2({X}^{t,x_2,2}) \|^p_{\cH^p }. 
\end{align*}
The desired estimate follows by choosing  $\eps=1/2$,
by applying Gr\"{o}nwall's inequality and by taking the $p$-th root on both sides.

We now prove the inequality 
\eqref{eq:X_b_difference} by assuming that
$\sigma^1 \equiv \sigma^2$
and $x^1=x^2=x$.
Let 
 $\beta \coloneqq 2\mu_1 + \nu_1^2 +1$.
 Applying It\^{o}'s formula to $(e^{-s \beta}|\Delta X_s|^2)_{s\in [t,T]}$ and taking expectations on both sides yield that
\begin{align}
\begin{split}
&\sE[e^{-T\beta}
 |\Delta X_T|^2] 
 \\
 & \leq 
 \sE\bigg[
 \int_{t}^T  \Big( e^{-r \beta}
 (  - \beta |\Delta X_r|^2 +
 2\langle \Delta X_r,  b^1_r(X^{1}_r)-{b}^2_r(X^{2}_r)  \rangle+ |\sigma^{1}(X_r^1) - \sigma^{1}(X_r^2) |^2 \Big)
 \d r  \bigg] 
 \\
 & \leq 
 \sE\bigg[
 \int_{t}^T  \big( e^{-r \beta}
 (  - \beta |\Delta X_r|^2+
 2\mu_1|\Delta X_r|^2
 +
 2| \Delta X_r| | b^1_r(X^{2}_r)-{b}^2_r(X^{2}_r)|  
   + \nu_1^2|\Delta X_r |^2 \big)
 \d r  \bigg] 
 \\
 & 
  \leq \mathbb{E} \left[  \int_{t}^{T}  e^{-r  \beta} |b_r^{1}(X^{2}_{r})-b^{2}_{r}(X^{2}_{r})|^2 \,\mathrm{d}r \right],
\end{split}
\end{align}
where the second inequality follows from the assumptions of $b^1$ and $\sigma^1$,
and the last inequality 
follows from the Cauchy--Schwarz inequality 
and the fact that $\b=2\mu_1+\nu_1^2+1$.
\end{proof}

\begin{proof}[Proof of Lemma \ref{lemma:H_gradient_estimate}]
By  
\eqref{eq:hat_b_lipschitz}
and \eqref{eq:bar_b_lipschitz},
\begin{align*}
\begin{split}
&\la y-y',
\partial_x H_t(x, a,y,z)
-\partial_x H_t(x, a,y',z)\ra 
=
\la y-y',
\partial_x b_t(x, a)^\top (y-y') -{
\rho(y-y')}
\ra 
\\
&=
\la y-y',
\big(\partial_x 
\hat{b}_t(x)
+\partial_x \bar{b}_t(x)
a
\big)^\top (y-y')-{
\rho(y-y')}
\ra 
\le (\kappa_{\hat{b}} -{
\rho}+L_{\bar{b}})|y-y'|^2.
\end{split}
\end{align*}
Moreover, 
by \eqref{eq:hat_b_lipschitz},
\eqref{eq:bar_b_lipschitz}
and 
\eqref{eq:f_x_lipschitz},
\begin{align*}
&|\partial_x H_t(x, a,y,z)
-\partial_x H_t(x', a',y,z')|
\\
&\le 
|\partial_x b_t(x, a)-\partial_x b_t(x', a')|| y| + |\tr(\partial_x \sigma_t(x)^\top z) - \tr(\partial_x \sigma_t(x')^\top z')|
+|\partial_x f_t(x, a)-\partial_x f_t(x', a')|
\\
&\le 
\big(L_{\hat{b}}|x-x'|+
L_{\bar{b}}(|x-x'|+|a-a'|)\big)| y| + L_{\sigma}|x-x'| |z|+L_{\sigma}|z-z'|
+L_{fx}(|x-x'|+|a-a'|).
\end{align*}
This completes the proof of the desired estimates.
\end{proof}

\begin{proof}[Proof of Lemma \ref{lem:MD}]
Throughout this proof,
for each $p>1$ and Euclidean space $E$, we denote by 
$\sD^{1,p}(E )$ and $\sL_{1,p} (E)$  
the spaces of Malliavin differentiable random variables and processes;
see \cite[Appendix A.2]{lionnet2015time} for details.

By Corollary 3.5 and Remark 3.4 in \cite{imkeller2019differentiability},  for any $(t,x) \in [0,T] \times \mathbb{R}^n$,  $\phi \in \cV_\bA$,  $s \in [t,T]$
and $p>1$,
$X^{t,x,\phi}_s \in \sD^{1,p}(\sR^n)$, and 
the derivative
$DX^{t,x,\phi}=(D X^{t,x,\phi,(1)},\ldots, D X^{t,x,\phi,(d)})$,
which is $\sR^{n\t d}$-valued, satisfies
for $t\le r< s\le T$ that $D_s X^{t,x,\phi}_r=0$ and 
for  $t \le s\le  r \le T$ that
\begin{align}\l{eq:X_M_derivative}
D_s X^{t,x,\phi}_r&=
 \sigma_s (X^{t,x,\phi}_s)+
\int_s^r \p b_u D_s X^{t,x,\phi}_u \, \d u +
\sum_{k=1}^d
\int_s^r \partial_x \sigma^{(k)}_uD_s X^{t,x,\phi}_u\, \d W^{(k)}_{u}
\end{align}
for some adapted process $\partial b:[t,T]\t \Om\to \sR^{n \t n}$.

We proceed to establishing the Malliavin differentiability of $(Y^{t,x,\phi},Z^{t,x,\phi})$.
By \cite[Theorem 5.1.3]{zhang2017backward},
 for all $s\in [t,T]$,
$Y^{t,x,\phi}_s = u_s(X^{t,x,\phi}_s)$, with the function $(t,x)\mapsto u_t(x) \coloneqq Y_t^{t,x,\phi}$.
Thus  the Lipschitz continuity of $x \mapsto Y_t^{t,x,\phi}$, the chain rule and   $X^{t,x,\phi}\in  \mathbb{L}_{1,p}(\sR^n)$ imply that $Y^{t,x,\phi}\in \mathbb{L}_{1,p}(\sR^n)$ for all $p>1$.

For the Malliavin differentiability of $Z^{t,x,\phi}$, 
for each $(t,x)\in [0,T]\t\sR^n$ and $\phi\in \cV_\bA$,
consider  the following BSDE: 
\begin{equation}\label{aux:BSDE2}
    U^{t,x,\phi}_s = \partial_x g(X^{t,x,\phi}_T) + \int_{s}^{T} \hat{f}_r(V^{t,x,\phi}_r) \, \mathrm{d}r - \int_{s}^{T} V_r^{t,x,\phi} \, \mathrm{d}W_r,
\end{equation}
where $\hat{f}: [t,T] \times \Omega \times \mathbb{R}^{n \times d} \to \mathbb{R}^{n}$ is defined  by 
\begin{equation*}
    \hat{f}_r(v) \coloneqq \partial_x H_r(X^{t,x,\phi}_r(\omega), \phi_r(X^{t,x,\phi}_r(\omega)),Y^{t,x,\phi}_r(\omega),v),
    \q 
    \fa (r,\om, v)\in [t,T] \times \Omega \times \mathbb{R}^{n \times d}.
\end{equation*}
The Lipschitz continuity of $\mathbb{R}^{n \times d} \ni v \mapsto \hat{f}_r(v)$ 
imply that
\eqref{aux:BSDE2} admits the unique solution  $( U^{t,x,\phi}, V^{t,x,\phi})$, 
which along with  the uniqueness of solutions to \eqref{bsde_feedback} shows that  $( U^{t,x,\phi}, V^{t,x,\phi} ) = ( Y^{t,x,\phi}, Z^{t,x,\phi})$.
Observe
from (H.\ref{assum:pgm}),
$\phi\in \cV_\bA$ and 
$X^{t,x,\phi},Y^{t,x,\phi}\in
\mathbb{L}_{1,p}$ for all $p>1$
that for all $v\in \mathbb{R}^{n \times d}$,
$$
 \partial_x H_{\cdot}(X^{t,x,\phi}_{\cdot}, \phi_{\cdot}(X^{t,x,\phi}_{\cdot}),Y^{t,x,\phi}_{\cdot},v) \in \mathbb{L}_{1,2}.
$$
Hence applying 
\cite[Proposition 5.3]{el1997backward} to \eqref{aux:BSDE2}
implies $ (U^{t,x,\phi}, V^{t,x,\phi})  =   (Y^{t,x,\phi}, Z^{t,x,\phi})$ are Malliavin differentiable. 
By further applying  
\cite[Proposition 5.3]{el1997backward}
to \eqref{bsde_feedback}, 
$Z_s^{t,x,\phi} = D_s Y_s^{t,x,\phi}$ for $\d t \otimes \d \sP$-a.e.

Recall the representation 
$Y^{t,x,\phi}_s = u_s(X^{t,x,\phi}_s)$ 
with  $ u_t(x) = Y_t^{t,x,\phi}$.
By the chain rule and 
\eqref{eq:X_M_derivative},
$Z_s^{t,x,\phi} = D_s Y_s^{t,x,\phi}  = \p u_s D_s X^{t,x,\phi}_s
=
 \p u_s  \sigma_s(X^{t,x,\phi}_s)
$ for a uniformly bounded process $\p u$.
In fact, $\p u$ is the weak limit  of the sequences 
$(\p_x  u^\eps(X^{t,x,\phi}))_{\eps>0}$
in $\cH^2(t,T;\sR^{n\t n})$,
where $(u^\eps)_{\eps>0}$ 
is a standard mollification of $u$.
Then by Proposition \ref{prop:Y_stability},
$|\p u|\le L_Y([\phi]_1)$,
which 
along with \eqref{eq:sigma}
leads to the desired result.
\end{proof}

\begin{Lemma}\label{lemma:m_ab}
Let $\mathfrak{m}_{(\alpha,\beta)}$ be defined as in \eqref{eq:Y_difference_phi_constant}, where $\alpha$ and $\beta$ are given by  \eqref{eq:L_Y_phi} and \eqref{eq:Y_difference_phi_constant},
 respectively. Then  $\mathfrak{m}_{(\alpha,\beta)} \to 0$, as $T \to 0$, or $\rho \to \infty$ or $\kappa_{\hat{b}} \to - \infty$.
\end{Lemma}
\begin{proof}
We start by fixing $T>0$ and $\kappa_{\hat{b}}$ and considering $\rho \to \infty$. Then
 $\beta$ is fixed and  $\alpha$ tends to $- \infty$.
By definition of $\mathfrak{m}_{(\alpha,\beta)}$, 
\begin{equation*}
    \mathfrak{m}_{(\alpha,\beta)} = \sup_{t \in [0,T]} e^{2\alpha(T-t)} \frac{e^{\beta(T-t)} -1}{\beta} = \sup_{t \in [0,T]} \frac{e^{(2\alpha + \beta)t} - e^{2\alpha t}}{\beta}.
\end{equation*}
Assume without loss of generality that $\alpha$ is sufficiently negative such that $2\alpha+\beta<0$.
A straightforward computation shows that the supremum is attained for $t^{\star}=  \frac{1}{\beta}\ln \left( \frac{2\alpha}{2\alpha + \beta} \right)$, and hence
\begin{equation*}
    \mathfrak{m}_{(\alpha,\beta)} =e^{2\alpha \frac{\ln \left( \frac{2\alpha}{2\alpha + \beta} \right)}{\beta}} \frac{e^{ {\ln \left( \frac{2\alpha}{2\alpha + \beta} \right)}} - 1}{\beta}
    =e^{2\alpha \frac{\ln \left( \frac{2\alpha}{2\alpha + \beta} \right)}{\beta}} \frac{ \frac{2\alpha}{2\alpha + \beta}  - 1}{\beta}
    =-e^{ \frac{2\alpha}{\beta}\ln \left( \frac{2\alpha}{2\alpha + \beta} \right)} { \frac{1}{2\alpha + \beta} }.
\end{equation*}
By  L'Hospital's rule,
$\lim_{\alpha\to -\infty} \frac{2\alpha }{\beta} \ln \left( \frac{2\alpha}{2\alpha + \beta} \right)= -1$,
which shows that 
$\lim_{\rho\to \infty}\mathfrak{m}_{(\alpha,\beta)} = 0$. 

Now we fix $\rho$ and send  $\kappa_{\hat{b}} \to -\infty$ (i.e., $\beta \to - \infty$) or $T\to 0$. Observe that 
\begin{align*}
     \mathfrak{m}_{(\alpha,\beta)} = \sup_{t \in [0,T]} e^{2\alpha(T-t)}  \int_{t}^{T} e^{(T-s) \beta} \, \mathrm{d}s \leq   \sup_{t \in [0,T]} e^{2\alpha(T-t)}  \int_{0}^{T} e^{(T-s) \beta} \, \mathrm{d}s \leq e^{2\alpha_{+} T} \frac{e^{T \beta} -1}{\beta}, 
\end{align*}
from which one can easily deduce that $\lim_{T\to 0}\mathfrak{m}_{(\alpha,\beta)}=\lim_{\kappa_{\hat{b}} \to -\infty }\mathfrak{m}_{(\alpha,\beta)}=0$. 
\end{proof}


\section*{Acknowledgements}
\noindent
Wolfgang Stockinger is supported by a special Upper Austrian Government grant.
\bibliographystyle{siam}
\bibliography{pgm.bib}

\end{document}